\documentclass[11pt,letterpaper]{article}
\usepackage{amsfonts, amsmath, amssymb, amscd, amsthm, color, graphicx, mathrsfs, wasysym, setspace, mdwlist, calc,float}
\usepackage{setspace}
\usepackage{hyperref}
\makeatletter
\renewcommand{\paragraph}{%
  \@startsection{paragraph}{4}%
  {\z@}{2ex \@plus .5ex \@minus .2ex}{-1em}%
  {\normalfont\normalsize\bfseries}%
}
\makeatother

\usepackage{titlesec}
\titleformat*{\section}{\LARGE\bfseries}
\titleformat*{\subsection}{\Large\bfseries}
\titleformat*{\subsubsection}{\large\bfseries}

\usepackage{tocloft}

\usepackage{hyperref}
\usepackage{xcolor}
\hypersetup{colorlinks,
    linkcolor={red!50!black},
    citecolor={blue!80!black},
    urlcolor={blue!80!black}}
\usepackage{float}

\newtheorem{thm}{Theorem}[section]
\newtheorem{cor}[thm]{Corollary}
\newtheorem{lem}[thm]{Lemma}
\newtheorem{prop}[thm]{Proposition}
\newtheorem{prob}[thm]{Problem}
\newtheorem{conj}[thm]{Conjecture}

\theoremstyle{definition}
\newtheorem{defn}[thm]{Definition}
\newtheorem{conv}[thm]{Convention}
\theoremstyle{remark}
\newtheorem{rem}[thm]{Remark}
\newtheorem{ex}[thm]{Example}

\renewcommand{\L }{{\mathcal L} }
\renewcommand{\d }{{\rm d} }
\newcommand{\dl}{\widehat{\d}}

\newcommand{\e }{\varepsilon }
\renewcommand{\kappa }{\varkappa}

\newcommand{\h}{\hookrightarrow _h }
\renewcommand{\H}{\mathcal{H}}
\newcommand{\AH}{\mathcal {AH}}
\newcommand{\AHG}{\mathcal{AH}(G)}
\newcommand{\HG}{\mathcal{H}(G)}
\newcommand{\AHH}{\mathcal{AH}(H)}
\newcommand{\GG}{\mathcal{G}(G)}

\newcommand{\AcG}{\mathcal{A}_{cb}(G)}
\newcommand{\LG}{\mathcal L(G)}

\newcommand{\Hi}{\{ H_1, \ldots, H_n\}}
\newcommand{\Hl }{\Hi}

\newcommand{\acts}{\curvearrowright}
\newcommand{\act}{\curvearrowright}
\newcommand{\ztwo} {\mathbb{Z} \times \mathbb{Z}}
\newcommand{\Ga}{\Gamma}

\newcommand{\PN}{\mathcal P (\omega)}

\newcommand{\op}{\operatorname}
\begin{document}

\title{Hyperbolic structures on groups}
\author{C. Abbott, S. Balasubramanya, D. Osin}
\date{}
\maketitle

\begin{abstract}
For every group $G$, we define the set of \emph{hyperbolic structures} on $G$, denoted $\H (G)$, which consists of equivalence classes of (possibly infinite) generating sets of $G$ such that the corresponding Cayley graph is hyperbolic; two generating sets of $G$ are \emph{equivalent} if the corresponding word metrics on $G$ are bi-Lipschitz equivalent. Alternatively, one can define hyperbolic structures in terms of cobounded $G$-actions on hyperbolic spaces.  We are especially interested in the subset $\AHG\subseteq \H (G)$ of \emph{acylindrically hyperbolic structures} on $G$, i.e., hyperbolic structures corresponding to acylindrical actions. Elements of $\H (G)$ can be ordered in a natural way according to the amount of information they provide about the group $G$. The main goal of this paper is to initiate the study of the posets $\H(G)$ and $\AHG$ for various groups $G$. We discuss basic properties of these posets such as cardinality and existence of extremal elements, obtain several results about hyperbolic structures induced from hyperbolically embedded subgroups of $G$, and study to what extent a hyperbolic structure is determined by the set of loxodromic elements and their translation lengths.
\end{abstract}

\tableofcontents


\section{Introduction}


It is customary in geometric group theory to study groups as metric spaces. The standard way to convert a group $G$ into a geometric object is to fix a generating set $X$ and endow $G$ with the corresponding word metric $\d_X$. However, not all generating sets are equally good for this purpose: the most informative metric space is obtained when $X$ is finite, while the space corresponding to $X=G$ forgets the group almost completely. We begin with an attempt to formalize this observation by ordering generating sets according to the amount of information about the group $G$ retained by $(G, \d_X)$.

\begin{defn}\label{def-GG}
Let $X$, $Y$ be two generating sets of a group $G$. We say that $X$ is \emph{dominated} by $Y$, written $X\preceq Y$, if the identity map on $G$ induces a Lipschitz map between metric spaces $(G, \d_Y)\to (G, \d_X)$. This is obviously equivalent to the requirement that $\sup_{y\in Y}|y|_X<\infty$, where $|\cdot|_X=\d_X(1, \cdot)$ denotes the word length with respect to $X$.  It is clear that $\preceq$ is a preorder on the set of generating sets of $G$ and therefore it induces an equivalence relation in the standard way:
$$
X\sim Y \;\; \Leftrightarrow \;\; X\preceq Y \; {\rm and}\; Y\preceq X.
$$
We denote by $[X]$ the equivalence class of a generating set $X$ and by $\GG$ the set of all equivalence classes of generating sets of $G$. The preorder $\preceq$ induces an order relation $\preccurlyeq $ on $\GG$ by the rule
$$
[X]\preccurlyeq [Y] \;\; \Leftrightarrow \;\; X\preceq Y.
$$
\end{defn}

For example, finite generating sets of a finitely generated group are all equivalent and the corresponding equivalence class is the largest element of $\GG$; for every group $G$, $[G]$ is the smallest element of $\GG$. Note also that our order on $\GG$ is ``inclusion reversing": if $X$ and $Y$ are generating sets of $G$ such that $X\subseteq Y$, then $Y\preceq X$.

We are now ready to introduce the main notion of this paper. Given a generating set $X$ of a group $G$, we denote by $\Gamma (G,X)$ the corresponding Cayley graph.

\begin{defn}
A \emph{hyperbolic structure} on $G$ is an equivalence class $[X]\in \GG$ such that $\Gamma (G,X)$ is hyperbolic.  We denote the set of hyperbolic structures by $\H (G)$ and endow it with the order induced from $\GG$.
\end{defn}

Since hyperbolicity of a space is a quasi-isometry invariant, the definition above is independent of the choice of a particular representative in the equivalence class $[X]$. Using the standard argument from the proof of the Svarc-Milnor Lemma, it is easy to show that elements of $\H (G)$ are in one-to-one correspondence with equivalence classes of cobounded actions of $G$ on hyperbolic spaces considered up to a natural equivalence: two actions $G\curvearrowright S$ and $G\curvearrowright T$ are equivalent if there is a coarsely $G$-equivariant quasi-isometry $S\to T$.

We are especially interested in the subset of \emph{acylindrically hyperbolic structures} on $G$, denoted $\AHG$, which consists of hyperbolic structures $[X]\in \H (G)$ such that the action of $G$ on the corresponding Cayley graph $\Gamma (G,X)$ is acylindrical. Recall that an isometric action of a group $G$ on a metric space $(S,\d)$ is \emph{acylindrical} \cite{Bow} if for every constant $\e$ there exist constants $R=R(\e)$ and $N=N(\e)$ such that for every $x,y\in S$ satisfying $\d(x,y)\ge R$, we have
$$
\# \{ g\in G\mid \d(x,gx)\le\e, \; \d(y, gy)\le \e\} \le N.
$$
Groups acting acylindrically on hyperbolic spaces have received a lot of attention in the recent years. For a brief survey we refer to \cite{Osi18}.

The goal of our paper is to initiate the study of the posets $\H (G)$ and $\AH (G)$ for various groups $G$ and suggest directions for the future research. Our main results are discussed in the next section. Some open problems are collected in \ref{sec:OP}.

\paragraph{Acknowledgements.} The authors would like to thank Spencer Dowdall for useful conversations about mapping class groups, and Henry Wilton, Dani Wise and Hadi Bigdely for helpful discussions of 3-manifold groups.  The authors also thank the anonymous referee for useful comments.  The first author was partially supported by the NSF RTG award DMS-1502553. The third author was supported by the NSF grants DMS-1308961 and DMS-1612473.


\section{Main results}


\paragraph{2.1. General classification and examples.}

We assume the reader to be familiar with the standard terminology and refer to Section 3.1 for definitions and details. Given a hyperbolic space $S$, we denote by $\partial S$ its Gromov boundary.

\begin{defn}
We say that a hyperbolic structure $[X]\in \H (G)$ is
\begin{enumerate}
\item[---] \emph{elliptic} if $|\partial \Gamma (G,X)|=0$ (equivalently, $\Gamma (G,X)$ is bounded);
\item[---] \emph{lineal} if $|\partial \Gamma (G,X)|=2$ (equivalently, $\Gamma (G,X)$ is quasi-isometric to a line);
\item[---] \emph{quasi-parabolic} if $|\partial \Gamma (G,X)|=\infty$ and $G$ fixes a point of $\partial \Gamma (G,X)$;
\item[---] \emph{of general type} if $|\partial \Gamma (G,X)|=\infty$ and $G$ does not fix any point of $\partial \Gamma (G,X)$.
\end{enumerate}
The sets of elliptic, lineal, quasi-parabolic, and general type hyperbolic structures on $G$ are denoted by $\H_e(G)$, $\H_{\ell} (G)$, $\H_{qp} (G)$, and $\H_{gt}(G)$ respectively.
\end{defn}

It is easy to verify that if $X,Y$ are equivalent generating sets of $G$ such that the corresponding Cayley graphs $\Gamma (G,X)$ and $\Gamma (G,Y)$ are hyperbolic, then the identity map $G\to G$ induces a $G$-equivariant homeomorphism $\partial \Gamma (G,X)\to \partial \Gamma (G,Y)$. Therefore the definition above is independent of the choice of a particular representative in the equivalence class $[X]$. Our first result is an immediate consequence of Gromov's classification of groups acting on hyperbolic spaces \cite[Section 8]{Gro} and the well-known fact that a cobounded action cannot be parabolic.

\begin{thm}\label{main0}
For every group $G$, we have
$$
\H(G)=\H_e(G)\sqcup \H_{\ell} (G)\sqcup \H_{qp} (G)\sqcup \H_{gt}(G).
$$
The subsets $\H_e(G)\sqcup \H_{\ell} (G)$ and $\H_e(G)\sqcup \H_{\ell} (G)\sqcup \H_{qp} (G)$ are initial segments of the poset $\H(G)$.
\end{thm}

Recall that a subset $B$ of a poset $A$ is called an \emph{initial segment} if for any $a\in A$ and any $b\in B$, $a\le b$ implies $a\in B$.

Elliptic structures are the easiest to classify: we always have $\H_e (G)=\{ [G]\}$. The only element of $\H_e (G)$ is called the \emph{trivial hyperbolic structure}; all other hyperbolic structures on $G$ are called \emph{non-trivial.}

Next, we discuss lineal hyperbolic structures. We say that a structure $[X]\in \H_\ell (G)$ is \emph{orientable} if $G$ fixes $\partial \Gamma (G,X)$ pointwise. For instance, the lineal structure on $\mathbb Z$ corresponding to a finite generating set is orientable, while the lineal structure on the infinite dihedral group corresponding to a finite generating set is not. We denote the set of orientable lineal hyperbolic structures on a group $G$ by $\H_\ell^+ (G)$. Our next theorem gives a complete description of possible isomorphism types of $\H _\ell(G)$ and $\H_\ell^+ (G)$.

\begin{thm}\label{lineal}
For every group $G$, the following holds.
\begin{enumerate}
\item[(a)] $\H_\ell (G)$ is an antichain (and hence so is $\H_\ell ^+(G)$).
\item[(b)] The cardinality of $\H _\ell^+(G)$ is $0$, $1$, or at least continuum. On the other hand, for every cardinal $\kappa$ there exists a group $G$ such that $|\H_\ell (G)|=\kappa$.
\end{enumerate}
\end{thm}

We note that even very ``small" groups can have huge sets of lineal hyperbolic structures; e.g., $\H_\ell^+ (\mathbb Z^2)=\H_\ell (\mathbb Z^2)$ is an antichain of cardinality continuum (see Example \ref{ex-Zn}).

The posets $\H _{qp} (G)$ and $\H _{gt} (G)$ can have a much more complicated structure, and a complete classification of them up to isomorphism seems to be out of reach at the moment. We mention two examples here.

\begin{ex}\label{ZwrZ}[Prop. \ref{ZwrZ}]
Let $\mathcal D$ denote the set of natural numbers ordered according to divisibility : $m\preccurlyeq n$ if $m \mid n$. It is not difficult to show that $\H_{qp}(\mathbb Z\, {\rm wr}\, \mathbb Z)$ contains an isomorphic copy of $\mathcal D$. In particular, every finite poset embeds in $\H_{qp}(\mathbb Z\, {\rm wr}\, \mathbb Z)$. And that is only the visible part of the iceberg: we also show that $\H_{qp}(\mathbb Z\, {\rm wr}\, \mathbb Z)$ contains an antichain of cardinality continuum.
\end{ex}

The next example is somewhat counterintuitive and should be compared to Theorem \ref{main1} discussed below.

\begin{thm}\label{n-gt-intr}
For every $n\in \mathbb N$, there exists a finitely generated group $G$ such that $\H_\ell(G)=\H_{qp}(G)=\emptyset$ and $\H_{gt}(G)$ is an antichain of cardinality $n$ (see Fig. \ref{hasse}).
\end{thm}

\begin{figure}
 \centering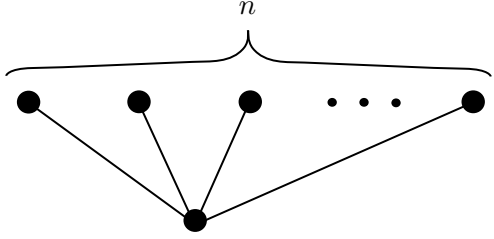\\
  \caption{The Hasse diagram of the poset $\H (G)$ for the group $G$ from Theorem \ref{n-gt-intr}.}\label{hasse}
\end{figure}

The reason we call Theorem \ref{n-gt-intr} counterintuitive is that for any hyperbolic structure $[X]\in \H_{gt} (G)$, we can find two loxodromic elements $g,h\in G$ with disjoint limit sets on $\partial \Gamma (G,X)$. By the standard ping-pong argument, sufficiently high powers of $g$ and $h$ generate a free subgroup $F\le G$ of rank $2$ whose orbits are quasi-convex in $\Gamma (G,X)$. It is well-known that collapsing collections of uniformly quasi-convex subsets of a hyperbolic space yields another hyperbolic space. Therefore passing from $[X]$ to $[X\cup H]$, where $H\le F$ is a non-trivial finitely generated subgroup of $F$, we obtain hyperbolic structures on $G$ which are strictly smaller than $[X]$ and it seems plausible that these structures are distinct for ``sufficiently distinct" subgroups $H\le F$.  Theorem \ref{n-gt-intr} applied to $n=1$ shows that this approach actually does not work: all hyperbolic structures on the group $G$ from the above theorem produced in this way will be trivial.

On the other hand, the idea described in the previous paragraph does work if the hyperbolic structure $[X]$ is acylindrical. For this reason, the poset of acylindrically hyperbolic structures exhibits a much more rigid behavior. In the next theorem we denote by $\PN$ the poset of all subsets of $\mathbb N$ ordered by inclusion.

\begin{thm}\label{main1}
For every group $G$, exactly one of the following conditions holds.
\begin{enumerate}
\item[(a)] ${\rm card}\, \AHG=1$, i.e., the only acylindrically hyperbolic structure on $G$ is trivial.
\item[(b)] ${\rm card}\, \AHG =2$. This is equivalent to $G$ being virtually infinite cyclic. In this case, the only non-trivial acylindrically hyperbolic structure on $G$ is lineal.
\item[(c)] ${\rm card}\,\AHG\ge 2^{\aleph _0}$. This is equivalent to $G$ being acylindrically hyperbolic. In this case, all non-trivial acylindrically hyperbolic structures on $G$ are of general type and $\AHG$ contains a copy of $\PN$.
\end{enumerate}
\end{thm}

Recall that a group is \emph{acylindrically hyperbolic} if it admits a non-elementary acylindrical action on a hyperbolic space \cite{Osi16}. The class of acylindrically hyperbolic groups  includes many examples of interest: all non-elementary hyperbolic and relatively hyperbolic groups, all but finitely many mapping class groups of punctured closed surfaces, $Out(F_n)$ for $n\ge 2$, non-virtually cyclic groups acting properly on proper $CAT(0)$ spaces and containing a rank one isometry, most $3$-manifold groups, groups of deficiency at least $2$ and many other examples. For details we refer to  \cite{DGO,MO,Osi15,Osi16} and references therein.

\paragraph{2.2. Induced hyperbolic structures.}
The proof of the second statement of part (c) of Theorem \ref{main1} (as well as the proof of a much stronger fact, Theorem \ref{main3} discussed below) makes use of hyperbolic structures on groups induced from hyperbolic structures on subgroups. We restrict to the case of a single subgroup here and refer to Section 5 for the general case.

Let $H$ be a subgroup of a group $G$ and let $X$ be a relative generating set of $G$ with respect to $H$.  That is, $G=\langle X\cup H\rangle $. It is easy to see that the map sending a generating set $Y$ of $H$ to $X\cup Y$ gives rise to a map $$\iota_X\colon \mathcal G(H)\to \GG,$$ which can be thought of as a particular case of the induced action map studied in \cite{AHO}. In general, very little can be said about $\iota _X$. However, it behaves well if $H$ is hyperbolically embedded in $G$ (see Section 5.1 for the definition).

\begin{thm}\label{main2}
Let $H$ be a hyperbolically embedded subgroup of a group $G$. Then there exists a relative generating set $X$ of $G$ with respect to $H$ such that $\iota _X$ defines injective, order preserving maps $\H (H)\to \H (G)$ and $\AHH\to \AHG$.
\end{thm}

This result is applied to prove that $\AHG$ is sufficiently complicated for every acylindrically hyperbolic group $G$ as follows: every acylindrically hyperbolic group $G$ contains a hyperbolically embedded subgroup isomorphic to $F_2\times K$, where $F_2$ is free of rank $2$ and $K$ is a finite group \cite[Theorem 2.24]{DGO}. By Theorem \ref{main2} it suffices to show that $\mathcal {AH}(F_2\times K)$ is sufficiently complicated, and the latter poset is much easier to understand. Yet another application is obtained by combining Theorem \ref{main2} with a ``relatively hyperbolic" version of the $SQ$-universality of $F_2$  (see \cite{Ols95} and \cite{AMO}).

\begin{cor}\label{SQ}
Let $G$ be an acylindrically hyperbolic group and let $H$ be a finitely generated group. Then $\H (H)$ embeds into $\H (G)$ as a poset. In particular, the posets of hyperbolic structures of any two finitely generated acylindrically hyperbolic groups embed in each other.
\end{cor}

\paragraph{2.3. $Out(G)$-action on $\AH (G)$.}
An additional motivation for studying the poset of hyperbolic structures of a group $G$ stems from the fact that $Out(G)$ admits a natural action on $\HG$. Indeed, for every automorphism $\alpha$ of a group $G$ and any generating set $X$ of $G$, $\alpha (X)$ also generates $G$. Obviously two generating sets $X$ and $Y$ of $G$ are equivalent if and only if $\alpha(X)$ and $\alpha(Y)$ are. This allows us to define an action of $Aut(G)$ on $\GG$ by the rule
$$
\alpha([X])=[\alpha(X)]
$$
for all $[X]\in \GG$.

It is easy to see that every inner automorphism $\alpha \in Aut(G)$ stabilizes $\GG$ pointwise and thus we obtain an action of $Out(G)$ on $\GG$. This action is order preserving and leaves $\H (G)$ and $\AHG$ setwise invariant. Recall that every acylindrically hyperbolic group has a (unique) maximal finite normal  subgroup, called the \emph{finite radical} of $G$  and denoted $K(G)$ \cite[Theorem 2.24]{DGO}. If $K(G)=1$, Theorem \ref{main2} can be used to construct, for every automorphism $\alpha \in Out (G)$, an acylindrically hyperbolic structure on $G$ that is not fixed by $\alpha$. More generally, we prove the following.

\begin{thm}\label{main6}
Let $G$ be an acylindrically hyperbolic group.
\begin{enumerate}
\item[(a)] If $G$ is finitely generated, then the kernel of the action of $Out(G)$ on $\AHG$ is finite.
\item[(b)] If $G$ has trivial finite radical, then the action of $Out(G)$ on $\AHG$ is faithful.
\end{enumerate}
\end{thm}

The hypothesis that $G$ is finitely generated in part (a) cannot be dropped (see Example \ref{infker}). Our proof of Theorem \ref{main6} also makes use of results about pointwise inner automorphisms of acylindrically hyperbolic groups obtained in \cite{AMS}.

\paragraph{2.4. Loxodromic equivalence and rigidity.}
Recall that an element $g$ of a group $G$ acting on a metric space $S$ is called \emph{loxodromic} if the map $\mathbb Z\to S$ defined by $n\mapsto g^ns$ is a quasi-isometry for some (equivalently, any) base point $s\in S$. It is easy to see that equivalent actions of $G$ have the same sets of loxodromic elements. Thus we can define the set of \emph{loxodromic elements of a hyperbolic structure} $A=[X]\in \HG$, denoted $\L (A)$, as the set of all elements acting loxodromically on $\Gamma (G, X)$.

\begin{defn}\label{def-loxeq}
We say that two structures $A,B\in \HG$ are \emph{loxodromically equivalent} (written $A\sim _\L B$) if $\L (A)=\L(B)$.  We denote by $[A]^{\AH}_\L =\{ B\in \AHG \mid A\sim _L B\}$ the \emph{loxodromic equivalence class} of $A\in \AHG$.
\end{defn}

We first show that hyperbolic structures (and even acylindrically hyperbolic ones) are not determined by their sets of loxodromic elements. Moreover, for every non-elementary $A\in \AHG$, the loxodromic equivalence class $[A]^\AH_{\L}$ is as complicated as the poset $\AHG$ itself.

\begin{thm}\label{main3}
For every non-elementary $A\in \AHG$, the loxodromic equivalence class $[A]^\AH_{\L}$ contains an isomorphic copy of $\PN$.
\end{thm}

We say that an acylindrically hyperbolic structure $A\in \AHG$ is \emph{purely loxodromic} if $\L (A)$ consists of all elements of $G$ of infinite order. For example, if $G$ is a hyperbolic group and $X$ is a finite generating set of $G$, then $[X]$ is purely loxodromic. Applying Theorem \ref{main3} to this structure, we obtain the following corollary. It can be thought as a generalization of a result of I. Kapovich \cite{Kap}, which in our terms states that the free group of rank $2$ has at least $2$ purely loxodromic acylindrically hyperbolic structures.

\begin{cor}\label{hpl}
For every non-elementary hyperbolic group $G$, there exist $2^{\aleph _0}$ distinct purely loxodromic acylindrically hyperbolic structures.
\end{cor}

Our next result can be thought of as a hyperbolic analogue of the Culler--Morgan theorem \cite{CM} stating that a minimal non-quasi-parabolic action of a group $G$ on an $\mathbb R$-tree is determined up to a $G$-equivariant isometry by the set of loxodromic elements of $G$ and their translation lengths. We prove a similar theorem for general hyperbolic spaces with $G$-equivariant isometry replaced by the equivalence of actions. Recall that two actions of a group $G$ on metric spaces $S$ and $T$ are \emph{equivalent} if there exists a coarsely $G$-equivariant quasi-isometry $S\to T$.

Let $A=G\curvearrowright S$ be an action of a group $G$ on a metric space $S$ and let $g\in G$. The corresponding \emph{translation number} of $g$ is defined by
\begin{equation}\label{def-trn}
\tau_{A} (g)=\lim_{n\to \infty}\frac{\d (s,g^ns)}{n},
\end{equation}
where $s\in S$. It is well-known (and easy to check) that the limit always exists and is independent of the choice of the basepoint $s$.

\begin{defn}[Coarsely isospectral actions]\label{def-ci}
We say that two actions $A=G\curvearrowright S$ and $B=G\curvearrowright T$ of the same group $G$ are \emph{coarsely isospectral} if for every sequence of elements $(g_i)_{i\in \mathbb N} \subseteq G$, we have $$\lim\limits_{i\to\infty}\tau_{A}(g_i)=\infty \;\; \Longleftrightarrow \;\; \lim\limits_{i\to\infty}\tau_{B}(g_i)=\infty .$$
\end{defn}

\begin{thm}\label{main4}
Let $G\curvearrowright R$ and $G\curvearrowright S$ be non-quasi-parabolic cobounded actions of a group $G$ on hyperbolic spaces $S$ and $R$. Then $G\curvearrowright R$ and $G\curvearrowright S$ are coarsely isospectral if and only if $G\curvearrowright R$ and $G\curvearrowright S$ are equivalent.
\end{thm}

In Example \ref{pqcis}, we show that for quasi-parabolic actions Theorem \ref{main4} does not hold. It is also worth noting that the assumption that the actions are cobounded cannot be dropped, see Example \ref{ex-non-cb}.

Equivalence classes of cobounded actions of a group $G$ on hyperbolic spaces are in one-to-one correspondence with hyperbolic structures on $G$ via the Svarc-Milnor map, see Section 2.2; thus our theorem can be equivalently restated in terms of hyperbolic structures (the notion of coarse isospectrality for structures is defined in the natural way, see Definition \ref{def-cis}). In particular, combining Theorem \ref{main4} with the fact that acylindrical actions cannot be quasi-parabolic, we obtain the following.

\begin{cor}\label{cormain4}
Coarsely isospectral acylindrically hyperbolic structures on the same group coincide.
\end{cor}

\paragraph{2.5.  Hyperbolic and acylindrically hyperbolic accessibility.}
The famous Stallings' theorem states that every finitely generated group with infinitely many ends splits as the fundamental group of a graph of groups with finite edge groups. This was a starting point of an accessibility theory developed by Dunwoody. A finitely generated group $G$ is said to be \emph{accessible} if the process of iterated nontrivial splittings of $G$ over finite subgroups always terminates in a finite number of steps. Not every finitely generated group is accessible \cite{Dun93}, but finitely presented groups are \cite{Dun85}, as well as torsion free groups \cite{Lin}.

More generally, one can ask whether a given group has a maximal, in a certain precise sense, action on a tree satisfying various additional conditions on stabilizers (see, for example,  \cite{BF,S}). Yet another problem of similar flavor studied in the literature is whether a given group admits a maximal relatively hyperbolic structure \cite{BDM}. It is natural to ask a similar question in our setting.

\begin{defn}
We say that a group $G$ is \emph{$\H$-accessible} (respectively \emph{$\AH$-accessible}) if $\H (G)$ (respectively $\AHG$) contains the largest element.
\end{defn}

We begin with examples of inaccessible groups. It is easy to find examples of groups which are not $\H$-accessible, e.g., the direct product $F_2\times F_2$; however, this group is $\AH$-accessible, see Section 7.1. Finding $\AH$-inaccessible groups, especially finitely generated or finitely presented ones, is more difficult. We first show the following (see Theorem \ref{AH-inacc}).

\begin{thm} \label{inacc}
There exists a finitely presented group that is neither $\H$-accessible nor $\AH$-accessible.
\end{thm}

On the other hand, we prove that many groups traditionally studied in geometric group theory are $\AH$-accessible.

\begin{thm}\label{main5} The following groups are $\AH$-accessible.
\begin{enumerate}
\item[(a)] Finitely generated relatively hyperbolic groups whose parabolic subgroups are not acylindrically hyperbolic.
\item[(b)] Mapping class groups of punctured closed surfaces.
\item[(c)] Right-angled Artin groups.
\item[(d)] Fundamental groups of compact orientable $3$-manifolds with empty or toroidal boundary.
\end{enumerate}
\end{thm}

We would like to note that following an early draft of this paper, parts (b) and (c) of the above theorem were independently and subsequently proven in \cite{ABD}, which additionally proves the $\AH$-accessibility of certain other groups using different methods.  The special case of part (d) of the above theorem when the 3-manifold has no Nil or Sol in its prime decomposition is also proven in \cite{ABD}.

\paragraph{2.6. Organization of the paper.} In the next section we introduce several useful notions (such as equivalence and weak equivalence of group actions, Svarc-Milnor map, etc.), which will be used throughout the paper. Section 4 is devoted to the general classification and examples of hyperbolic structures; Theorems \ref{main0}, \ref{lineal}, and \ref{n-gt-intr} are proved there; we also prove Theorem  \ref{main1} modulo Theorem \ref{main3}, which is proved later. In Section 5 we recall the definition of a hyperbolically embedded subgroup and discuss induced hyperbolic structures and their applications; in particular we prove Theorems \ref{main2}, and \ref{main6}. Theorems \ref{main3}, \ref{main4} and other results about loxodromic equivalence  are proved in Section 6. In Section 7 we prove results about the $\H$-- and $\AH$--accessibility of groups and prove Theorems \ref{inacc} and \ref{main5}. Finally we discuss some open problems in Section 8.


\section{Preliminaries}


\subsection{Comparing group actions}\label{lattice}

We begin by recalling some standard terminology. Throughout this paper, all group actions on metric spaces are isometric by default. Our standard notation for an action of a group $G$ on a metric space $S$ is $G\curvearrowright S$. Given a point $s\in S$ or a subset $X\subseteq S$ and an element $g\in G$, we denote by $gs$ (respectively, $gX$) the image of $s$ (respectively $X$) under the action of $g$. Given a group $G$ acting on a space $S$ and some $s\in S$, we also denote by $Gs$ the $G$-orbit of $s$.

In order to avoid dealing with proper classes we fix a cardinal number $c\ge 2^{\aleph _0}$ and, henceforth, we assume that all metric spaces have cardinality at most $c$.

An action of a group $G$ on a metric space $S$ is said to be

\begin{enumerate}
\item[--] \emph{proper} is for every bounded subset $B\subseteq S$ the set $\{ g\in G\mid gB\cap B\ne \emptyset\}$ is finite;

\item[--] \emph{cobounded} if there exists a bounded subset $B\subseteq S$  such that $S=\bigcup_{g\in G} gB$;

\item[--] \emph{geometric} if it is proper and cobounded.
\end{enumerate}

Given a metric space $S$, we denote by $\d_S$ the distance function on $S$ unless another notation is introduced explicitly. A map $f\colon R\to S$ between two metric spaces $R$ and $S$ is a \emph{quasi-isometric embedding} if there is a constant $C$ such that for all $x,y\in R$ we have
\begin{equation}\label{def-qi}
\frac1C\d_R(x,y)-C\le \d_S(f(x),f(y))\le C\d_R(x,y)+C;
\end{equation}
if, in addition, $S$ is contained in the $C$--neighborhood of $f(R)$, $f$ is called a \emph{quasi-isometry}.  Two metric spaces $R$ and $S$ are \emph{quasi-isometric} if there is a quasi-isometry $R\to S$. It is well-known and easy to prove that quasi-isometry of metric spaces is an equivalence relation.

If a group $G$ acts on metric spaces $R$ and $S$, a map $f\colon R\to S$ is called \emph{coarsely $G$-equivariant} if for every $r\in R$, we have
\begin{equation}\label{def-cGe}
\sup_{g\in G} \d_S(f(gr), gf(r))<\infty.
\end{equation}

Finally we recall the definition of equivalent group actions introduced in \cite{AHO}. Two actions  $G\curvearrowright R$ and $G\curvearrowright S$ are \emph{equivalent}, denoted $G\curvearrowright R\sim G\curvearrowright S$, if there exists a coarsely $G$-equivariant quasi-isometry $R\to S$. It is easy to prove (see \cite{AHO}) that $\sim$ is indeed an equivalence relation.

\begin{defn}\label{def-poset}
Let $G$ be a group. We say that $G\curvearrowright R$ \emph{dominates} $G\curvearrowright S$ and write $G\curvearrowright S \preceq G\curvearrowright R$ if there exist $r\in R$, $s\in S$, and a constant $C$ such that
\begin{equation}\label{dSdR}
\d_S(s,gs)\le C\d_R(r, gr)+C
\end{equation}
for all $g\in G$.
\end{defn}

\begin{ex}\label{bo-ex}
Assume that the action $G \curvearrowright S$ has bounded orbits. Then  $G \curvearrowright S\preceq G\curvearrowright R$ for any other action of $G$ on a metric space $R$.
\end{ex}

Equivalently, we could define the relation $\preceq$ as follows.

\begin{lem}\label{EA}
$G\curvearrowright S \preceq G\curvearrowright R$ if and only if for any $r\in R$ and any $s\in S$ there exists a constant $C$ such that (\ref{dSdR}) holds for all $g\in G$.
\end{lem}

\begin{proof}
The backward implication is obvious. The forward implication follows immediately from the obvious observation that for any action of a group $G$ on a metric space $Z$ and any $x,y\in Z$, we have $|\d_Z(x,gx)-\d_Z(y,gy)|\le 2 \d_Z(x,y)$.
\end{proof}

Recall that we assume all metric spaces to have cardinality at most $c$.

\begin{cor}\label{preorder}
The relation $\preceq $ is a preorder on the set of all $G$-actions on metric spaces.
\end{cor}

\begin{proof}
The relation $\preceq$ is obviously reflexive and is transitive by Lemma \ref{EA}.
\end{proof}

In general, $G \curvearrowright S\preceq G\curvearrowright R$ and $G \curvearrowright R\preceq G\curvearrowright S$ does not imply $G \curvearrowright S\sim G\curvearrowright R$.
However, Corollary \ref{preorder} allows us to introduce the following.

\begin{defn} \label{def-we}
We say that two actions of a group $G$ on metric spaces $R$ and $S$ are \emph{weakly equivalent } if $G\curvearrowright R\preceq G\curvearrowright S$ and $G\curvearrowright S \preceq G\curvearrowright R$.  We use the notation $\sim _w$ for weak equivalence of group actions.
\end{defn}

It is sometimes convenient to use the following alternative definition of weak equivalence.

\begin{lem}\label{we-alt}
Two actions $G\curvearrowright R$ and $G\curvearrowright S$ are weakly equivalent if and only if there exists a coarsely $G$--equivariant quasi-isometry from a $G$-orbit in $R$ (endowed with the metric induced from $R$) to a $G$-orbit in $S$ (endowed with the metric induced from $S$).
\end{lem}

\begin{proof}
The backward implication is an immediate consequence of Definition \ref{def-we}. To prove the forward implication, we fix any $r\in R$, $s\in S$, and define $\d_G^R\colon G\times G\to \mathbb R$ by $$\d_{G}^R(g,h)=\left\{ \begin{array}{cc}
                                                                 \d_R (gr, hr)+1 & {\rm if}\; g\ne h,  \\
                                                                 0 & {\rm if}\; g=h
                                                               \end{array}\right.$$
and similarly $\d_G^S\colon G\times G\to \mathbb R$. It is easy to see that $\d_G^R$ and $\d_G^S$ are metrics on $G$. Weak equivalence of the actions $G\curvearrowright R$ and $G\curvearrowright S$ together with Lemma \ref{EA} imply that the identity map on $G$ gives rise to a quasi-isometry between metric spaces $(G, \d_G^R)$ and $(G, \d_G^S)$; it is also obvious that the orbit map $G\to Gr$ gives rise to a quasi-isometry $(G, \d_G^R)\to (Gr, \d_R)$ and similarly $(G, \d_G^S)$ is quasi-isometric to $(Gs, \d_S)$. Note that all quasi-isometries mentioned in the previous sentence are $G$-equivariant. Therefore, we have
$$
(Gr, \d_R)\sim (G, \d_G^R)\sim (G, \d_G^S)\sim (Gs,\d_S),
$$
which implies that there exists a coarsely $G$-equivariant quasi-isometry between $(Gr,\d_R)$ and $(Gs,\d_S)$.
\end{proof}

\begin{rem}
If the action $G\curvearrowright R$ is free, we can simplify the proof of the lemma by verifying that the map $gr\mapsto gs$ is a $G$-equivariant quasi-isometry from $(Gr, \d_R)$ to $(Gs, \d_S)$. If the action is not free, this map may not be well-defined. The auxiliary spaces used in the proof allow us to overcome this problem.
\end{rem}

\begin{lem}\label{e-we} Let $G\curvearrowright R$ and $G \curvearrowright S$ be two actions of a group $G$ on metric spaces.
\begin{enumerate}
\item[(a)] If $G\curvearrowright R\sim  G \curvearrowright S$, then $G\curvearrowright R\sim_w G \curvearrowright S$.
\item[(b)] Suppose that the actions are cobounded and $G\curvearrowright R\sim _w G \curvearrowright S$. Then $G\curvearrowright R\sim G \curvearrowright S$.
\end{enumerate}
\end{lem}

\begin{proof}
(a) Let $f\colon R\to S$ be a coarsely $G$-equivariant quasi-isometry. Let $C$ be a constant such that both (\ref{def-qi}) and (\ref{def-cGe}) are satisfied. We fix any $r\in R$ and let $s=f(r)$. Then for every $g\in G$, we have
$$
\d_S(s,gs)=\d_S(f(r),gf(r))\le \d_S(f(r), f(gr))+C \le C\d_R(r, gr)+2C.
$$
Hence $G\curvearrowright S\preceq  G \curvearrowright R$ and similarly $G\curvearrowright R\preceq  G \curvearrowright S$.

(b) By Lemma \ref{we-alt}, there is a coarsely $G$-equivariant quasi-isometry from a $G$-orbit $Gr\subseteq R$ to a $G$-orbit $Gs\subseteq S$. Since the actions are cobounded, the inclusions $(Gr, \d_R)\to R$ and $(Gs, \d_S)\to S$ are $G$-equivariant quasi-isometries as well and the claim follows.
\end{proof}

\subsection{Cobounded actions and the Svarc-Milnor map} \label{Sec-SM}

Given a group $G$, let $\AcG$ denote the set of all equivalence classes of cobounded $G$-actions on geodesic metric spaces (of cardinality at most $c$).  We define a relation $\preccurlyeq$ on $\AcG$ by
$$
[G\curvearrowright R]\preccurlyeq [G\curvearrowright S] \;\; \Leftrightarrow \;\; G\curvearrowright R\preceq G\curvearrowright S.
$$

It follows from Corollary \ref{preorder} and Lemma \ref{e-we} that $\preccurlyeq $ is a (well-defined) order relation on $\AcG$. We will show that the poset $\AcG$ can be naturally identified with the set of equivalence classes of generating sets of $G$ introduced in Definition \ref{def-GG}. There is little originality (if any) in this result, which is essentially a reformulation of the well-known Svarc-Milnor Lemma. The real goal of this subsection is rather to introduce convenient notation and terminology for future use.

\begin{lem}\label{eCay}
Let $X$, $Y$ be generating sets of a group $G$. Then $G\curvearrowright \Gamma(G,X)\preceq G\curvearrowright \Gamma(G,Y)$ if and only if $X\preceq Y$. In particular, $G\curvearrowright \Gamma(G,X)\sim G\curvearrowright \Gamma(G,Y)$ if and only if $X\sim Y$.
\end{lem}

\begin{proof}
Suppose that $G\curvearrowright \Gamma(G,X)\preceq G\curvearrowright \Gamma(G,Y)$. By Lemma \ref{EA} we can assume that (\ref{dSdR}) holds for the base points $s=1$ in $S=\Gamma(G,X)$ and $r=1$ in $R=\Gamma(G,Y)$. Thus for every $y\in Y$, we obtain
$$
|y|_X =\d_X(y,1)\le C\d_Y(y,1)+C=2C
$$
Thus $X\preceq Y$. Conversely, if $X\preceq Y$, then (\ref{dSdR}) holds with $C=\sup_{y\in Y}|y|_X$ for the choice of base points as above.
\end{proof}

\begin{lem}\label{lem-Lip}
Let $G$ be a group generated by a set $X$ and acting on a metric space $S$. Suppose that for some $s\in S$, we have $\sup _{x\in X} \d_S(s, xs)<\infty $. Then the orbit map $g\mapsto gs$ is a Lipschitz map from $(G, \d_X)$ to $S$. In particular, if $G$ is finitely generated, the orbit map is always Lipschitz.
\end{lem}

\begin{proof}
Let $g\in G$. Suppose that $g=x_1x_2\ldots x_m$ for some $x_1, x_2, \ldots, x_m\in X^{\pm 1}$, where $m=\d_X(1,g)$.  Then we have
\begin{equation}\label{qi1}
\begin{array}{rcl}
\d_S(s, gs)& \le &  \d_S (s, x_1s) + \d_S (x_1s, x_1x_2s) +\cdots + \d_S(x_1\cdots x_{m-1}s , x_1\cdots x_ms)\\&&\\&\le & \d_S (s, x_1s) + \d_S (s, x_2s) +\cdots + \d_S(s , x_ms) \\ &&\\ &\le & \d_X(1,g) \sup _{x\in X} \d_S(s, xs).
\end{array}
\end{equation}
\end{proof}

The following is a variation of the well-known Svarc-Milnor Lemma. The proof is standard; we provide it for the convenience of the reader.

\begin{lem}\label{lemMS}
Let $G$ be a group acting coboundedly on a geodesic metric space $S$. Let $B\subseteq S$ be a bounded subset such that $\bigcup_{g\in G} gB=S$. Let $D={\rm diam} (B)$ and let $b$ be any point of $B$. Then the group $G$ is generated by the set
$$
X= \{ g\in G\mid \d_S(b, gb)\le 2D+1\}
$$
and the natural action of $G$ of its Cayley graph $\Gamma (G, X)$ is equivalent to $G\curvearrowright S$.
\end{lem}

\begin{proof}
It suffices to show that the orbit map $\phi \colon g\mapsto gb$ is a $G$--equivariant quasi-isometry from the vertex set of $\Gamma (G, X)$ (which we identify with $G$) to $S$.

The $G$-equivariance of $\phi$ is obvious. Given $g\in G$, let $\gamma\colon [0,L]\to S $ be a geodesic between $b$ and $gb$ parameterized by length. Let $n=\lfloor L\rfloor$ and let
$$
s_0=\gamma(0)=b,\; s_1=\gamma(1),\; \ldots ,\; s_n=\gamma (n),\; s_{n+1}=\gamma (L)=gb.
$$
Since $\bigcup_{g\in G} gB=S$, for every $i=0, 1, \ldots, n+1$, there exists $g_i\in G$ such that $\d_S(s_i, g_ib)\le D$. Without loss of generality we can assume that $g_0=1$ and $g_{n+1}=g$. Let $x_i=g_{i-1}^{-1}g_i$ for $i=1, \ldots, n$. For every $i=1, \ldots, n$, we have
$$
\d_S( b, x_ib)= \d_S (g_{i-1}b, g_i{b})\le \d_S(g_{i-1}b, s_{i-1})+\d_S(s_{i-1}, s_i) + \d_S(s_i, g_ib) \le 2D+1.
$$
Thus $x_i\in X$. Clearly $g=x_1x_2\cdots x_{n+1}$ and hence $X$ generates $G$ and $\d_X(1,g)\le n+1\le \d_S(b, gb)+1$. It follows immediately from Lemma \ref{lem-Lip} and the definition of $X$ that $\phi $ is Lipschitz. Since the action of $G$ on $S$ is cobounded, $\phi$ is indeed a quasi-isometry.
\end{proof}

\begin{prop}\label{prop-SM}
The map $\GG\to \AcG$ defined by $[X]\mapsto [G \curvearrowright\Gamma (G,X)]$ for every $[X]\in \GG$ is well-defined and is an isomorphism of posets.
\end{prop}
\begin{proof}
That the map is order-preserving and injective follows from Lemma \ref{eCay}. Surjectivity follows from Lemma \ref{lemMS}.
\end{proof}

\begin{defn}[Svarc-Milnor map]
Given a group $G$, we denote by $\sigma \colon \AcG\to \GG$ the inverse of the isomorphism described in Proposition \ref{prop-SM}. We call $\sigma $ the  \emph{Svarc-Milnor map}.
\end{defn}

It follows from Proposition \ref{prop-SM} that $\sigma $ can be alternatively defined as an isomorphism of posets $\AcG\to \GG$ such that for every cobounded action $G\curvearrowright S$, we have

\begin{equation}\label{SM-altdef}
G\curvearrowright \Gamma (G, X) \sim G\curvearrowright S
\end{equation}
for every $X\in \sigma ([G\curvearrowright S])$.

In particular, the Svarc-Milnor map associates hyperbolic (respectively, acylindrically hyperbolic) structures on a group $G$ to cobounded actions (respectively, cobounded acylindrical actions) of $G$ on hyperbolic spaces. Indeed this follows from (\ref{SM-altdef}), the well-known fact that hyperbolicity of a geodesic space is a quasi-isometry invariant, and the obvious fact that acylindricity of an action is preserved under the equivalence.

Sometimes, we can also extract hyperbolic structures from non-cobounded actions. The following result will be used several times in this paper.  A subset $T$ of a hyperbolic space $S$ is called \emph{quasi-convex} if there exists a constant $\rho \ge 0$ such that every geodesic in $S$ connecting two points of $T$ belongs to the closed $\rho$-neighborhood of $T$.

\begin{prop} \label{cobdd}
Let $G\curvearrowright S$ be an action (respectively, acylindrical action) of $G$ on a hyperbolic space such that $G$ has a quasi-convex orbit $Gs$ for some $s\in S$.  Then there exists $[X] \in \HG$ (respectively, $\AHG$) such that $G \acts S \sim_w G \acts \Ga(G, X)$.
\end{prop}

\begin{proof} The idea of the proof is to first construct a cobounded action of $G$ on a graph $\Gamma$ quasi-isometric to the orbit $Gs$, and then apply the Svarc-Milnor map.  Letting $\rho$ be the quasi-convexity constant of the orbit $Gs$, we construct the graph $\Ga$ as follows:  vertices of $\Ga$ are elements of the orbit $Gs$, and there is an edge between two vertices $g_1s$ and $g_2s$ if $\d_S(g_1s,g_2s) \leq 2\rho +1$.  We consider $\Gamma$ to be a metric space with the combinatorial metric.  The action of $G$ on the vertices of $\Ga$ is induced by the action of $G$ on $Gs$, and then extended to edges. It is easy to check that
\begin{equation} \label{qiineq}
d_\Ga (u ,v) \leq d_S(u,v) \leq (2\rho +1) d_\Ga(u,v)
\end{equation}
for all $u,v \in Gs.$

Define a map $\phi\colon \Gamma \to S$ as follows:  for each vertex $u$ of $\Gamma$, let $\phi(u)=u$, and for each edge $e$ of $\Gamma$, let $\phi(e)$ be a geodesic in $S$ from $\phi(e_-)$ to $\phi(e_+)$.  Then (\ref{qiineq}) implies that $\phi$ is a quasi-isometric embedding, and so by \cite[Theorem III.H.1.9]{BH}, $\Gamma$ is hyperbolic.

The inequality (\ref{qiineq}) also implies that $G\acts \Gamma \sim _w G\acts S$. Taking $X\in \sigma ([G\acts \Gamma])$ and applying Proposition \ref{prop-SM} completes the proof.
\end{proof}


\section{Hyperbolic structures on groups: general classification and examples}


\subsection{Types of hyperbolic structures}

We begin by recalling some standard facts about groups acting on hyperbolic spaces. For details the reader is referred to \cite{Gro}.

In this paper we employ the definition of hyperbolicity via the Rips condition. That is, a metric space $S$ is called \emph{$\delta$-hyperbolic} if it is geodesic and for any geodesic triangle $\Delta $ in $S$, each side of $\Delta $ is contained in the union of the closed $\delta$-neighborhoods of the other two sides.

The \emph{Gromov product} of points $x,y$ with respect to a point $z$ in a metric space $(S,d)$ is defined by
$$
(x|y)_z=\frac12( \d(x,z) +\d (y,z) -\d (x,y)).
$$

Given a hyperbolic space $S$, by $\partial S$ we denote its Gromov boundary. In general, we do not assume that $S$ is proper. Thus the boundary is defined as the set of equivalence classes of sequences convergent at infinity. More precisely, a sequence $(x_n)$ of elements of $S$ converges at infinity if $(x_i|x_j)_s\to \infty $ as $i,j\to \infty$ (this definition is clearly independent of the choice of $s$). Two such sequences $(x_i)$ and $(y_i)$ are equivalent if $(x_i|y_j)_s\to \infty$ as $i,j\to \infty$. If $a$ is the equivalence class of $(x_i)$, we say that the sequence $x_i$ converges to $a$. This defines a natural topology on $S\cup \partial S$ with respect to which $S$ is dense in $S\cup \partial S$.

From now on, let $G$ denote a group acting (by isometries) on a hyperbolic space $S$. By $\Lambda (G)$ we denote the set of limit points of $G$ on $\partial S$. That is, $$\Lambda (G)=\partial S\cap \overline{Gs},$$ where $\overline{Gs}$ denotes the closure of a $G$-orbit in $S\cup \partial S$; it is easy to show that this definition is independent of the choice of $s\in S$. The action of $G$ is called \emph{elementary} if $|\Lambda (G)|\le 2$ and \emph{non-elementary} otherwise. The action of $G$ on $S$ naturally extends to a continuous action of $G$ on $\partial S$.

Recall that an element $g\in G$ is called
\begin{enumerate}
\item[--] \emph{elliptic} if $\langle g\rangle $ has bounded orbits;
\item[--] \emph{loxodromic} if the orbits of $\langle g\rangle $ are quasi-convex (equivalently, the translation number of $g$ is positive);
\item[--] \emph{parabolic} otherwise.
\end{enumerate}

Every loxodromic element $g\in G$ has exactly $2$ fixed points $g^{\pm\infty}$ on $\partial S$, where $g^{+\infty}$ (respectively, $g^{-\infty}$) is the limit of the sequence $(g^ns)_{n\in \mathbb N}$ (respectively, $(g^{-n}s)_{n\in \mathbb N}$) for any fixed $s\in S$. We clearly have $\Lambda (\langle g\rangle) =\{ g^{\pm \infty}\}$. Loxodromic elements $g,h\in G$ are called \emph{independent} if the sets $\{ g^{\pm \infty}\}$ and $\{ h^{\pm \infty}\}$ are disjoint.

Recall that a path $p$ in a hyperbolic space $S$ is called \emph{$(\lambda, c)$-quasi-geodesic} for some constants $\lambda\ge 1$, $c\ge 0$ if for every subpath $q$ of $p$ we have $$\d (q_-, q_+)\le \lambda \ell (q)+c,$$ where $q_-$, $q_+$ denote the starting and the ending points of $q$, respectively, and $\ell (q)$ denotes the length of $q$. We also say that a path is \emph{quasi-geodesic} if it is $(\lambda, c)$-quasi-geodesic for some constants $\lambda\ge 1$, $c\ge 0$; of course, this definition only makes sense for infinite paths.

Every loxodromic element $g\in G$ preserves a bi-infinite quasi-geodesic $l_g$ in $S$; adding $g^{\pm\infty}$ to $l_g$, we obtain a path in $S\cup \partial S$  that connects $g^{+\infty}$ to $g^{-\infty}$. Such a path is called a \emph{quasi-geodesic} \emph{axis} (or simply an axis) of $g$. Given any $s\in S$, we can always construct an axis of $g$ that contains $s$: take the bi-infinite sequence $\ldots, g^{-2}s, g^{-1}s, s, gs, g^2s, \ldots $ and connect consecutive points by geodesics in $S$.

Given a space $S$ with a metric $\d$, we denote by $\d^{Hau}$ the corresponding Hausdorff pseudo-metric on the set of subsets of $S$.  We record the following elementary (and well-known) observation which will be used in the proof of Theorem \ref{main00} below. A brief sketch of the proof is provided for convenience of the reader.

\begin{lem}\label{indlox}
Let $f,g\in G$ be two loxodromic elements of a group $G$ acting on a hyperbolic space $S$. Then $f^{+\infty}=g^{+\infty}$ if and only if for some (equivalently, any) $s\in S$, we have
\begin{equation}\label{dHaufg}
\d^{Hau} (\{ f^ns\mid n\in \mathbb N\},  \{ g^ns\mid n\in \mathbb N\})<\infty .
\end{equation}
\end{lem}

\begin{proof}
Let $l_g$, $l_f$ be quasi-geodesic axes of $f$ and $g$, respectively. Without loss of generality we can assume that $s\in l_g\cap l_f$. Let $l^+_g$, $l^+_f$ be subpaths of $l_g$ and $l_f$ starting at $s$ and going to $f^{+\infty}=g^{+\infty}$. Then $\d^{Hau} (\{ g^ns\mid n\in \mathbb N\},  l^+_g)<\infty $ and similarly for $f$. Thus (\ref{dHaufg}) is equivalent to $\d^{Hau} (l_f^+, l_g^+)<\infty $. Since the Hausdorff distance between two quasi-geodesic rays in $S$ is finite if and only if these rays converge to the same point of $\partial S$, we obtain the result.
\end{proof}

The following theorem summarizes the standard classification of groups acting on hyperbolic spaces due to Gromov \cite[Section 8.2]{Gro}  (see also \cite{H} for complete proofs in a more general context) and some results from \cite[Propositions 3.1 and 3.2]{CCMT}.

\begin{thm}\label{ClassHypAct}
Let $G$ be a group acting on a hyperbolic space $S$. Then exactly one of the following conditions holds.
\begin{enumerate}
\item[1)] $|\Lambda (G)|=0$. Equivalently,  $G$ has bounded orbits. In this case the action of $G$ is called \emph{elliptic}.

\item[2)] $|\Lambda (G)|=1$. Equivalently, $G$ has unbounded orbits and contains no loxodromic elements. In this case the action of $G$ is called \emph{parabolic}. A parabolic action cannot be cobounded and the set of points of $\partial S$ fixed by $G$ coincides with $\Lambda (G)$.

\item[3)] $|\Lambda (G)|=2$. Equivalently, $G$ contains a loxodromic element and any two loxodromic elements have the same limit points on $\partial S$. In this case the action of $G$ is called \emph{lineal}.

\item[4)] $|\Lambda (G)|=\infty$. Then $G$ always contains loxodromic elements. In turn, this case breaks into two subcases.
\begin{enumerate}
\item[(a)] $G$ fixes a point of $\partial S$. Equivalently, any two loxodromic elements of $G$ have a common limit point on the boundary. In this case the action of $G$ is called \emph{quasi-parabolic}. Orbits of quasi-parabolic actions are always quasi-convex.
\item[(b)] $G$ does not fix any point of $\partial S$. Equivalently, $G$ contains infinitely many independent loxodromic elements. In this case the action of $G$ is said to be of \emph{general type}.
\end{enumerate}
\end{enumerate}
\end{thm}

Parabolic and quasi-parabolic acylindrical actions do not exist. Moreover, we have the following \cite{Osi16}.

\begin{thm}\label{class}\label{tricho}
Let $G$ be a group acting acylindrically on a hyperbolic space. Then exactly one of the following three conditions holds.
\begin{enumerate}
\item[(a)] The action is elliptic.
\item[(b)] The action is lineal and $G$ is virtually cyclic.
\item[(c)] The action is of general type.
\end{enumerate}
\end{thm}

In particular, being non-elementary is equivalent to being of general type for acylindrical actions.

\begin{lem}\label{we-type}
Let $G\curvearrowright R$ and $G\curvearrowright S$ be weakly equivalent actions of $G$ on hyperbolic spaces. Then $G\curvearrowright R$ and $G\curvearrowright S$ have the same type.
\end{lem}

\begin{proof}
By Lemma \ref{we-alt}, there exists a coarsely $G$-equivariant quasi-isometry from a $G$-orbit in $R$ to a $G$-orbit in $S$. It is straightforward to check that this quasi-isometry gives rise to a $G$-equivariant map (in fact, homeomorphism) $\Lambda _R(G)\to \Lambda _S(G)$, where $\Lambda _R(G)$ and $\Lambda _S(G)$ denote the limit sets of $G$ for the actions $G\curvearrowright R$ and $G \curvearrowright S$ respectively. It remains to notice that the type of action of $G$ on a hyperbolic space is uniquely determined by the cardinality of the corresponding limit set $\Lambda (G)$ and the existence of fixed points in $\Lambda (G)$; clearly these are invariant under $G$-equivariant maps.
\end{proof}

The following result will be used many times throughout the paper. It shows that whether an equivalence class $[X]\in \GG$ belongs to $\H (G)$ or $\AHG$ or is of a certain type is independent of the choice of a particular representative in $[X]$

\begin{prop}
Let $X$ and $Y$ be equivalent generating sets of a group $G$. Then the following hold.
\begin{enumerate}
\item[(a)] $\Gamma (G,X)$ is hyperbolic if and only if $\Gamma (G,Y)$ is.
\item[(b)] The action $G \curvearrowright \Gamma (G,X)$ is acylindrical if and only if $G \curvearrowright \Gamma (G,Y)$ is.
\item[(c)] The action $G \curvearrowright \Gamma (G,X)$ is elliptic (respectively lineal, quasi-parabolic, of general type) if and only if so is $G \curvearrowright \Gamma (G,Y)$.
\end{enumerate}
\end{prop}

\begin{proof}
Since $X\sim Y$, the identity map $G\to G$ gives rise to a (coarsely $G$-equivariant) quasi-isometry of metric spaces $\Gamma (G,X)\to \Gamma (G,Y)$. This easily implies (a), (b). Part (c) follows from Lemma \ref{eCay} and Lemma \ref{we-type}.
\end{proof}

Thus we obtain the following classification of hyperbolic structures. Recall that the sets of elliptic, lineal, quasi-parabolic, and general type hyperbolic structures on $G$ are denoted by $\H_e(G)$, $\H_{\ell} (G)$, $\H_{qp} (G)$, and $\H_{gt}(G)$ respectively. We use analogous notation for acylindrically hyperbolic structures.

\begin{thm}\label{main00}
For every group $G$, the following holds.
\begin{enumerate}
\item [(a)] $$\H(G)=\H_e(G)\sqcup \H_{\ell} (G)\sqcup \H_{qp} (G)\sqcup \H_{gt}(G)$$ and the subsets $\H_e(G)\sqcup \H_{\ell} (G)$ and $\H_e(G)\sqcup \H_{\ell} (G)\sqcup \H_{qp} (G)$ are initial segments of $\H(G)$.
\item[(b)] Either $$\AHG=\AH_e(G)\sqcup \AH_\ell(G)$$ (if $G$ is virtually cyclic) or $$\AHG=\AH_e(G)\sqcup \AH_{gt}(G)$$ (if $G$ is acylindrically hyperbolic).
\end{enumerate}
\end{thm}

\begin{proof}
The first claim in part (a) follows immediately from the fact that parabolic actions cannot be cobounded, see Theorem \ref{ClassHypAct}.

Let us prove that $\H_e(G)\sqcup \H_{\ell} (G)\sqcup \H_{qp} (G)$ is an initial segment of $\H(G)$. Arguing by contradiction, assume that there exists $[X]\in \H _{gt}(G)$ and $[Y]\in \H_e(G)\sqcup \H_{\ell} (G)\sqcup \H_{qp} (G)$ such that $X\preceq Y$. By Theorem \ref{ClassHypAct}, there are two independent loxodromic elements $f,g\in G$ with respect to the hyperbolic structure $[X]$. Since $f^{\pm \infty}$ and $g^{\pm \infty}$ are disjoint,
we have
\begin{equation}\label{fgindep}
\d^{Hau} (\{ f^{\pm n}\mid n\in \mathbb N\},  \{ g^{\pm n}\mid n\in \mathbb N\})=\infty
\end{equation}
in $\Gamma (G,X)$ for any fixed choice of the signs in the exponents by Lemma \ref{indlox}. Since $X\preceq Y$, $f$ and $g$ are also loxodromic with respect to $[Y]$ (this is obvious if one uses the definition of loxodromic elements based on translation numbers) and (\ref{fgindep}) also holds in $\Gamma (G,Y)$. In turn, this implies that $f$ and $g$ are independent loxodromic elements with respect to the action of $G$ on $\Gamma (G,Y)$. By Theorem \ref{ClassHypAct}, this implies that $[Y]\in \H_{gt}(G)$.  The proof of the claim that $\H_e(G)\sqcup \H_{\ell} (G)$ is an initial segment of $\H(G)$ is analogous using the fact that lineal actions can be characterized by the property that any two loxodromic elements have the same limit points, see Theorem \ref{ClassHypAct}.

Finally, we note that part (b) follows immediately from Theorem \ref{tricho}.
\end{proof}

We now turn to the proof of Theorem \ref{main1}. The first step is the following elementary lemma.

\begin{lem}\label{vc}
Let $G$ be a virtually cyclic group. Let $X$ denote a finite generating set of $G$. Then $\AHG=\{  [G], [X]\}$.
\end{lem}

\begin{proof}
Let $[X]\in \AHG$. We apply Theorem \ref{class} to the action of $G$ on $\Gamma (G,X)$. It is well known that a virtually cyclic group cannot satisfy (c) (for instance, one can use the standard ping-pong argument to show that (c) implies the existence of non-cyclic free subgroups in $G$); thus we only have to consider cases (a) and (b). If $G$ has bounded orbits, we have $\sup_{g\in G}|g|_X<\infty$ and hence $X\sim G$. If $G$ contains a loxodromic element $g$, then the action of $\langle g\rangle $ on $\Gamma (G,X)$ is proper by the definition of a loxodromic element. Since $G$ is virtually cyclic, every infinite cyclic subgroup has finite index in $G$. In particular, $|G:\langle g\rangle|<\infty $ and the action of $G$ on $\Gamma (G,X)$ is also proper. This means that $X$ is finite.
\end{proof}

\begin{proof}[Proof of Theorem \ref{main1}]
If the group $G$ is virtually cyclic, then case (b) is realized by Lemma \ref{vc}. If $G$ is not virtually cyclic and not acylindrically hyperbolic, then all acylindrical actions of $G$ on hyperbolic spaces have bounded orbits by Theorem \ref{tricho}. This implies that the only acylindrically hyperbolic structure on $G$ is the trivial one, i.e., case (a) is realized. It remains to consider the case when $G$ is acylindrically hyperbolic. In this case we have (c), which follows from Theorem \ref{main3}. The latter theorem will be proved in Section 6.1 (without using Theorem \ref{main1}, of course).
\end{proof}

\subsection{Sufficient conditions for extremality}

In this section we provide two sufficient conditions for extremality of hyperbolic structures, namely Proposition \ref{minact} and Proposition \ref{maxact}, stated below. They will be later used in several places, including Sections 4.3, 4.5, and 7.2.

Recall that given a metric space $S$ with a metric $\d_S$, by $\d^{Hau}_S$ we denote the corresponding Hausdorff distance on the set of non-empty subsets of $S$.
Let $G$ be a group acting on a metric space $S$.

\begin{defn}\label{eCT}
We say that the \emph{action of $G$ on (unordered) pairs of equidistant points in $S$ is coarsely transitive} if there exists $\e\ge 0$ such that for every $x,y,s,t\in S$ satisfying $\d_S(x,y)=\d_S(s,t)$ there is $g\in G$ such that $$\d_S^{Hau}(\{ gx, gy\}, \{ s,t\})\le \e.$$ We use the term \emph{$\e$-coarsely transitive} whenever we want to stress that the definition is satisfied with a particular constant $\e$.
\end{defn}

It is not difficult to show that the property of being coarsely transitive on pairs of equidistant points is invariant under the equivalence of actions on geodesic spaces; we will not use this fact in our paper and so we leave the proof as an exercise.

In what follows, by a \emph{minimal hyperbolic structure} of a group $G$ we mean a minimal element in $\H (G)\setminus \H_e(G)$.

\begin{prop}\label{minact}
Let $G$ be a group acting coboundedly and non-elliptically on a hyperbolic space $S$. If the action of $G$ on pairs of equidistant points in $S$ is coarsely transitive, then $\sigma ([G\curvearrowright S])$ is a minimal hyperbolic structure on $G$.
\end{prop}

\begin{proof}
Fix any $s\in S$. Let $[X]= \sigma([G\curvearrowright S])$. Suppose that for some non-trivial $[Y]\in \H (G)$, we have $[Y]\preccurlyeq [X]$, i.e., $Y\preceq X$. Combining Lemma \ref{eCay} and (\ref{SM-altdef}), we obtain
$$
G\curvearrowright\Gamma (G,Y)\preceq G\curvearrowright \Gamma (G,X)\sim G\curvearrowright S.
$$
Thus there exists a constant $C$ such that for every $f\in G$ we have $\d_Y(1, f) \le C\d _S(s,fs) +C$. This easily implies
\begin{equation}\label{minact1}
\d_Y^{Hau} (A,B)\le C\d_S^{Hau}(As, Bs)+C
\end{equation}
for any $A,B\subseteq G$.

Suppose that $Y\not\sim X$. Then there exists a sequence $(y_i)_{i\in \mathbb N}\subseteq Y$ such that $|y_i|_X\to \infty $ as $i\to \infty$. Since $G\curvearrowright \Gamma (G,X)\sim G\curvearrowright S$, we have $$\ell_i=\d_S(s, y_is)\to \infty $$ as $i\to \infty$.

Let $g\in \L([Y])$. We fix any $i\in \mathbb N$ and any $n\in \mathbb N$. Let $p$ denote a geodesic in $S$ such that $p_-=s$ and $p_+=g^ns$. Let $z_0=s, z_1, \ldots, z_{k+1}=g^ns$ be a sequence of consecutive vertices on $p$ such that
\begin{equation}\label{dzj}
\d_S(z_{j-1}, z_j)=\ell_i\; {\rm for}\; 1\le j\le k,\;\; {\rm and}\;\;  \d_S(z_{k}, z_{k+1}) < \ell_i.
\end{equation}

Since the action of $G$ on pairs of equidistant points in $S$ is $\e$-coarsely transitive for some $\e$, there exist $f_0=1, f_1, \ldots , f_{k+1}=g^n$ such that
\begin{equation}\label{zfs}
\d_S(z_j, f_js)\le \e
\end{equation}
for all $1\le j\le k+1$ (we apply Definition \ref{eCT} to the pairs $(z_j,z_j)$ and $(s,s)$ here). By (\ref{dzj}) there also exists $a_j\in G$ such that
$$
\d_S^{Hau}(\{a_jz_{j-1},a_jz_j\}, \{ s, y_is\}) \le \e
$$
for all $1\le j\le k$. Combining this with (\ref{minact1}) and (\ref{zfs}), for all $1\le j\le k$, we obtain
$$
\begin{array}{rcl}
\d_Y^{Hau}(\{a_jf_{j-1}, a_jf_j\}, \{1, y_i\}, ) &\le  &C\d_S^{Hau}(\{ a_jf_{j-1}s, a_jf_js\}, \{ s, y_is\}) +C\\ && \\ &\le & C(\d_S^{Hau}(\{ a_jz_{j-1}, a_jz_j\}, \{ s, y_is\})+\e) +C \\ && \\ &\le & 2C\e +C.
\end{array}
$$
Therefore,
$$
\d_Y(f_{j-1}, f_j)= \d_Y(a_jf_{j-1}, a_jf_j) \le \d_Y(1, y_i) + 4C\e +2C\le 4C\e +2C+1
$$
for all $1\le j\le k$. Note also that
$$
\d_Y(f_k, f_{k+1})\le C\d_S(f_ks, f_{k+1}s)+C \le C (\d_S(z_k, z_{k+1})+2\e )+C <C(\ell_i+2\e+1).
$$
Since $f_0=1$ and $f_{k+1}=g^n$, applying the triangle inequality we obtain
$$
\begin{array}{rcl}
\frac{|g^n|_Y }{n}& \le & \frac{1}{n}\sum\limits_{j=1}^{k+1} \d_Y (f_{j-1}, f_j)\\&&\\&\le&  \frac{1}{n}\left((4C\e +2C+1)\frac{\d_S(s, g^ns)}{\ell_i} +C(\ell_i+2\e+1)  \right)\\&&\\&\le& \frac{\d_S(s, gs)}{\ell_i}(4C\e +2C+1) +\frac{C(\ell_i+2\e+1)}{n}.
\end{array}
$$
Letting $i\to \infty $ and $n/\ell_i\to \infty$, we obtain that $\inf_n\frac{|g^n|_Y}{n}=0$, which contradicts the assumption that $g\in \L([Y])$.
\end{proof}

\begin{rem}
The use of a loxodromic element in the proof of the proposition is essential. In particular, we cannot conclude that under the assumptions of the theorem that $\sigma ([G\curvearrowright S])$ is a minimal element in $\GG\setminus \{ [G]\}$. Indeed it is easy to see that the standard translation action of $\mathbb Z$ on $\mathbb R$ is coarsely transitive on pairs of equidistant points  but the corresponding hyperbolic structure $[X]$, where $X$ is any finite generating set of $\mathbb Z$, is not minimal in $\GG\setminus \{ [G]\}$ (for example, taking the generating set $Y=\{ 2^n\mid n\in \mathbb N\}$ we get a strictly smaller non-trivial element $[Y]\in \GG$).
\end{rem}

\begin{ex} It is well-known and easy to prove that the standard action of $PSL(2, \mathbb R)$ on $\mathbb H^2$ is transitive on pairs of equidistant points. Hence the action of every dense subgroup of $PSL(2, \mathbb R)$ is coarsely transitive on pairs of equidistant points. This yields examples of minimal hyperbolic structures on dense subgroups of $PSL(2, \mathbb R)$, e.g., on $PSL(2, \mathbb Q)$ or $F_2$ (see \cite{Ghy,Bre} for examples of dense free subgroups of $PSL(2, \mathbb R)$).
\end{ex}

It is clear that every lineal action satisfies the assumptions of Proposition \ref{minact} and hence we obtain the following.

\begin{cor}\label{linmin}
For every group $G$, every element of $\H_\ell (G)$ is minimal.
\end{cor}

In the next result, we use the order on group actions introduced in Definition  \ref{def-we}. Here and in what follows, we always think of connected graphs as metric spaces with respect to the combinatorial metric.

\begin{prop}\label{maxact} Let $G$ be a group acting cocompactly on a connected graph $\Delta$ and let $\mathcal A$ be a set of actions of $G$ on metric spaces. Suppose that for every vertex $v \in V(\Delta)$ and every action $G \acts S \in \mathcal A$, the induced action of the stabilizer $Stab_G(v)$ on $S$ has bounded orbits. Then $A\preceq G \acts \Delta$ for all $A\in \mathcal A$.
\end{prop}

\begin{proof}
Since the action $G \acts \Delta$ is cocompact, there are finitely many orbits of edges. Let $E=\{ e_1, \ldots, e_k\} $ be a set of representatives of these orbits and let $V=\{ v_1, \ldots, v_n\}$ be the set of vertices incident to edges from $E$. Let $H_i=Stab_G(v_i)$ for $i=1, \ldots, n$. Since $V$ is finite, there exists a finite set $X \subset G$ such that if $gV\cap V\ne\emptyset$ for some $g\in G$, then $g\in xH_i$ for some $x\in X$ and $i\in \{ 1, \ldots, n\}$. In particular, if we set $Y=X\cup H_1\cup \ldots\cup H_n$, then $gV\cap V\ne\emptyset$ implies $|g|_Y\le 2$.

We first show that $G$ is generated by the set $Y$. Let $g\in G$ and let $p$ be a geodesic in $\Delta $ from $v_1$ to $gv_1$, with $p=f_1\ldots f_m$, where $f_1, \ldots, f_m$ are edges of $\Delta$. For every $j\in \{ 1, \ldots, m\}$, there exists $g_j\in G$ such that $f_j\in g_jE$. In this notation, $g_jV\cap g_{j-1}V$ contains the common vertex of $f_{j-1}$ and $f_j$. Hence $g_{j-1}^{-1}g_jV \cap V\ne \emptyset$. By the choice of $X$, we have $g_{j-1}^{-1}g_j \in \langle Y\rangle $ and, moreover, $|g_{j-1}^{-1}g_j|_Y \le 2$ for all $j=2, \ldots , m$. Note also that $v_1\in V\cap g_1V$ and hence $g_1\in \langle Y\rangle $ and $|g_1|_Y\le 2$. Similarly, $gv_1\in g_mV\cap gV$, hence $g_m^{-1}gV\cap V\ne \emptyset$ and we have $g_m^{-1}g\in \langle Y\rangle $ and $|g_m^{-1}g|_Y\le 2$. Taking all these together, we obtain
$$
g=g_1 (g_1^{-1}g_2) \cdots (g_{m-1}^{-1}g_m) (g_m^{-1}g)\in \langle Y\rangle
$$
and
\begin{equation}\label{dy1g}
\d_Y(1,g)=|g|_Y\le |g_1|_Y +\sum\limits_{j=2}^m |g_{j-1}^{-1}g_j|_Y+ |g_m^{-1}g|_Y \le 2(m+1) = 2 \d_\Delta (v_1, gv_1)+2.
\end{equation}

Thus $G=\langle Y\rangle$ and moreover (\ref{dy1g}) implies that
\begin{equation}\label{GSGGY}
G \curvearrowright (G, \d_Y) \preceq G\curvearrowright \Delta.
\end{equation}

Let $A=G\curvearrowright S\in \mathcal A$. We fix any $s\in S$. Since $X$ is finite and the orbit of each $H_i$ in $S$ is bounded, there is a uniform bound on $\d_S(s, xs)$ for all $x\in Y$. By Lemma \ref{lem-Lip}, the orbit map $(G, \d_Y)\to Gs$ is Lipschitz. Hence $G\curvearrowright S \preceq G \curvearrowright (G, \d_Y)$. Combining this with (\ref{GSGGY}), we obtain $G\curvearrowright S \preceq G\curvearrowright \Delta$.
\end{proof}

As an immediate corollary of Proposition \ref{maxact} and Proposition \ref{prop-SM}, we obtain the following.

\begin{cor}\label{maxact-cor}
Let $G$ be a group acting cocompactly on a connected graph $\Delta$ and let $\mathcal F\subseteq \GG$ be any subset. Suppose that for every $v\in V(\Delta)$ and every $[X]\in \mathcal F$, the stabilizer $Stab_G(v)$ has bounded diameter with respect to $\d_X$. Then $\sigma([G\curvearrowright \Delta])$ is an upper bound for $\mathcal F$ in $\GG$.
\end{cor}

\subsection{Lineal hyperbolic structures and pseudocharacters}

Let $G$ be a group. Recall that a map $q\colon G\to \mathbb R$ is a \emph{quasi-character} (or \emph{quasi-morphism}) if there exists a constant $D$ such that $$|q(gh)-q(g)-q(h)|\le D$$ for all $g,h\in G$; one says that $q$ has \emph{defect at most $D$}. If, in addition, the restriction of $q$ to every cyclic subgroup of $G$ is a homomorphism, $q$ is called a \emph{pseudocharacter} (or \emph{homogeneous quasi-morphism}). Every quasi-character $q\colon G\to \mathbb R$ gives rise to a pseudocharacter $p\colon G\to \mathbb R$ defined by
$$
p(g)=\lim_{n\to \infty} \frac{q(g^n)}n
$$
(the limit always exists); $p$ is called the \emph{homogenization of $q$.} It is straightforward to check that
\begin{equation}\label{pg-qg}
 |p(g) -q(g)|\le D
\end{equation}
for all $g\in G$ if $q$ has defect at most $D$.

Recall that a lineal action of a group $G$ on a hyperbolic space $S$ is \emph{orientable} if no element of $G$ permutes the two limit points of $G$ on $\partial S$.

Clearly the property of being orientable is invariant under the equivalence of actions and thus we can speak of orientable lineal hyperbolic structures on a given group $G$. We denote the set of such structures by $\H^+_\ell(G)$.

\begin{lem}\label{pc-to-ls}
Let $p\colon G\to \mathbb R$ be a non-zero pseudocharacter. Let $C$ be any constant such that the defect of $p$ is at most $C/2$ and there exists a value of $p$ in the interval $(0, C/2)$. Let $$X=X_{p, C}=\{ g\in G \mid |p(g)|< C\}.$$ Then $X$ is a generating set of $G$ and the map $p\colon (G, \d_X)\to \mathbb R$ is a quasi-isometry. In particular, $[X]\in \H_\ell^+(G)$ and $\L([X]) =\{ g\in G\mid p(g)\ne 0\}$.
\end{lem}

\begin{proof}
Fix any $x\in X$ such that $p(x)\in (0, C/2)$. Given any $g\in G$, let $n= \lfloor p(g)/p(x)\rfloor$. Then $|p(g)-np(x)|<p(x)<C/2$ and hence $|p(gx^{-n})|\le |p(g) - np(x)| +C/2 < C$. Therefore $gx^{-n}\in X$. Thus $g\in \langle X\rangle $ and we have
\begin{equation}\label{|g|1}
|g|_X\le |p(g)/p(x)| +2.
\end{equation}
Note that we also have $p(g)< 1.5C|g|_X$. Combining this inequality with (\ref{|g|1}) we obtain that $p\colon (G, \d_X)\to \mathbb R$ is a quasi-isometry. In particular, this implies that $[X]\in \H_\ell (G)$ and $\L([X]) =\{ g\in G\mid p(g)\ne 0\}$.

It remains to prove that $[X]$ is orientable. Arguing by contradiction, suppose that some element $a\in G$ permutes the boundary points of $S=\Gamma (G,X)$. Then for every $h\in \L([X])$, the sequences $(a^{-1}h^na)$ and $(h^n)$ must converge to different points of $\partial S$ as $n\to \infty$. In particular, $|[h^{n}, a]|_X=\d_X (a^{-1}h^na, h^n) \to \infty $ which contradicts the obvious fact that the values $p([h^n,a])$ are uniformly bounded.
\end{proof}

Lemma \ref{pc-to-ls} can be used to construct somewhat surprising group actions on quasi-lines. Namely we say that an action of a group $G$ on a hyperbolic space is \emph{purely loxodromic} if every element of $G$ of infinite order acts loxodromically.

\begin{cor}\label{plox}
Every hyperbolic group without infinite dihedral subgroups admits a purely loxodromic action on a quasi-line.
\end{cor}

\begin{proof}
Let $G$ be a hyperbolic group without infinite dihedral subgroups, $X$ a finite generating set of $G$. Let $Q$ be the set of all pseudocharacters on $G$ of defect at most $1$ satisfying $|q(x)|\le 1$ for all $x\in X$ endowed with the topology of pointwise convergence. For every $q\in Q$ and every $g\in G$, we have $|q(g)|\le 2|g|_X-1$. Thus $Q$ can be naturally identified with a subset of the set $P=\prod_{g\in G} [-2|g|_X+1, 2|g|_X-1]$. Since $Q$ is closed in $P$ and $P$ is compact by the Tychonoff theorem, $Q$ is compact.

Since $G$ contains no copies of $D_\infty$, for every infinite order element $g\in G$, there exists a pseudocharacter $q_g\colon G\to \mathbb R$ such that $q_g(g)\ne 0$; this result can be extracted from \cite{EF} and is also a particular case of \cite[Example 1.6]{HO}. Rescaling $q_g$ if necessary, we can assume that $q_g\in Q$. Obviously for every $p\in Q$, $(1-\alpha) p+\alpha q_g\in Q$ converges to $p$ in $Q$ as  $\alpha \to 0$. Therefore, the set $Q_g=\{ q\in Q\mid q(g)\ne 0\}$ is dense in $Q$. Note also that $Q_g$ is (obviously) open. By the Baire category theorem, the intersection of $Q_g$ for all infinite order elements $g\in G$ is non-empty and every pseudocharacter from this intersection gives a required action by Lemma \ref{pc-to-ls}.
\end{proof}

\begin{rem}
It is not difficult to show that no non-elementary hyperbolic group $G$ containing a copy of $D_\infty$ has a purely loxodromic action on a quasi-line. We leave the proof as an exercise for the reader. (Hint: show that an involution $a$ of $D_\infty$ must permute the limit points of $G$ for every lineal action of $G$ on a hyperbolic space; if the action is purely loxodromic, this implies that $ata=t^{-1}$ for every element of infinite order in $G$, which contradicts the assumption that $G$ is hyperbolic and non-elementary.)
\end{rem}

To every action of a group on a hyperbolic space fixing a point on the boundary, one can associate the so-called \emph{Busemann pseudocharacter}. We briefly recall the construction and necessary properties here and refer to \cite[Sec. 7.5.D]{Gro} and \cite[Sec. 4.1]{Man} for more details.

\begin{defn}\label{Bpc}
Let $G$ be a group acting on a hyperbolic space $S$ and fixing a point $\xi\in \partial S$.
Fix any $s\in S$ and let ${\bf x}=(x_i)$ be any sequence of points of $S$ converging to $\xi$. Then  the function $q_{\bf x}\colon G\to \mathbb R$ defined by
$$
q_{\bf x}(g)=\limsup\limits_{n\to \infty}(\d_S (gs, x_n)-\d_S(s, x_n))
$$
is a quasi-character. Its homogenization $p_{\bf x}$ is called the \emph{Busemann pseudocharacter}. It is known that for any two sequences ${\bf x}=(x_i)$ and ${\bf y}=(y_i)$ converging to $\xi$, we have $\sup _{g\in G} |q_{\bf x}(g)-q_{\bf y}(g)|<\infty $  \cite[Lemma 4.6]{Man}; in particular, this implies that $p_{\bf x}=p_{\bf y}$ and thus we can drop the subscript in $p_{\bf x}$. It is straightforward to verify that $g\in G$ acts loxodromically on $S$ if and only if $p(g)\ne 0$; in particular, $p$ is non-zero whenever $G \curvearrowright S$ is orientable lineal or quasi-parabolic.
\end{defn}

\begin{rem}\label{rempq}
By \cite[Lemma 4.5]{Man}, $|q_{\bf x} (g) - (\d_S (gs, x_n)-\d_S(s, x_n))|$ is uniformly bounded when $g$ ranges in $G$ and hence so is $|p (g) - (\d_S (gs, x_n)-\d_S(s, x_n))|$ (see (\ref{pg-qg})).
\end{rem}

\begin{lem}\label{product}
Let $G=A\times B$ for some groups $A$, $B$. Suppose that $G$ acts coboundedly on a hyperbolic space. Then the action is lineal or the induced action of at least one of the subgroups $A$,$B$ is elliptic.
\end{lem}

\begin{proof}
Suppose that both $A$ and $B$ act non-elliptically. Since the action of $G$ is cobounded it cannot be parabolic. Hence there exists a loxodromic element in $G$, which implies that at least one of $A$, $B$ contains a loxodromic element. Let $a\in A$ be loxodromic. Since $B$ commutes with $\langle a\rangle $ it must fix $a^+$ and $a^-$. In particular, the action of $B$ cannot be parabolic and hence $B$ also contains a loxodromic element $b$, whose limit points are necessarily $a^\pm $. Now using the fact that $A$ commutes with $\langle b\rangle$ we conclude that both $A$ and $B$ fix $a^+$ and $a^-$. Therefore, so does $G$.
\end{proof}

Recall that a subgroup $H$ of a group $G$ is called \emph{commensurated} if $H\cap gHg^{-1}$ has finite index in both $H$ and $gHg^{-1}$ for all $g\in G$.

\begin{lem}\label{pqp}
Let $G$ be a group acting on a hyperbolic space $S$ and let $H$ be a commensurated subgroup of $G$. Suppose that the action of $H$ is parabolic. Then the action of $G$ is parabolic or quasi-parabolic. In particular, if the action of $G$ is cobounded, then it is quasi-parabolic.
\end{lem}

\begin{proof}
Let $\xi $ be the limit point of $H$ on $\partial S$. Let $g\in G$ and $I=H\cap gHg^{-1}$. Since $I$ has finite index in $H$ its action on $S$ is also parabolic with $\Lambda (I)= \Lambda (H)=\{ \xi\}$. Applying the same argument to $gHg^{-1}$ we conclude that $\Lambda (I)=\Lambda (gHg^{-1})=\{g\xi\}$. Thus $g\xi=\xi $ for all $g\in G$, i.e., the action of $G$ is parabolic or quasi-parabolic.
\end{proof}

The following theorem completely describes possible types of $\H_\ell (G)$ and $\H_\ell^+ (G)$.

\begin{thm}\label{thm-lineal}
\begin{enumerate}
\item[(a)] For every group $G$, $\H_\ell (G)$ (and hence $\H_\ell ^+(G)$) is an antichain.
\item[(b)] For every group $G$, the cardinality of $\H _\ell^+(G)$ is $0$, $1$, or at least continuum.
\item[(c)] For every cardinal number $\kappa $ there exists a group $G_\kappa$ generated by a subset of cardinality at most  $\kappa$ such that $$|\H_\ell (G_\kappa)|=\kappa\;\; {\rm and}\;\; \H^+_\ell(G_\kappa)=\H_{qp}(G_\kappa)=\H_{gt}(G_\kappa)=\emptyset.$$ In addition, if $\kappa \in \mathbb N$, then $G_\kappa $ is finitely generated.
\end{enumerate}
\end{thm}

\begin{proof}
(a) This follows immediately from Corollary \ref{linmin}.

(b) We denote by $\mathcal P(G)$ the linear vector space of pseudocharacters on $G$. We consider three cases.

{\it Case 1.} ${\rm dim\,} \mathcal P (G)=0$. In this case $\H_\ell ^+(G)=\emptyset $ as otherwise there would exist a non-zero Busemann quasi-character on $G$.

{\it Case 2.} ${\rm dim\,} \mathcal P (G)=1$. Let $[X]$, $[Y]$ be any elements of $\H ^+_\ell (G)$ and let $p$, $q$ be Busemann pseudocharacters associated to the actions of $G$ on the corresponding Cayley graphs. Taking $s=1$ in the definition of both pseudocharacters, we conclude that $| p(x)| \le 1$ for all $x\in X$. Since $p$ and $q$ are linearly dependent, $q$ is uniformly bounded on $X$. It follows that $|\d_Y (x, x_n)-\d_Y(1, x_n)|$ is uniformly bounded on $X$, where $(x_n)\subseteq G$ is the fixed sequence used to construct $q$. Since $\Gamma (G,Y)$ is a quasi-line, it follows that $\d_Y (x,1)$ is uniformly bounded, i.e., $Y\preceq X$. Similarly we prove that $X\preceq Y$. Thus $[X]=[Y]$.

{\it Case 3.} ${\rm dim\,} \mathcal P (G)\ge 2$. Let $p, q$ be two linearly independent pseudocharacters of defect at most $D$ on $G$. Let $x$, $y$ be elements of $G$ such that the vectors $u=(p(x), p(y))$ and $v=(q(x), q(y))$ are linearly independent. Rescaling $p$ and $q$ if necessary, we can assume that $\|u\|=\|v\|=1$. For $\lambda \in [0,1]$, let $$r_\lambda = \lambda p + (1-\lambda)q.$$ Clearly $r_\lambda $ is a pseudocharacter of defect at most $D$.

Given real numbers $a,b, c$, we write $a\approx _c b$ if $|a-b|\le c$. It is straightforward to check that for any $m,n\in \mathbb Z$, we have
\begin{equation}\label{wxmyn}
r_\lambda(x^my^n) \approx _D  m(\lambda p(x)+(1-\lambda)q(x)) + n(\lambda p(y)+(1-\lambda)q(y)) = {\bf proj}_{w_\lambda}\left(
                                                \begin{matrix}
                                                  m \\
                                                  n \\
                                                \end{matrix}
                                              \right),
\end{equation}
where $w_\lambda = \lambda u+(1-\lambda)v$. Since $u$ and $v$ are linearly independent, $w_\lambda$ and $w_\mu$ are not collinear whenever $\lambda \ne \mu$. Together with (\ref{wxmyn}) this implies that for every $\lambda \ne \mu$, there is a sequence $(z_i)$ of elements of $G$ of the form $z_i=x^{m_i}y^{n_i}$ such that $|r_\lambda (z_i)|$ is bounded by $D+1$,  while $r_\mu (z_i) \to \infty $ as $i\to \infty$. Applying Lemma \ref{pc-to-ls} to pseudocharacters $r_\lambda$, $\lambda \in [0,1]$, and $C> D+1$ we obtain a continuum of pairwise non-equivalent lineal hyperbolic structures on $G$.

(c) Let $D_\infty^\kappa$ denote the direct sum of $\kappa$ copies of the infinite dihedral group. We enumerate copies of $D_\infty $ using indices from an index set $I$ of cardinality $\kappa$. Let $t_i, a_i$, $i\in I$, be generators of copies of $D_\infty$, where $|t_i|=\infty$, $|a_i|=2$.  Let $[X]\in \H_\ell (D_\infty^\kappa)$ and let $S=\Gamma (G,X)$. Then each $t_i$ is either elliptic or loxodromic with respect to the action on $S$ (this dichotomy holds for elements of any group acting on a quasi-line, see for example \cite[Corollary 3.6 ]{Man06}). Suppose that $t_i$, $t_j$ are loxodromic for some $i\ne j$. Then the relation $a_it_i=t_i^{-1}a_i$ forces $a_i$ to permute the boundary points of $S$ while $a_it_j=t_ja_i$ implies that $a_i$ fixes $\partial S$ pointwise. This contradiction shows that exactly one element $t_i$ is loxodromic. From now on, we fix this index $i$.

Let $C_i\le D_\infty^\kappa$ be the direct sum of all copies of $D_\infty$ except the $i$th copy. Then $C_i$ contains no loxodromic elements. Since $C_i\lhd D_\infty^\kappa$, by Lemma \ref{pqp} the action of $C_i$ on $S$ cannot be parabolic, hence it is elliptic. Thus $X\sim \{t_i, a_i\}\cup C_i$. It is clear that such structures are not equivalent for different indices $i$. Thus $\H_\ell (D_\infty ^\kappa)= \kappa$.

Finally we notice that every action of $D_\infty ^\kappa$ on a hyperbolic space must be lineal or elliptic by Lemma \ref{product} (and transfinite induction). Thus we get a complete description of $\H(D_\infty ^\kappa)$: it is obtained from the antichain of cardinality $\kappa$ by adding one element, which is smaller than every element of the antichain.
\end{proof}

\begin{ex} \label{ex-Zn} Let $G=\mathbb Z \times \mathbb Z$. Then $\H (G)=\H _\ell (G)\cup \H_e(G)$ by Lemma \ref{product}; notice also that $\H_\ell^+(G)=\H_\ell (G)$ for every abelian group. Clearly $G$ has at least $2$ non-equivalent lineal hyperbolic structures. Therefore, $\H_\ell (\mathbb Z \times \mathbb Z)$ is an antichain of cardinality continuum. The poset $\H (\mathbb Z \times \mathbb Z)$ is obtained from this antichain by adding one element (corresponding to the elliptic structure), which is smaller than every element of the antichain. The same argument works for $G=\mathbb Z^n$ for every $n \in \mathbb N, n \geq 2$. For $n=1$, $\H(G)$ is the chain of length $2$ (this follows from the proof of Lemma \ref{vc}).
\end{ex}

\subsection{Examples of quasi-parabolic structures}

The main goal of this section is to write down a couple of useful observations about quasi-parabolic structures and discuss examples of quasi-parabolic structures on $\mathbb Z\, {\rm wr}\, \mathbb Z$.

\begin{lem}\label{qpstr}
Let $G\curvearrowright S$ be a quasi-parabolic action of a group $G$ on a hyperbolic space. Then there exists $[X]\in \H _{qp}(G)$ on $G$ such that $G\curvearrowright S\sim _w G\curvearrowright \Gamma (G, X)$.
\end{lem}

\begin{proof}
This is a combination of the fact that quasi-parabolic actions have quasi-convex orbits (see part 4 (a) of Theorem \ref{ClassHypAct}) and Proposition \ref{cobdd}.
\end{proof}

\begin{ex}\label{BS}
Let us consider the Baumslag-Solitar group $$G=\langle a,b\mid bab^{-1}=a^2\rangle. $$ It is well-known and easy to prove that $G$ is isomorphic to the subgroup of $SL_2(\mathbb R)$ generated by $$b=\left(
                                                 \begin{array}{cc}
                                                   \sqrt{2} & 0 \\
                                                   0 & 1/\sqrt{2} \\
                                                 \end{array}
                                               \right)$$ and $$a=\left(
                                                                   \begin{array}{cc}
                                                                     1 & 1 \\
                                                                     0 & 1 \\
                                                                   \end{array}
                                                                 \right).$$
Thus we obtain an action of $G$ on $\mathbb H^2$ that factors through the action of $SL_2(\mathbb R)$. It is easy to see that this action is quasi-parabolic. Indeed, $G$ contains loxodromic elements (e.g., $b$) and parabolic elements (e.g., $a$), hence the action cannot be elementary by Theorem \ref{ClassHypAct}. The action cannot be of general type either since $G$ is solvable while every group admitting a general type action on a hyperbolic space contains a non-cyclic free subgroup by the standard ping-pong argument, see \cite{Gro}. Hence the action is quasi-parabolic.

Although the action of $G$ on $\mathbb H^2$ is not cobounded, Lemma \ref{qpstr} provides us with a quasi-parabolic structure on $G$. Yet another element of $\H_{qp}(G)$ can be obtained using the action of $G$ on the Bass-Serre tree associated to the HNN-extension structure; in this case the action is cobounded so we do not need Lemma \ref{qpstr} and can simply use the Svarc-Milnor map. It is not hard to show that these two quasi-parabolic structures on $G$ are not equivalent. Indeed there are no non-trivial elliptic elements with respect to the first structure, while all elements from $[G, G]$ are elliptic with respect to the second structure.
\end{ex}

The next observation implies that a quasi-parabolic structure is never minimal. This should be compared to Theorem \ref{lineal}, which states that a lineal structure is always minimal, and Theorem \ref{n-gt-intr}, which implies that a general type structure \emph{can} be minimal.

\begin{cor}\label{qp-to-l}
For any $A\in \H _{qp}(G)$, there exists $B\in \H _{\ell}^+(G)$ such that $B\preccurlyeq A$.  In particular, if $\H _{qp}(G)\ne \emptyset$, then $\H^+_\ell (G)\ne \emptyset$.
\end{cor}

\begin{proof}
Let $p$ be the Busemann pseudocharacter associated to the action of $A$ (see Definition \ref{Bpc}), and let $B$ be the orientable lineal structure corresponding to $p$ (see Lemma \ref{pc-to-ls}). It easily follows from the definition of the Busemann pseudocharacter and Lemma \ref{pc-to-ls} that $B\preccurlyeq A$.
\end{proof}

Recall that $\mathcal D$ denotes the set of natural numbers ordered according to divisibility: $m\preccurlyeq n$ if $m\mid n$. The main result of this section is the following proposition. Its proof essentially uses the Bass-Serre theory of groups acting on trees; for details we refer to \cite{Serre}.

\begin{prop}\label{ZwrZ}
Let $G=\mathbb Z\, {\rm wr}\, \mathbb Z$.
\begin{enumerate}
\item[(a)] $\H_{qp}(G)$ contains an isomorphic copy of $\mathcal D$.
\item[(b)] $\H_{qp}(G)$ contains an antichain of cardinality continuum.
\end{enumerate}
\end{prop}

\begin{proof}
(a) It is well-known and easy to prove that the group $G$ has the presentation
$$
G=\langle s, \{ a_i\}_{i\in \mathbb Z} \mid sa_is^{-1}=a_{i+1},\; [a_i, a_j]=1,\;  i,j\in \mathbb Z\rangle .
$$
For every $n\in \mathbb N$, there is a natural homomorphism $G\to G_n=\mathbb Z_n\, {\rm wr}\, \mathbb Z$, where
$$
G_n=\langle t, \{ b_i\}_{i\in \mathbb Z} \mid tb_it^{-1}=b_{i+1},\; [b_i, b_j]=b_i^n=1,\;  i,j\in \mathbb Z\rangle ,
$$
sending $s$ to $t$ and $a_i$ to $b_i$ for all $i$. The group $G_n$ can be thought of as an ascending HNN-extension of $B_n=\langle b_0, b_1, \ldots\rangle $ associated to the endomorphism sending $b_{i-1}$ to $b_i$ for all $i\in \mathbb N$. Let $T_n$ be the associated Bass-Serre tree. Since the HNN-extension is ascending, $G_n$ fixes an end of $T_n$ and hence the action $G_n \curvearrowright T_n$ is quasi-parabolic. Therefore so is the action of $G$ on $T_n$ that factors through $G_n\curvearrowright T_n$.

Let $[X_n]=\sigma ([G\curvearrowright T_n])$. We have $[X_n]\in \H _{qp} (G)$ by (\ref{SM-altdef}) and Lemma \ref{we-type}. We claim that the map $f\colon n\mapsto [X_n]$ defines an order preserving embedding $\mathcal D\to \H_{qp} (G)$.

To prove that $f$ is order preserving, it suffices to show that for every $m,n \in \mathbb N$ such that $n\mid m$, there is a $G$-equivariant Lipschitz map $\lambda $ from the vertex set $V(T_m)=G_m/B_m$ to the vertex set $V(T_n)=G_n/B_n$. Indeed then $G\curvearrowright T_n\preceq G\curvearrowright T_m$ by Definition \ref{def-poset} and passing to equivalence classes of actions and using Lemma \ref{e-we} and Proposition \ref{prop-SM} we obtain the required inequality. We define $\lambda$ in the obvious way by sending a vertex $gB_m$ of $T_m$ to the vertex $\e(g)B_n$, where $\e$ is the natural homomorphism $G_m\to G_n$; of course, it only exists if $n \mid m$. It follows immediately from the construction of the trees $T_m$ and $T_n$ that if two vertices $u$ and $v$ of $T_m$ are connected by an edge in $T_m$, then $\lambda (u)$ and $\lambda (v)$ are connected by an edge in $T_n$. Thus $\lambda $ is Lipschitz. This finishes the proof of the fact that $f$ is order preserving.

It remains to show that $f$ is injective. Let $m<n$ be two natural numbers and let $M= Ker (G\to G_m)$. If the action of $M$ on $T_n$ has bounded orbits, then $M$ must fix a vertex of $T_n$ (this is well-known for every group acting elliptically on a tree); since $M\lhd G$ and the action of $G$ on vertices of $T_n$ is transitive, this implies that $M$ fixes all vertices of $T_n$. In particular, $M\le B_n$, which is obviously not the case. This contradiction shows that $M$ has unbounded orbits on $T_n$. On the other hand, $M$ acts trivially on $T_m$. This easily implies that $f(m)=\sigma ([G\curvearrowright T_m])\not\sim \sigma ([G\curvearrowright T_n])=f(n)$. Thus, $f$ is injective.

(b) Note that for every $\xi> 0$, the mappings
$$ s\mapsto \left(
\begin{array}{cc}
\sqrt{\xi} & 0 \\
0 & 1/\sqrt{\xi} \\
\end{array}
\right)$$
and
$$a_i\to \left(
\begin{array}{cc}
1 & \xi ^i\\
0 & 1 \\
\end{array}
\right)$$ extend to a homomorphism $\phi_\xi\colon G\to SL_2(\mathbb R)$ (in fact, this is an embedding if and only if $\xi$ is transcendental, but we will not use this). Arguing as in Example \ref{BS}, it is easy to verify that the action $A_\xi$ of $G$ that factors through the standard action of $SL_2(\mathbb R)$ on $\mathbb H^2$ is quasi-parabolic whenever $\xi \ne  1$. Applying Lemma  \ref{qpstr}, we obtain a quasi-parabolic structure $[X_\xi] \in \H_{qp}(G)$ such that $G\curvearrowright \Gamma (G, X_\xi)\sim_w A_\xi$.

We identify $\mathbb H^2$ with the upper half-plane model. To distinguish between the actions of $G$ we use the notation $\phi_\xi(g)z$ to denote the image of $z\in \mathbb H^2$ under the action of $g\in G$ with respect to $A_\xi$. Let $$g_n=a_1^{\alpha _n}a_0^{\beta_n}$$ for some $\alpha_n, \beta_n\in \mathbb Z$. Then
$$
\phi_\xi (g_n)= \left(
\begin{array}{cc}
1 & \alpha_n\xi +\beta_n\\
0 & 1 \\
\end{array}
\right)
$$
and therefore we have
\begin{equation}\label{phi1}
\phi_\xi (g_n)z= z+\alpha_n\xi +\beta_n.
\end{equation}

Let $\xi\ne \eta$ be two positive numbers not equal to $1$. We want to show that $[X_\xi]$ and $[X_\eta]$ are incomparable. Arguing by contradiction, assume that $X_\eta\preceq X_\xi$. By Proposition \ref{prop-SM}, we have $A_\eta \preceq A_\xi$, i.e., there exists a constant $C$ such that
\begin{equation}\label{phi2}
\d_{\mathbb H^2}(z, \phi_\eta(g)z)\le C \d_{\mathbb H^2}(z, \phi_\xi(g)z) +C
\end{equation}
for all $g\in G$.

Let us choose sequences of integers $(\alpha_n)$, $(\beta_n)$ such that $\lim_{n\to \infty} \alpha _n= \infty$ and
\begin{equation}\label{abn1}
|\alpha _n \xi + \beta _n|\le 1
\end{equation}
Since $\xi\ne \eta $, it follows that
\begin{equation}\label{abn2}
\lim_{n\to \infty} |\alpha _n \eta + \beta _n|=\infty.
\end{equation}
Using (\ref{phi1}) and (\ref{abn1}), we obtain that $\d_{\mathbb H^2} (\phi_\xi(g_n)z,z)= \d_{\mathbb H^2} (z+\alpha_n\xi +\beta_n, z)$ is uniformly bounded for all $n\in \mathbb N$. Replacing $\xi$ with $\eta$ and arguing in the same way, we obtain that $\d_{\mathbb H^2} (\phi_\eta(g_n)z, z)\to \infty $ as $n\to \infty$ by (\ref{abn2}). Clearly this contradicts (\ref{phi2}). This contradiction shows that $A_\xi$ and $A_\eta$ are incomparable for any positive $\eta\ne \xi$ and thus $\{ [A_\xi] \mid \xi\in (0,1)\cup (1, +\infty)\}$ is an antichain in $\H _{qp}(G)$ of cardinality continuum.
\end{proof}

\subsection{Groups with finitely many hyperbolic structures of general type}

Our main goal here is to prove the following.

\begin{thm}\label{n-gt}
For every $n\in \mathbb N$, there exists a finitely generated group $G_n$ such that $\H_\ell(G_n)=\H_{qp}(G_n)=\emptyset$ and $\H_{gt}(G_n)$ is an antichain of cardinality $n$. In particular, we have $|\H(G_n)|=n+1$ and the poset $\H (G_n)$ has the structure described in Fig. \ref{hasse}.
\end{thm}

\begin{proof}
Let $T$ be an $n$-regular tree with $n\ge 3$ and let $G$ be a finitely generated group of automorphisms of $T$ such that the following conditions hold:

\begin{enumerate}
\item[(a)] $G$ is dense in $Aut(T)$ with respect to the topology of pointwise convergence;
\item[(b)] for every vertex $v$ of $T$, $Stab_G(v)$ is a commensurated torsion subgroup of $G$.
\end{enumerate}

There do exist examples of groups satisfying these properties. For instance, for any simply transitive subgroup $F\le S_n$, the group $G=G(F)\le Aut(T)$ considered in \cite{lB} is finitely generated by \cite[Corollary 3.8]{lB} and satisfies (a) by \cite[Proposition 3.5]{lB}. In addition, for every $v\in V(T)$, the subgroup $Stab_G(v)$ is locally finite (see the last two paragraphs of Section 3.1 in \cite{lB}). Since $T$ is locally finite, $Stab_G(v) \cap gStab_G(v)g^{-1}= Stab_G(v) \cap Stab_G(gv)$ has finite index in both $Stab_G(v)$ and $Stab_G(gv)$ for all $g\in G$. Thus (b) is also satisfied.

Since $T$ is regular, every map between two pairs of equidistant vertices of $T$ extends to an automorphism of $T$. By (a) the action of $G$ on pairs of equidistant vertices of $T$ is coarsely transitive (in particular, the action of $G$ on $T$ is cocompact) and therefore  $A=\sigma ([G\curvearrowright T])$ is a minimal hyperbolic structure on $G$ by Proposition \ref{minact}. Note that (a) also implies that $A\in \H_{gt}(G)$. Indeed let $a,b,c$ be any triple of vertices of $T$ such that $\d_T(a,b)=\d_T(b,c)=\d_T (a,c)/2$. By (a) there is $g\in G$ such that $ga=b$ and $gb=c$. Then $g$ acts loxodromically on $T$ and preserves a geodesic in $T$ passing through $a,b,c$. Obviously this construction yields independent loxodromic elements of $G$.

For any $[X]\in \H_\ell (G)\cup \H_{gt}(G)$, the induced action of any vertex stabilizer $Stab_G(v)$ on $\Gamma (G,X)$ cannot be lineal, quasi-parabolic, or of general type as $Stab_G(v)$ has no elements of infinite order. Neither can it be parabolic by (b) and Lemma \ref{pqp}. Thus the induced action of $Stab_G(v)$ on $\Gamma (G,X)$ is elliptic for every vertex $v$ and we can apply Corollary \ref{maxact-cor}, which implies that $A$ is the largest element of $\H_\ell (G)\cup \H_{gt}(G)$. Being largest and minimal, $A$ must be the only element of $\H_\ell (G)\cup \H_{gt}(G)$. In particular, $G$ has no lineal structures and hence it has no quasi-parabolic structures by Corollary \ref{qp-to-l}. Thus $A$ is the only non-trivial element of $\HG$.

Let $G_n=G^n$. If $G$ acts coboundedly on a hyperbolic space and the induced action of two distinct multiples in $G^n$ is non-elliptic, then the action of $G_n$ must be lineal by Lemma \ref{product}, and we get a contradiction again since $G$ has no lineal hyperbolic structures. Thus the only hyperbolic structures on $G^n$ are those for which the corresponding actions factor through the actions of one of the multiples. Clearly the equivalence classes of these actions form an antichain of cardinality $n$.
\end{proof}


\section{Induced hyperbolic structures}


\subsection{Strongly hyperbolically embedded subgroups}

In this section, we introduce the notion of a strongly hyperbolically embedded collection of subgroups.  This is a strengthening of the notion of a hyperbolically embedded collection of subgroups introduced in \cite{DGO}. Our main result is Proposition \ref{lem:Hacyl}, which provides a rich source of examples; it can be thought of as a generalization of the fact that the action of a relatively hyperbolic group on the corresponding relative Cayley graph is acylindrical (see \cite[Proposition 5.2]{Osi16}). In Section 5.3 we will show that the induced structure map behaves especially well for strongly hyperbolically embedded subgroups; this result will have numerous applications in Sections 5.4 and 6.1.

We begin by recalling basic definitions and results from \cite{DGO}.

Suppose that we have a group $G$, a collection of subgroups $\Hl$ of $G$, and a subset $X\subseteq G$ such that $X$ together with the union of all $H_i$ generate $G$. Let
\begin{equation}\label{calH}
\mathcal H=H_1\sqcup H_2\sqcup \ldots \sqcup H_n.
\end{equation}
We think of $X$ and $\mathcal H$ as abstract sets and consider the alphabet
\begin{equation}\label{calA}
\mathcal A= X\sqcup \mathcal H
\end{equation}
together with the map $\mathcal A\to G$ induced by the obvious maps $X\to G$ and $H_i \to G$. By abuse of notation, we do not distinguish between subsets $X$ and $H_i$ of $G$ and their preimages in $\mathcal A$. Note, however, the map $\mathcal A\to G$ is not necessarily injective. Indeed if $X$ and a subgroup $H_i$ (respectively, subgroups $H_i $ and $H_j$ for some $i\ne j$) intersect in $G$, then every element of $H_i \cap X\subseteq G$ (respectively, $H_i \cap H_j$) will have at least two preimages in $\mathcal A$: one in $X$ and another in $H_i$ (respectively, one in $H_i $ and one in $H_j$) since we use disjoint unions in (\ref{calH}) and (\ref{calA}).

In these settings, we consider the Cayley graphs $\Gamma(G,X\sqcup \mathcal H)$ and $\Gamma (H_i, H_i)$, and we naturally think of the latter as subgraphs of the former. For every $i\in \{ 1, \ldots, n\}$, we introduce a \textit{relative metric} $\dl_i \colon H_i \times H_i \to [0, +\infty]$ as follows: we say that a path $p$ in $\Gamma(G,X\sqcup\mathcal H)$ is \emph{admissible} if it contains no edges of $\Gamma (H_i, H_i)$. Then $\dl_i (h,k)$ is defined to be the length of a shortest admissible path in $\Gamma(G,X\sqcup\mathcal H) $ that connects $h$ to $k$. If no such a path exists, we set $\dl_i (h,k)=\infty $. Clearly $\dl_i $ satisfies the triangle inequality, where addition is extended to $[0, +\infty]$ in the natural way.

It is convenient to extend the relative metric $\dl _i$ to the whole group $G$ by assuming
$$
\dl_i (f,g)\colon =\left\{\begin{array}{ll}\dl_i (f^{-1}g,1),& {\rm if}\;  f^{-1}g\in H_i \\ \dl_i (f,g)=\infty , &{\rm  otherwise.}\end{array}\right.
$$
If the collection $\Hl$ consists of a single subgroup $H$, we use the notation $\widehat \d$ instead of $\dl_i$.

\begin{defn}\label{defhe}
A collection of subgroups $\Hl$ of $G$ is \emph{hyperbolically embedded  in $G$ with respect to a subset $X\subseteq G$}, denoted $\Hl \h (G,X)$, if the following conditions hold.
\begin{enumerate}
\item[(a)] The group $G$ is generated by $X$ together with the union of all $H_i$ and the Cayley graph $\Gamma(G,X\sqcup\mathcal H) $ is hyperbolic.
\item[(b)] For every $i$, the metric space $(H_i,\dl_i)$ is proper, i.e., every ball (of finite radius) in $H_i$ with respect to the metric $\dl_i$ contains finitely many elements.
\end{enumerate}
If, in addition, the action of $G$ on $\Gamma(G,X\sqcup\mathcal H) $ is acylindrical, we say that $\Hi$ is \emph{strongly hyperbolically embedded in $G$ with respect to $X$.}

Finally, we say  that the collection of subgroups $\Hl$ is \emph{hyperbolically embedded} in $G$ and write $\Hl\h G$ if $\Hl\h (G,X)$ for some $X\subseteq G$.
\end{defn}

\begin{rem}
Unlike the notion of a hyperbolically embedded subgroup, the notion of a strongly hyperbolically embedded subgroup depends on the choice of a generating set. In general, $\Hi\h (G,X)$ does not imply that $\Hi$ is strongly hyperbolically embedded in $G$ with respect to $X$, but does imply that $\Hi$ is strongly hyperbolically embedded in $G$ with respect to some other relative generating set $Y$ containing $X$, see \cite[Theorem 5.4]{Osi16} for details.
\end{rem}

For any group $G$ we have $G\h G$. Indeed we can take $X=\emptyset $. Then the Cayley graph $\Gamma(G, X\sqcup H)$ has diameter $1$ and the corresponding relative metric satisfies $\widehat d(h_1, h_2)=\infty $ whenever $h_1\ne h_2$. Further, if $H$ is a finite subgroup of a group $G$, then $H\h G$. Indeed $H\h (G,X)$ for $X=G$.

Since hyperbolically embedded subgroups and the metric $\dl$ introduced above play a crucial role in this paper, we consider two additional examples borrowed from \cite{DGO}.

\begin{figure}
\centering
\begin{minipage}{.49\textwidth}
  \centering
  \def\svgscale{0.325}
  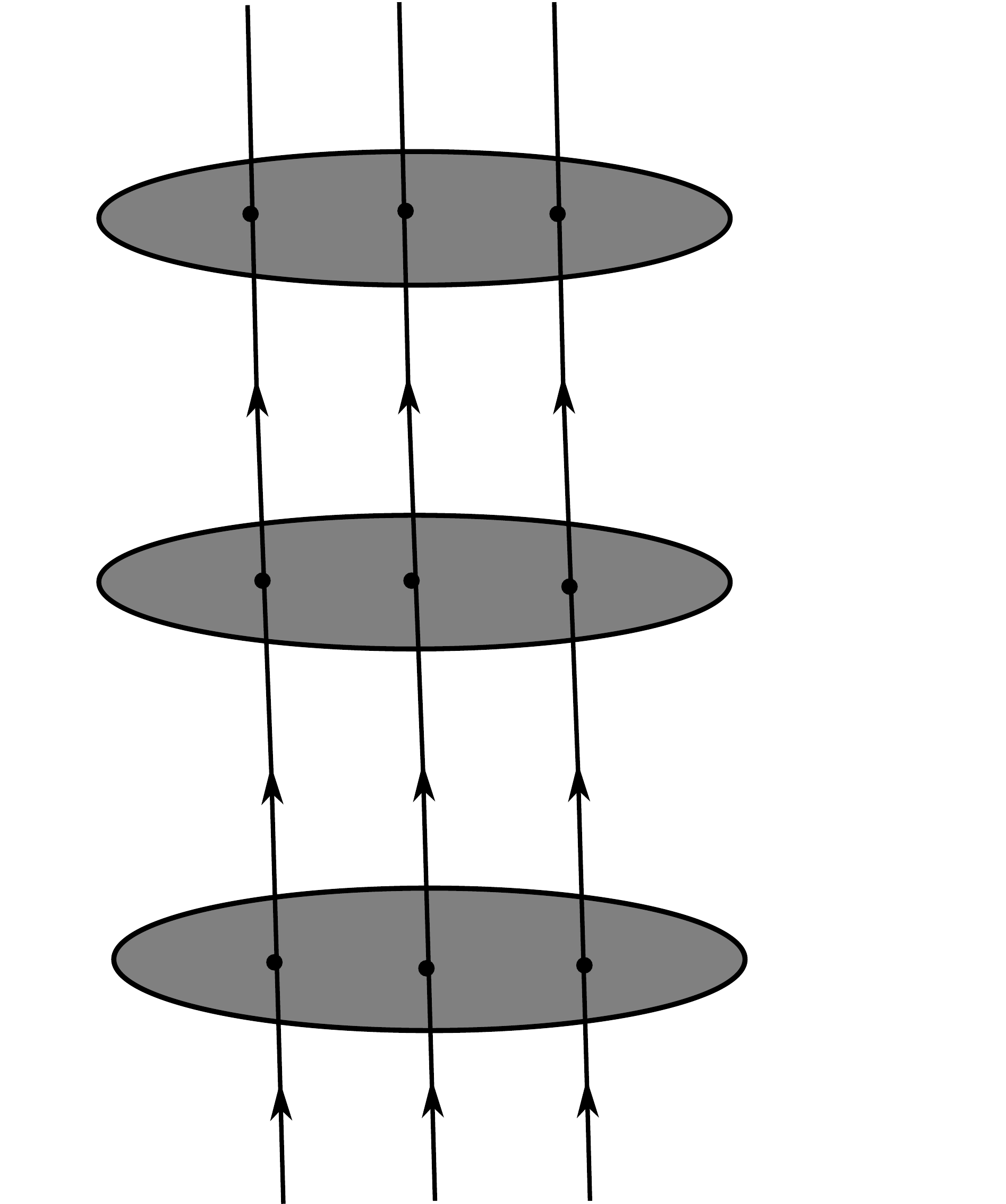
\caption{$H \times \mathbb{Z}$}
\label{1a}
\end{minipage}%
\begin{minipage}{.49\textwidth}
  \centering
  \def\svgscale{0.325}
  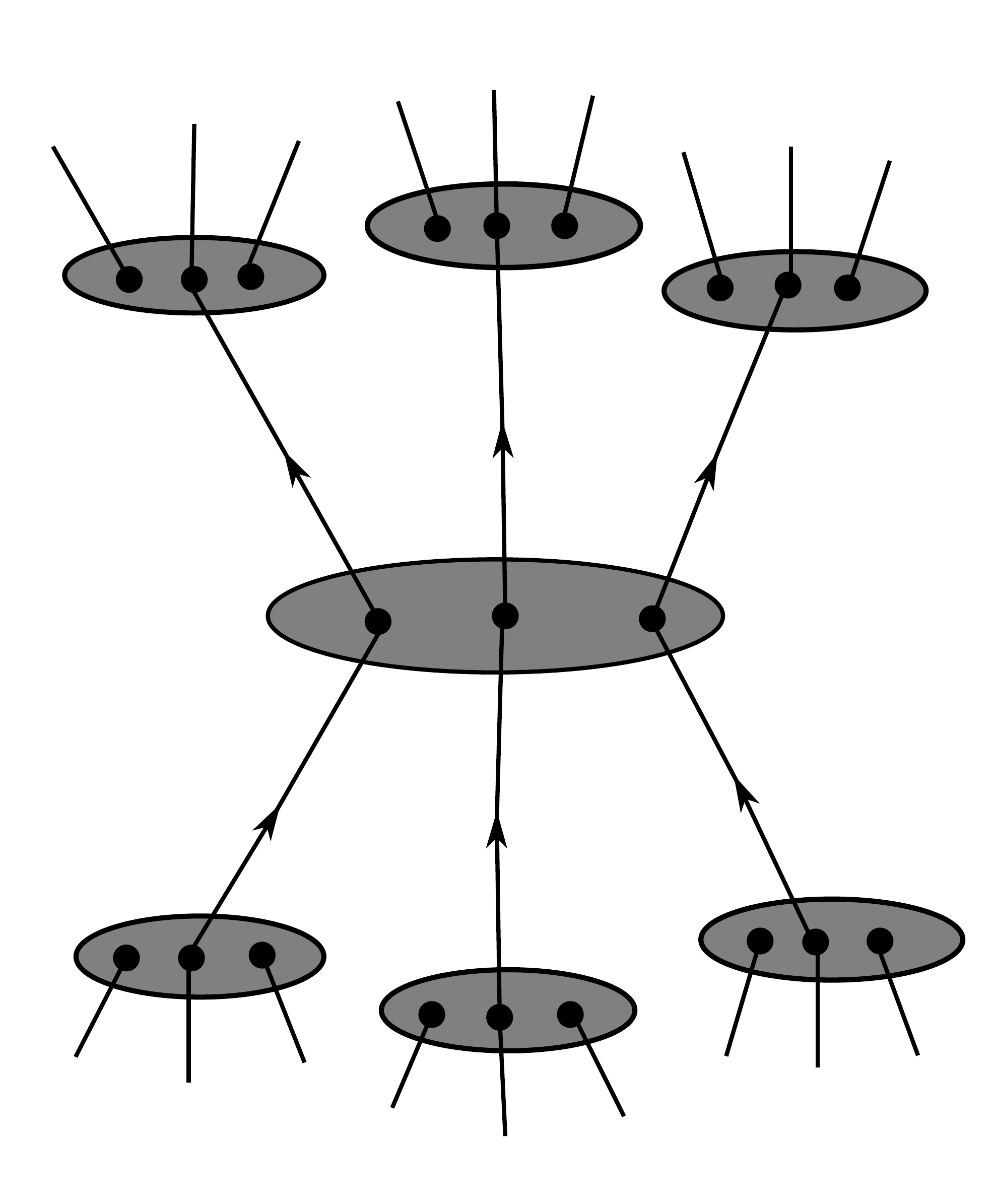
 \caption{$H * \mathbb{Z}$}
\label{1b}
\end{minipage}
\end{figure}

\begin{ex}
\begin{enumerate}
\item[(a)]
Let $G=H\times \mathbb{Z}$, $X=\{ x\} $, where $x$ is a generator of $\mathbb Z$. Then $\Gamma (G, X\sqcup H)$ is quasi-isometric to a line and hence it is hyperbolic. However the corresponding relative metric satisfies $\widehat\d(h_1, h_2)\le 3$ for every $h_1, h_2\in H$ (See Fig. \ref{1a}). Indeed let $\Gamma _H$ denote the Cayley graph $\Gamma (H, H ) $. In the shifted copy $x\Gamma _H$ of $\Gamma _H$ there is an edge (labeled by $h_1^{-1}h_2\in H$) connecting $h_1x$ to $h_2x$, so there is an admissible path of length $3$ connecting $h_1$ to $h_2$. Thus if $H$ is infinite, then $H\not\h (G,X)$.

\item[(b)]  Let $G=H\ast \mathbb Z$, $X=\{ x\} $, where $x$ is a generator of $\mathbb Z$. In this case $\Gamma (G, X\sqcup H)$ is quasi-isometric to a tree (see Fig. \ref{1b}) and $\widehat\d(h_1, h_2)=\infty $ unless $h_1=h_2$. Thus $H\h (G,X)$. In fact, $H$ is strongly hyperbolically embedded in $G$ in this case.
\end{enumerate}
\end{ex}

The following result proved in \cite{DGO} relates the notions of hyperbolically embedded collections of subgroups and relatively hyperbolic groups. (Readers unfamiliar with relative hyperbolicity can take this result as the definition of relatively hyperbolic groups.)

\begin{thm}\label{rhhe}
Let $G$ be a group, $\Hl$ a collection of subgroups of $G$. Then $\Hl\h(G,X)$ for a finite $X\subseteq G$ if and only if $G$ is hyperbolic relative to $\Hl$.
\end{thm}

We will make use of several technical notions first introduced in \cite{Osi06,ESBG} for relatively hyperbolic groups and then generalized in the context of hyperbolically embedded subgroups in \cite{DGO}.

\begin{defn}\label{comp}
Let $q$ be a path in the Cayley graph $\Gamma(G,X\sqcup\mathcal H) $. A (non-trivial) subpath $p$ of $q$ is called an \emph{$H_i$-subpath} if the label of $p$ is a word in the alphabet $H_i$. An $H_i$-subpath $p$ of $q$ is an {\it $H_i$-component} if $p$ is not contained in a longer $H_i$-subpath of $q$; if $q$ is a loop, we require in addition that $p$ is not contained in any longer $H_i$-subpath of a cyclic shift of $q$.

Two $H_i$-components $p_1, p_2$ of a path $q$ in $\Gamma(G,X\sqcup\mathcal H) $ are called {\it connected} if there exists a
path $c$ in $\Gamma(G,X\sqcup\mathcal H) $ that connects some vertex of $p_1$ to some vertex of $p_2$, and the label of $c$ is a word consisting only of letters from $H_i$. In algebraic terms this means that all vertices of $p_1$ and $p_2$ belong to the same left coset of $H_i$. Note also that we can always assume that $c$ is an edge as every element of $H_i$ is included in the set of generators. A component of a path $p$ is called \emph{isolated} in $p$ if it is not connected to any other component of $p$.
\end{defn}

The following result is a simplified version of \cite[Proposition 4.13]{DGO}. Given a path $p$ in a metric space, we denote by $p_-$ (respectively $p_+$) its initial (respectively, terminal) point.

\begin{lem}\label{C}
Suppose that $\Hl\h (G,X)$. Then there exists a constant $C$ such that for any $m$-gon $p$ with geodesic sides in $\Gamma(G,X\sqcup\mathcal H)$ and any isolated $H_i$-component $a$ of $p$, we have $\dl_i (a_-, a_+)\le Cm$.
\end{lem}

The next lemma is an immediate corollary of a particular case of \cite[Proposition 4.11(b)]{DGO}.  Note that we state it for a finite collection of subgroups, which allows us to find a uniform constant $B$.

\begin{lem} \label{lem:Z}  If $\Hi\hookrightarrow_h(G,X)$, then there exist a constant $B$ and finite subsets $Z_i\subset H_i$ such that $\d_{Z_i}(f,g)\le B\widehat\d_i (f,g)$ for all $i\in\{ 1, \ldots, n\}$ and all $f,g\in H_i$.
\end{lem}

Combining these lemmas we obtain the following corollary, which will be used in later sections. Note that we use word metrics on $G$ associated to arbitrary (not necessarily generating) subsets: given $Y\subseteq G$ and $g,h\in G$, we define $\d_Y(g,h)$ to be $|g^{-1}h|_Y$ if $g^{-1}h\in \langle Y\rangle$ and set $\d_Y(g,h)=\infty$ otherwise.

\begin{cor} \label{Yibound}
Let $\Hi\hookrightarrow_h(G,X)$ and for $i=1,\dots, n$, let $Y_i$ be a subset of $G$ such that $H_i$ is a subgroup of $\langle Y_i\rangle$. Then there exists a constant $D$ such that for any $m$-gon $p$ with geodesic sides in $\Gamma(G,X\sqcup\mathcal H)$ and any isolated $H_i$-component $a$ of $p$, we have
$$
\d_{Y_i} (a_-, a_+)\le Dm.
$$
\end{cor}

\begin{proof}
For each $i$, let $Z_i$ be the finite subset of $H_i$ provided by Lemma \ref{lem:Z} and let $M=\max_{i=1, \ldots, n}\max_{z\in Z_i}|z|_{Y_i}<\infty $. It suffices to take $D=BCM$, where the $B$ and $C$ are constants provided by Lemma \ref{C} and Lemma \ref{lem:Z}, respectively.
\end{proof}

The following proposition gives a sufficient condition for a hyperbolically embedded subgroup to be strongly hyperbolically embedded.

\begin{prop} \label{lem:Hacyl} Suppose a group $G$ is generated by a subset $X$ and $H\hookrightarrow_h(G,X)$.  If the action of $G$ on $\Gamma(G,X)$ is acylindrical, then $H$ is strongly hyperbolically embedded in $G$ with respect to $X$.
\end{prop}

Before proving Proposition \ref{lem:Hacyl}, we need the following lemma.

\begin{lem} \label{defineE}
Let $S$ be a $\delta$--hyperbolic space and $G$ a group acting by isometries on $S$.  For every $\e\geq 0$ and every pair of points $x,y\in S$, there exists a constant $E$ depending only on $\delta$ and $\e$ such that the following condition is satisfied: whenever $g\in G$ satisfies $\max\{\d_S(x,gx),\d_S(y,gy)\}\leq \e$, any point $z$ on a geodesic from $x$ to $y$ satisfies $$\d_S(z,gz)\leq E.$$
\end{lem}

\begin{proof}
Let us fix $\e\geq 0$, and let $x,y$ be any two points in $S$ and $\gamma$ a geodesic in $S$ from $x$ to $y$.  Suppose there exists a $g\in G$ such that $\d_S(x,gx)\leq \e$ and $\d_S(y,gy)\leq \e$.   Let $z\in S$ be any point that lies on $\gamma$, and let $\delta$ be the hyperbolicity constant of $S$.  Then there is a point $t$ on $g\gamma$ such that $\d_S(z,t)\leq 2\delta+\e$.  Without loss of generality, assume $t$ lies on the subpath of $g\gamma$ between $gx$ and $gz$.  Then $\d_S(gx,t)\geq \d_S(x,z)-2\delta-2\e$.  Since $\d_S(gx,gz)=\d_S(x,z)$, it follows that $\d_S(t,gz) = \d_S(gx,gz)-\d_S(gx,t)\leq 2\e+2\delta$.  Thus by the triangle inequality, $\d_S(z,gz)\leq 3\e+4\delta$.  Setting $E=3\e+4\delta$ completes the proof. \end{proof}

\begin{rem} \label{acylwithequal}
It is shown in \cite[Lemma 2.4]{Osi16} that the action of a group $G$ on a hyperbolic space $S$ is acylindrical if and only if for every $\e>0$ there exist $R,N>0$ such that for every two points $x,z$ satisfying $\d(x,z)=R$, $$\#\{g\in G\mid \max\{\d(x,gx),\d(z,gz)\}\leq \e\}\leq N.$$
\end{rem}

\begin{proof}[Proof of Proposition \ref{lem:Hacyl}]
By Definition \ref{defhe}(a), there exists $\delta\geq 0$ such that $\Gamma(G,X\sqcup H)$ is $\delta$--hyperbolic.  For any $\eta\geq 0$, let $R(\eta)$ and $N(\eta)$ be the constants of acylindricity associated to the action of $G$ on $\Gamma(G,X)$.

Our goal is to prove that the action of $G$ on $\Gamma(G,X\sqcup H)$ is acylindrical.   Let us fix $\varepsilon>0$. Let $D$ be the constant provided by  Corollary \ref{Yibound} with $n=1$ and $Y_1=X$, and let $E$  be the constant provided by Lemma \ref{defineE} applied to $\Gamma(G,X\sqcup H)$.  Fix a natural number $K$ such that $$K>\max\{R(20DE),\varepsilon+1,2E+1\}.$$

 Let  $g\in G$ be such that $$\d_{X\cup H}(1,g)= 3K,$$  and suppose an element $a$ in $G$ satisfies
 \begin{equation} \label{epsilonbound} \d_{X\cup H}(1,a)\leq \varepsilon \quad \textrm{ and } \quad \d_{X\cup H}(g,ag)\leq \varepsilon.
 \end{equation}  By Remark \ref{acylwithequal}, it suffices to give a uniform bound on the number of such $a\in G$.

For any $x,y\in G$, $[x,y]$ will always denote a geodesic in $\Gamma(G,X\sqcup H)$ connecting $x$ to $y$.  Whenever we use this notation, the choice of particular geodesic will be irrelevant.  Choose two points $x$ and $y$ on $ [1,g]$ such that $$\d_{X\cup H}(1,x)=K \quad \textrm{ and } \quad \d_{X\cup H}(y,g)=K.$$  It follows from Lemma \ref{defineE} that
\begin{equation}\label{E}\d_{X\cup H}(x,ax)\leq E \quad \textrm{ and } \quad \d_{X\cup H}(y,ay)\leq E.
\end{equation}  There are two cases to consider: either all the $H$--components of $[x,y]$ are sufficiently short in $\Gamma(G,X)$, or there exists an $H$--component of $[x,y]$ that is long.  In the first case, we will bound the distances from $x$ and $y$ to $ax$ and $ay$ in $\Ga(G,X)$, respectively, and use the acylindricity of the action of $G$ on $\Gamma(G,X)$ to bound the number of $a\in G$ satisfying (\ref{epsilonbound}).  In the second case we will use the local finiteness of the metric space $(H,\widehat \d)$ to bound the number of such $a$.

 {\bf Case 1.} Suppose all $H$--components $u$ of $[x,y]$ satisfy
 \begin{equation} \label{Hbound}
 \d_X(u_-,u_+)\leq 4D,\end{equation} and let $b$ be an $H$--component of $[x,ax]$.    There are four possibilities.
	
	a) If $b$ is connected to an $H$--component $v$ of $[y,ay]$, then there is an edge $e$ labeled by an element of $H$ connecting $b_-$ to $v_-$ (see Figure  \ref{5_3_1a}).

\begin{figure}
\centering
\begin{minipage}{.49\textwidth}
  \centering
  \def\svgscale{0.43}
  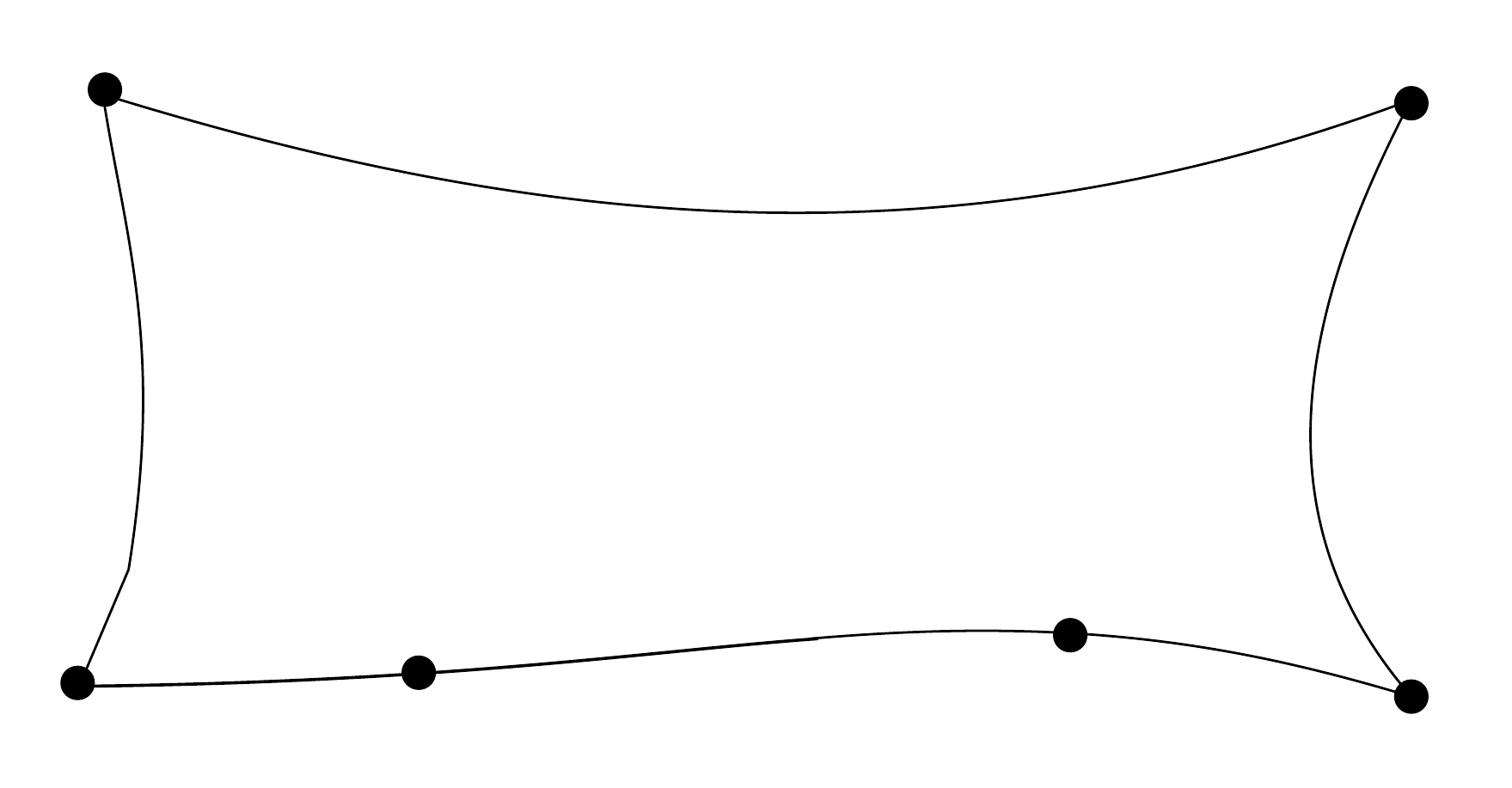
\caption{Case 1 (a)}
\label{5_3_1a}
\end{minipage}%
\begin{minipage}{.49\textwidth}
  \centering
  \def\svgscale{0.44}
 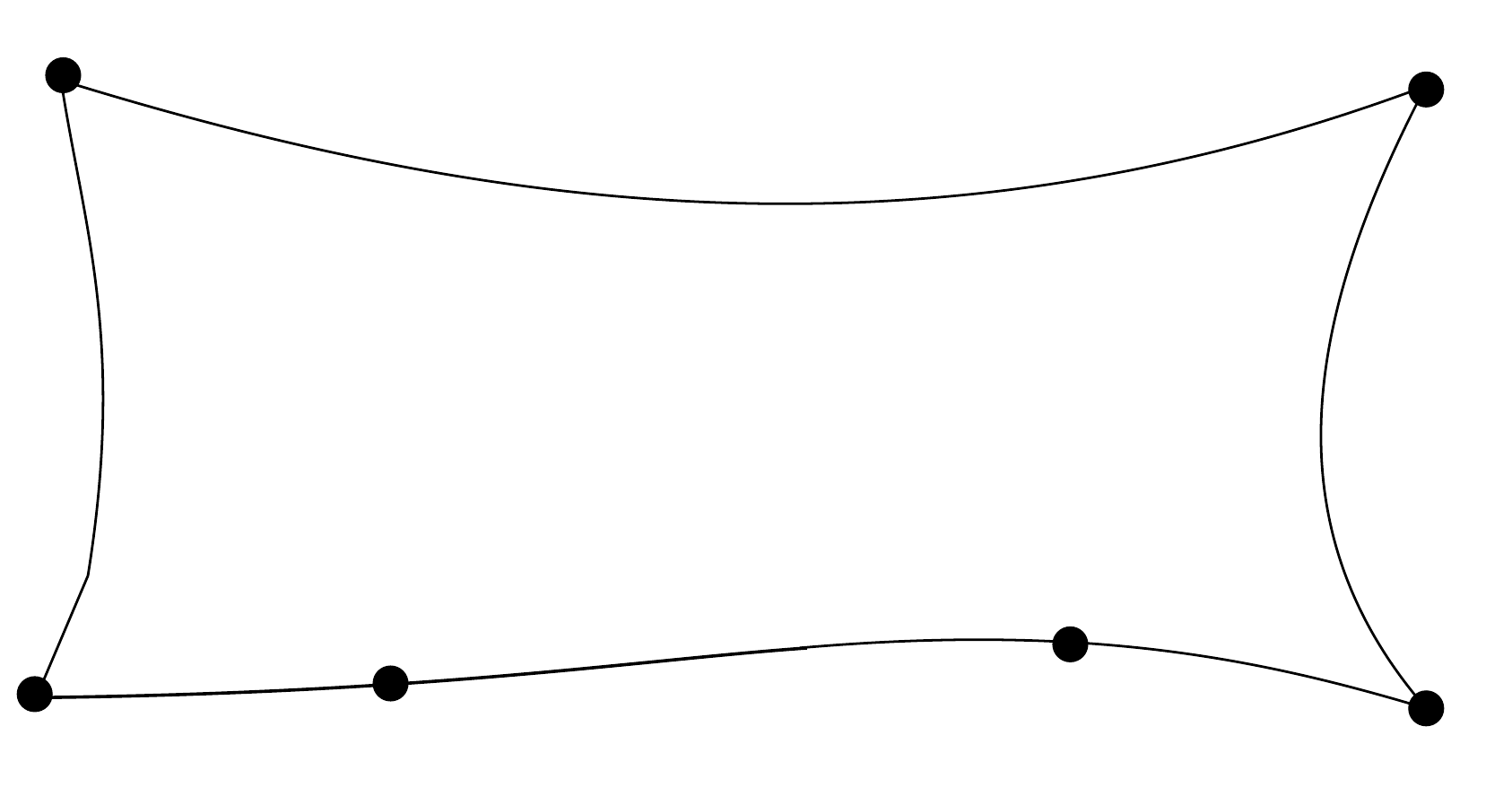
\caption{Case 1 (b)}
\label{5_3_1c}
\end{minipage}%

\end{figure}

By the triangle inequality,
	\begin{align*}
	\d_{X\cup H}(x,y) & \leq \d_{X\cup H}(x,b_-)+\d_{X\cup H}(b_-,v_-)+\d_{X\cup H}(v_-,y) \\
	& \leq \d_{X\cup H}(x,ax)+\d_{X\cup H}(y,ay)+1 \\
	&\leq 2E+1.
	\end{align*}
However, as $\d_{X\cup H}(x,y) = K>2E+1$, we reach a contradiction.

	b) If $b$ is connected only to an $H$--component $u$ of $[x,y]$, then let $e_1,e_2$ be the edges labeled by elements of $H$ connecting $b_-$ to $u_-$ and $b_+$ to $u_+$, respectively (see Figure  \ref{5_3_1c}).

The edges $e_1$ and $e_2$ are isolated in $e_1\cup [u_-,x]\cup [x,b_-]$ and $e_2\cup[u_+,y]\cup [y,ay]\cup[ay,ax]\cup[ax,b_+]$, respectively, so by Corollary \ref{Yibound}, $ \d_X((e_i)_-,(e_i)_+)\leq 5D$ for $i=1,2$.  By (\ref{Hbound}) and the triangle inequality, \begin{equation}\label{c}\d_X(b_-,b_+)\leq 14D.\end{equation}

Without loss of generality, we may assume that $a[1,g]=[a,ag]$.  Then any $H$--component of $[ax,ay]$ is the image under $a$ of an $H$--component of $[x,y]$.  Thus by an analogous argument, we get the same bound if $b$ is connected only to an $H$--component of $[ax,ay]$.
	
	c) If $b$ is connected to an $H$--component $u$ of $[x,y]$ and an $H$--component $v$ of $[ax,ay]$, then let $e_1,e_3,e_4$ be the edges labeled by elements of $H$ connecting $u_-$ to $b_-$, $b_+$ to $v_-$, and $v_+$ to $u_+$, respectively (see Figure  \ref{5_3_1b}).

By the reasoning in a), $b$ cannot also connect to an $H$-component of $[y,ay]$.  Thus, the edges $e_1,e_3,e_4$ are isolated in $e_1\cup [u_-,x]\cup [x,b_-]$, $e_3\cup [v_-,ax]\cup[ax,b_+]$, and $e_4\cup[u_+,y]\cup[y,ay]\cup[ay,v_+]$, respectively.  By Corollary \ref{Yibound}, $\d_X((e_i)_-,(e_i)_+)\leq 4D$ for $i=1,3,4$, and by (\ref{Hbound}), $\d_X(u_-,u_+)\leq 4D$.  As above, we may assume that any $H$--component of $[ax,ay]$ is the image under $a$ of an $H$--component of $[x,y]$, and it follows from (\ref{Hbound}) that $\d_X(v_-,v_+)\leq 4D$.  By the triangle inequality, \begin{equation}\label{b}\d_X(b_-,b_+)\leq 20D.\end{equation}

	d) If $b$ is isolated in $[x,y]\cup [y,ay]\cup[ay,ax]\cup [ax,x]$, then by Corollary \ref{Yibound}, \begin{equation}\label{d}\d_X(b_-,b_+)\leq 4D.\end{equation}

	Equations (\ref{b}), (\ref{c}), and (\ref{d}) show that for any $H$--component $b$ of $[x,ax]$, we have  $\d_X(b_-,b_+)\leq 20D.$  Combining this with (\ref{E}) yields that $$\d_X(x,ax)\leq 20DE.$$  By a symmetric argument, $$\d_X(y,ay)\leq 20DE.$$  Since $$\d_{X}(x,y)\geq \d_{X\cup H}(x,y)=K>R(20DE),$$ the acylindricity of the action of $G$ on $\Gamma(G,X)$ allows us to conclude that there are at most $N(20DE)$ elements $a\in G$ satisfying (\ref{epsilonbound}).

 {\bf Case 2.} Suppose there exists an $H$--component $c$ of $[x,y]$ with $\d_X(c_-,c_+) > 4D$.  Then by Corollary \ref{Yibound}, $c$ cannot be isolated in the quadrilateral $[1,g]\cup [g,ag]\cup[ag,a]\cup [a,1]$.  However, $c$ cannot connect to $[1,a]$ or $[g,ag]$.  Indeed, if $c$ connects to an $H$--component $b$ of $[1,a]$, then there is an edge labeled by an element of $H$ connecting $b_-$ to $c_-$ (see Figure  \ref{5_3_2not}), and so $$\d_{X\cup H}(1,c_-)\leq \d_{X\cup H}(1,b_-)+\d_{X\cup H}(b_-,c_-)\leq \varepsilon+1.$$

\begin{figure}
\centering
\begin{minipage}{.49\textwidth}
  \centering
  \def\svgscale{0.43}
  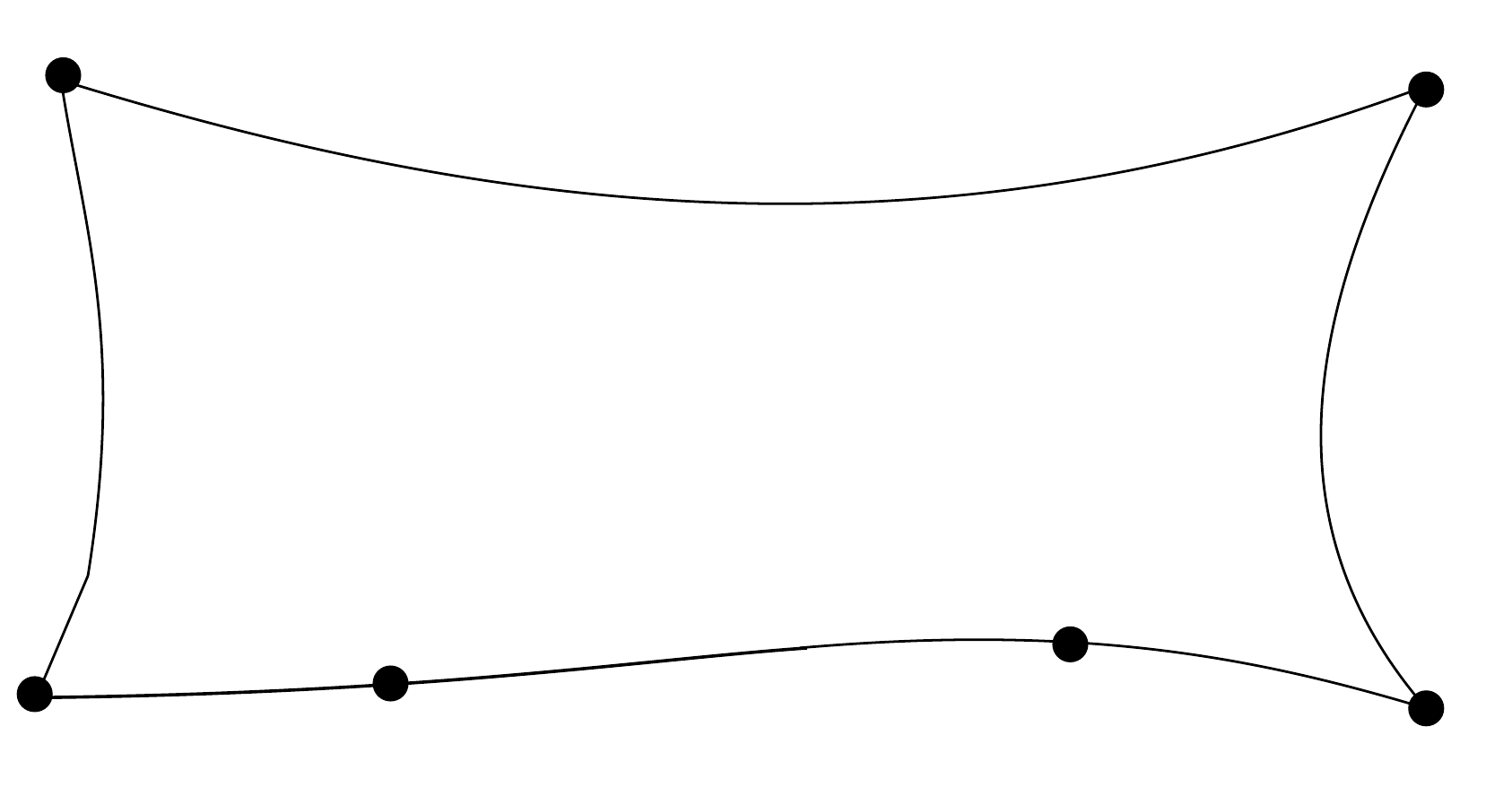
\caption{Case 1 (c)}
\label{5_3_1b}
\end{minipage}%
\begin{minipage}{.49\textwidth}
  \centering
  \def\svgscale{0.44}
 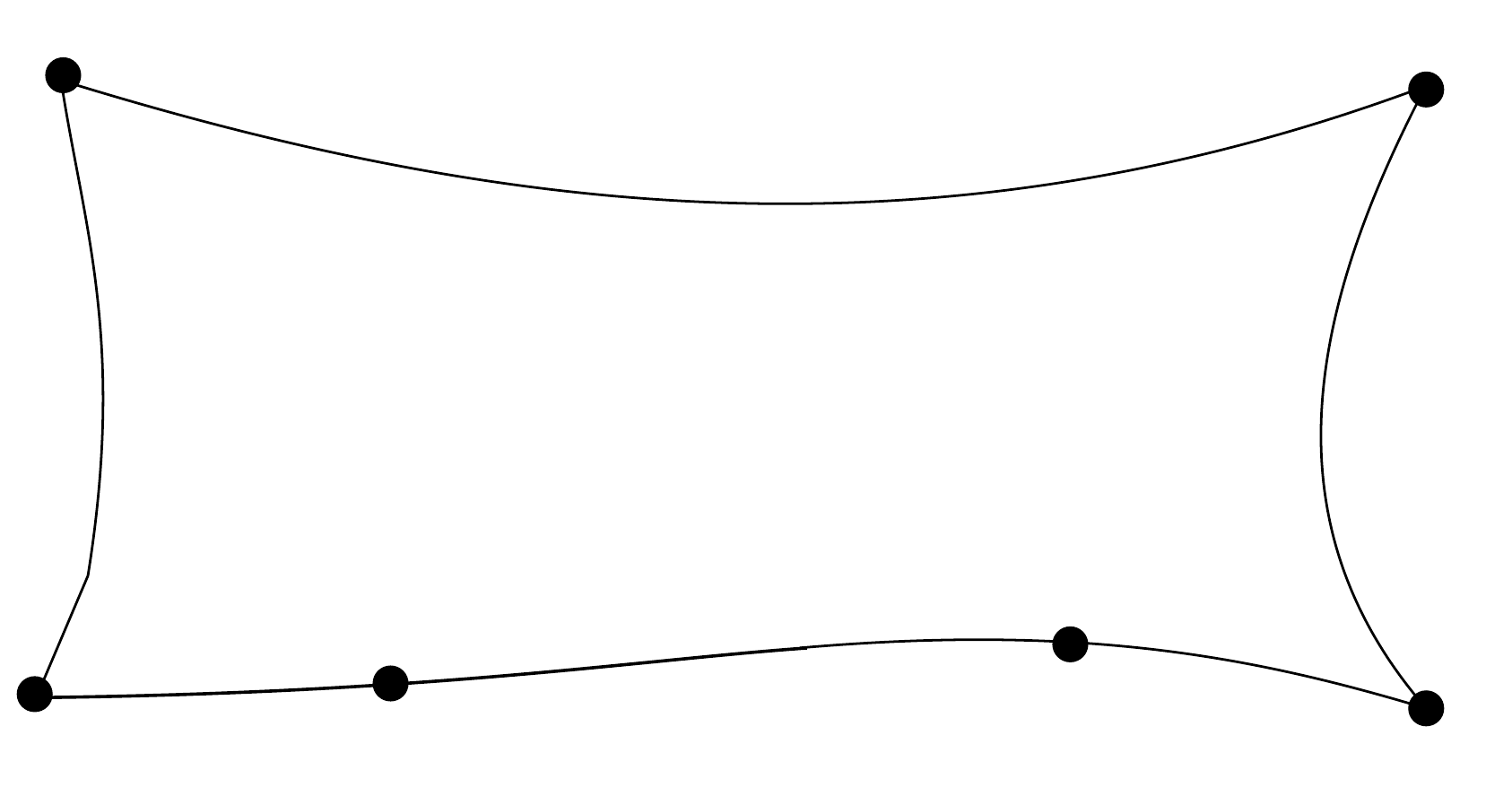
\caption{The component $c$ cannot be connected to $[1,a]$ or $[g, ag]$.}
\label{5_3_2not}
\end{minipage}
\end{figure}

As $$\d_{X\cup H}(1,c_-)\geq \d_{X\cup H}(1,x)=K>\varepsilon +1,$$ we reach a contradiction.  A similar contradiction will be reached if $c$ connects to an $H$--component of $[g,ag]$.  Therefore $c$ must connect only to an $H$--component $c'$ of $[a,ag]$.  Let $e$ be an edge labeled by an element of $H$ connecting $c_-$ to $c'_-$ (see Figure  \ref{5_3_2}).

The edge $e$ must be isolated in $[1,c_-]\cup  e \cup [c_-',a]\cup [a,1]$, so by Lemma \ref{C} there is a constant $C$ such that $\widehat \d(e_-,e_+)\leq 4C$.  Let $u=[1,c_-]$ and $v=[c_-',a]$.  Then the label of $uev$ and $a$ represent the same element of $G$.  There are at most $3K$ choices for $u$ and $v$, as they are subpaths of $[1,g]$ and $[a,ag]$, respectively, both of which have length $3K$.  Let $B$ be the number of elements in a ball of radius $4C$ in the metric space $(H,\dl)$. The number of choices for $e$ is bounded by $B$, and by Definition \ref{defhe}(b), $B<\infty$.    Therefore, there are at most $9K^2B$ choices for $a$.

In either case, there are at most $\max\{N(20DE),9K^2B\}$ elements $a\in G$ satisfying (\ref{epsilonbound}), which completes the proof. \end{proof}

\begin{figure}
\centering
\def\svgscale{0.4}
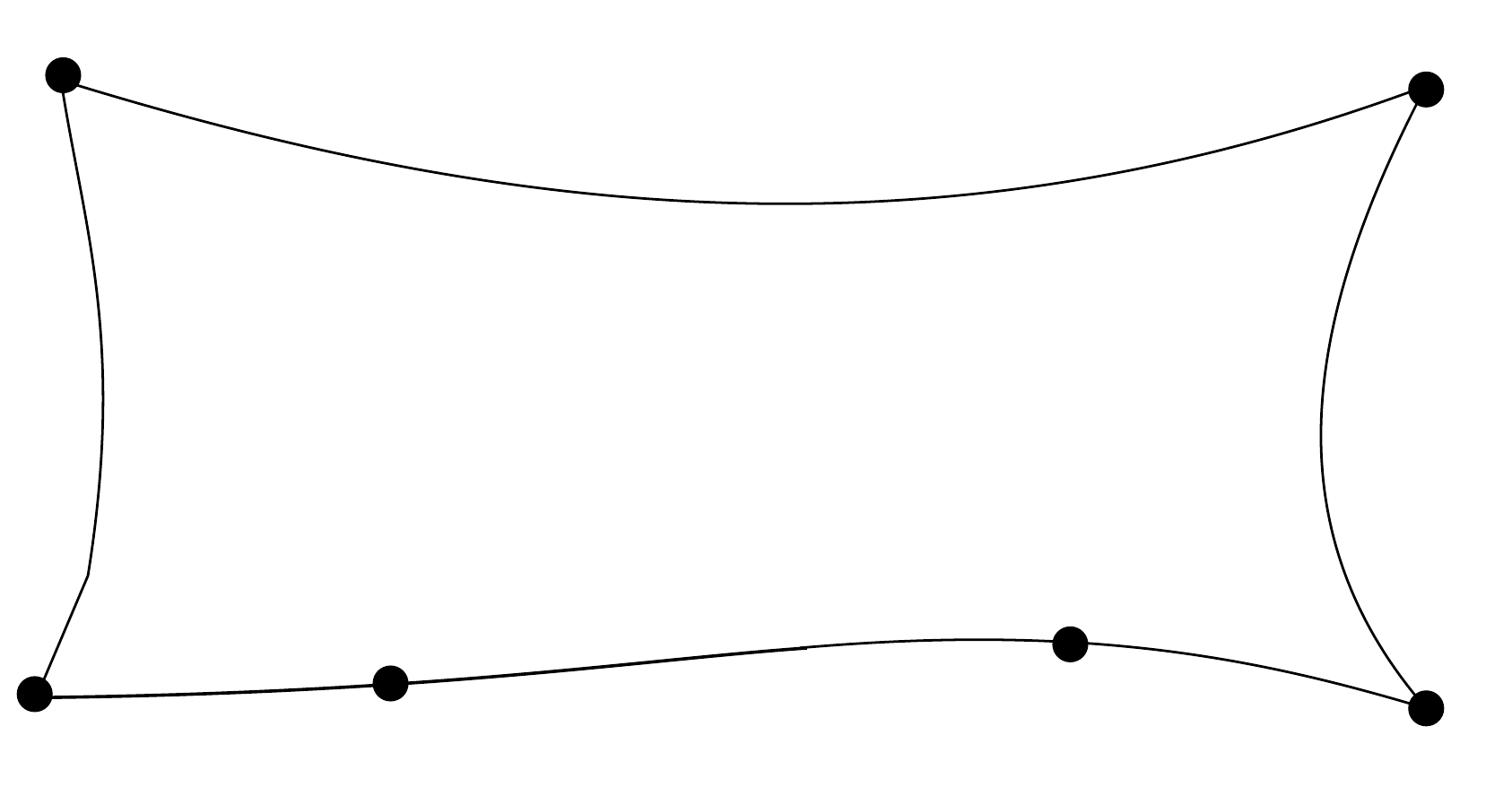
\caption{Case 2 }
\label{5_3_2}
\end{figure}

Note that we do not require $\Gamma (G,X)$ to be hyperbolic in Proposition \ref{lem:Hacyl}. In particular, we recover a well-known fact that the action of a relatively hyperbolic group on the relative Cayley graph is acylindrical, see \cite[Proposition 5.2]{Osi16}.

\begin{ex}\label{rhshe}
Let $G$ be a finitely generated group hyperbolic relative to a finite collection of subgroups $\Hl$. Then $\Hl$ is strongly hyperbolically embedded into $G$ with respect to any finite generating set $X$ of $G$. Indeed this follows from Theorem \ref{rhhe} and Proposition \ref{lem:Hacyl} since the action of $G$ on $\Gamma (G,X)$ in this case is obviously acylindrical.
\end{ex}

Our proof of Theorems \ref{main1} and \ref{main3} makes use of strongly hyperbolically embedded subgroups of the form $F_2\times K(G)$ in every acylindrically hyperbolic group.  More precisely, we will need the following.

\begin{prop} \label{strongheF2}
For every acylindrically hyperbolic group $G$ and every non-elementary $[Y]\in \AHG$, there is a subgroup $H$ of $G$ isomorphic to $F_2\times K(G)$ which is strongly hyperbolically embedded in $G$ with respect to $Y$.  Moreover, the action of $H$ on $\Gamma(G,Y)$ is purely loxodromic.
\end{prop}

We begin with a lemma.

\begin{lem} \label{FheY}
Let $H_1,\dots, H_n,F$ be subgroups of $G$ and $Y\subset G$ a relative generating set of $G$ with respect to $F$ such that $\{H_1,\dots,H_n,F\}\h (G,Y)$.  If $H_i$ is hyperbolic for $i=1,\dots, n$, then $F\h(G,Y)$.
\end{lem}

\begin{proof}
Let $X_1,\dots, X_n$ be finite generating sets for $H_1,\dots, H_n$, respectively, and let $\mathcal H=\sqcup_{i=1}^n H_i$ and $\mathcal X=\sqcup_{j=1}^n X_i$.  Since  $|\mathcal X|<\infty$, it suffices to show that $F\h (G,Y\sqcup\mathcal X)$ by \cite[Corollary 4.27]{DGO}.  As $[X_i]\in\mathcal H(H_i)$ and $[F]\in\mathcal H(F)$, by Theorem \ref{AHO}, $[Y\cup X_1\cup \cdots \cup X_n\cup F]\in\mathcal H(G)$.  Thus the first part of Definition \ref{defhe} is satisfied.

Let $\widehat \d_1$ and $\widehat \d_2$ be the relative metrics on $F$ defined by taking admissible paths in $\Gamma(G,Y\sqcup\mathcal H\sqcup F)$ and $\Gamma(G,Y\sqcup \mathcal X\sqcup F)$, respectively, as in Definition \ref{defhe}.

It remains to show the second condition of Definition \ref{defhe} holds, that is, that $(F,\widehat \d_2)$ is a proper metric space.  We naturally think of $\Gamma(F,F)$ as a subgraph of $\Gamma(G,Y\sqcup \mathcal X\sqcup F)$ and fix $f\in F$.  Consider a ball $B$ in $F$ centered at $f$ of  radius $R<\infty$ in the metric $\widehat \d_2$. Then for any $f'\in B$,  $f$ and $f'$ are connected in $\Gamma(G,Y\sqcup \mathcal X\sqcup F)$ by an admissible path $p$ of  length at most $R$, i.e., a path that does not contain any edge of $\Gamma(F,F)$.  Since $\Gamma(G,Y\sqcup \mathcal X\sqcup F)$ is itself a subgraph of $\Gamma(G,Y\sqcup\mathcal H\sqcup F)$, we can consider $p$ as a path in $\Gamma(G,Y\sqcup\mathcal H\sqcup F)$.  It is clear that $p$ is an admissible path in $\Gamma(G,Y\sqcup\mathcal H\sqcup F)$, as well, and so $\widehat \d_1(f,f')< R$.  Since $F$ is part of a hyperbolically embedded collection of subgroups, any ball in $F$ with respect to the metric $\widehat \d_1$ contains finitely many elements.  Therefore $B$ contains finitely many elements, completing the proof.
\end{proof}

Let $G$ be a group acting on a metric space $S$.  A subgroup $H\leq G$ is called \emph{geometrically separated} if for all $\varepsilon\geq 0$ and all $s\in S$, there exists $R\geq 0$ such that the following holds.  If for some $g\in G$, $\operatorname{diam}(Hs\cap(gHs)^{+\varepsilon})\geq R$ in $S$, then $g\in H$.  Here, $(gHs)^{+\varepsilon}$ denotes the closed $\varepsilon$--neighborhood of $gHs$.

\begin{proof}[Proof of Proposition \ref{strongheF2}]
By \cite[Lemma 6.18]{DGO}, there exist independent loxodromic elements $g_1,\dots, g_6\in \mathcal L([Y])$ such that $E(g_i)\simeq \langle g_i\rangle \times K(G)$.  By \cite[Corollary 3.17]{Hull}, $$\{E(g_1),\dots, E(g_6)\}\h(G,Y).$$

 Let
 \begin{equation} \label{defH}
 H=\langle a,b,  K(G)\rangle,
 \end{equation}
  where $a=g_1^ng_2^ng_3^n$ and $b=g_4^ng_5^ng_6^n$ for sufficiently large $n$, and let $\mathcal E=E(g_1)\setminus\{1\}\sqcup \cdots \sqcup E(g_6)\setminus\{1\}$.  It is shown in the proof of \cite[Theorem 6.14]{DGO} that $\langle a,b\rangle$ is isomorphic to $F_2$, that $H\simeq \langle a,b\rangle \times K(G)$, and that $H$ is quasi-convex and geometrically separated in $\Gamma(G, Y\sqcup\mathcal E)$.  By \cite[Theorem 3.16]{Hull}, it follows that $H\h(G,Y\cup \mathcal E)$. Therefore by \cite[Remark 3.4]{AMS} and \cite[Theorem 3.9]{AMS},  $$\{E(g_1),\dots,E(g_6),H\}\h(G,Y).$$  Since each $E(g_i)$ is virtually cyclic, and so hyperbolic, and $Y$ is a generating set of $G$, Lemma \ref{FheY} implies that $$H\h(G,Y).$$
The action of $H$ on $\Gamma(G,Y)$ is acylindrical as $[Y]\in\AHG$, and therefore $H$ is strongly hyperbolically embedded in $(G,Y)$ by Proposition \ref{lem:Hacyl}.

It is shown in the proof of \cite[Theorem 6.14]{DGO} that $H$ acts properly on $\Gamma(G,Y\sqcup\mathcal E)$.  Thus $H$ acts properly on $\Gamma(G,Y)$, as well, and so the action is purely loxodromic, completing the proof.
\end{proof}

\subsection{Acylindricity of induced structures}\label{Sec:AIA}

We begin by recalling some useful results from \cite{AHO}, which play the central role in the proof of Theorems \ref{main1}, \ref{main6}, and \ref{main3}. Let $H_1,\dots, H_n$ be subgroups of a group $G$ and let $X$ be a relative generating set for $G$ with respect to $H_1,\dots,H_n$.  Let
$$\iota_X\colon \mathcal G(H_1)\times\cdots\times\mathcal G(H_n)\to\mathcal G(G)$$
be the map defined by
\begin{equation}\label{iotadef}
\iota_X([Y_1],\dots, [Y_n])= \left[X\cup\left(\bigcup_{i=1}^nY_i\right)\right].
\end{equation}

This map can be thought of as the analogue of the induced action map defined in \cite{AHO} for equivalence classes of group actions on geodesic metric spaces. In the theorem below, we restate some of the results of \cite{AHO} using terminology of this paper.

\begin{thm}\label{AHO}
Let $G$ be a group, let $H_1,\dots, H_n$ be subgroups of $G$, and let $X$ be a relative generating set for $G$ with respect to $H_1,\dots,H_n$.
Then the map $\iota_X$ defined by (\ref{iotadef}) is well-defined and order preserving. If, in addition, $\Hi\h(G,X)$, then the following hold.
\begin{enumerate}
\item[(a)] $\iota_X$ sends $\mathcal H(H_1)\times\cdots\times \mathcal H(H_n)$ to $\mathcal H(G)$.
\item[(b)] Let $([Y_1],\dots, [Y_n])\in \mathcal G(H_1)\times\cdots\times\mathcal G(H_n)$ and let $[Z]=\iota_X([Y_1],\dots, [Y_n])$. Then for every $i=1, \ldots, n$, we have
    \begin{equation}\label{extension}
    H_i\curvearrowright \Gamma (G, Z)\sim_w H_i\curvearrowright\Gamma (H_i, Y_i).
    \end{equation}
    In particular, $\iota _X$ is injective.
\end{enumerate}
\end{thm}

\begin{proof}
The first claim of the theorem is obvious, so we only need to explain how (a) and (b) follow from results proved in \cite{AHO}. In our situation, the induced action map studied in \cite{AHO} is given by  $$([H_1\curvearrowright\Gamma(H_1,Y_1)],\cdots,[H_n\curvearrowright\Gamma(H_n,Y_n)])\mapsto [G\curvearrowright\Gamma(G,X\cup (\cup_{i=1}^n Y_i))],$$ see \cite[Theorem 3.26(b)]{AHO}. Using the isomorphism of posets $\mathcal A_{cb}(G)$ and $\mathcal G(G)$ defined in Proposition \ref{prop-SM}, it is easy to see that the corresponding map on equivalence classes of generating sets (defined by the obvious commutative diagram) is exactly $\iota_X$. Therefore, $\iota_X$ maps hyperbolic structures to hyperbolic structures by \cite[Theorem 4.9(a)]{AHO}.

Part (b) is also an immediate consequence of \cite[Theorem 4.9(a)]{AHO}. Indeed, in our notation \cite[Theorem 4.9(a)]{AHO} states that $G \curvearrowright\Gamma (G, Z)$ is an extension of $H_i\curvearrowright \Gamma (H_i, Y_i)$, which means that there exists a coarsely $H_i$-equivariant quasi-isometric embedding $f\colon \Gamma (H_i, Y_i)\to \Gamma (G,Z)$. Assume that $f$ satisfies the definitions of a quasi-isometric embedding and a coarsely $H_i$-equivariant map with the constant $C$. Then for every $h\in H_i$, we have
$$
\d_{Y_i} (1,h)\le C(\d_Z (f(1), f(h)) +C) \le C\d_Z (f(1), hf(1)) +2C^2
$$
Thus $H_i\curvearrowright \Gamma (H_i, Y_i)\preceq H_i\curvearrowright\Gamma (G,Z)$. The opposite inequality is obvious since we can assume that $Y_i\subseteq Z$ without loss of generality, see (\ref{iotadef}). Thus we have (\ref{extension}).

In particular, if $\iota_X ([X_1],\dots, [X_n])=\iota _X([Y_1],\dots, [Y_n])$, then the actions $H_i\curvearrowright \Gamma (H_i, X_i)$ and $H_i\curvearrowright \Gamma (H_i, Y_i)$ are weakly equivalent for all $i$. By Lemma \ref{e-we} (b), these actions are equivalent and consequently $[X_i]=[Y_i]$ by Proposition \ref{prop-SM}. Thus $\iota _X$ is injective.
\end{proof}

The main result of this section shows that $\iota_X$ preserves acylindricity whenever $\Hi$ is strongly hyperbolically embedded in $G$ with respect to $X$ (see Definition \ref{defhe}).

In order to simplify constants involved in the proof of the main theorem of this section, it is convenient to accept the following.
\begin{conv}
All constants are assumed to be positive integers.
\end{conv}\label{posint}
In other words, we use the word ``constant" as a synonym of ``positive integer" throughout the rest of this section. It will be obvious in each case that this assumption can be made without loss of generality.

\begin{thm}\label{genmain2}
Suppose that a collection of subgroups $\{H_1,\dots, H_n\}$ is strongly hyperbolically embedded in a group $G$ with respect to a relative generating set $X$. Then for every $A\in \mathcal{AH}(H_1)\times\dots\times\mathcal{AH}(H_n)$, we have $\iota_X (A)\in \AHG$.
\end{thm}

\begin{proof}
Let $A=([Y_1], \ldots, [Y_n])$, and let $\mathcal Y=\sqcup_{i=1}^n Y_i$.  Since each $\Gamma(H_i,Y_i)$ is hyperbolic, it follows from  Theorem \ref{AHO} that $\Gamma(G,X\sqcup\mathcal Y)$ is hyperbolic, as well. By assumption the action of $G$ on $\Gamma(G,X\sqcup \mathcal H)$ and the action of each $H_i$ on $\Gamma(H_i,Y_i)$ is acylindrical.  For any constant $\eta$, let $R(\eta)$ and $N(\eta)$ denote the corresponding constants such that the definition of acylindricity is satisfied for each of these (finitely many) actions.

Our goal is to show that the action of $G$ on $\Gamma(G,X\sqcup\mathcal Y)$ is acylindrical.   Let us fix any constant $\varepsilon$ and let $D$ be the constant from Corollary \ref{Yibound}.
We fix a constant $M$ satisfying
\begin{equation}\label{M}
M>\max\{R(10\e D),18\e D\}.
\end{equation}

 Let $g$ be an element in $G$ such that
\begin{equation}\label{d1g}
 \d_{X\cup \mathcal Y}(1,g)=R(\e)M,
\end{equation}
and suppose $a\in G$ satisfies
\begin{equation}\label{thinquad}
\d_{X\cup \mathcal Y}(1,a)\leq \varepsilon \quad \textrm{and} \quad \d_{X\cup \mathcal Y}(g,ag)\leq \varepsilon.
\end{equation}    By Remark \ref{acylwithequal}, it suffices to give a uniform bound on the number of such $a\in G$.

The vertex sets of $\Gamma(G,X\sqcup \mathcal H)$ and $\Gamma(G,X\sqcup\mathcal Y)$ coincide, and we naturally consider the latter as a subgraph of the former.  As $\mathcal Y\subset\mathcal H$, it follows from (\ref{thinquad}) that $\d_{X\cup \mathcal H}(1,a)\leq \varepsilon$ and $\d_{X\cup \mathcal H}(g,ag)\leq \varepsilon$.  Recall that for any $x,y\in G$, $[x,y]$ always denotes a geodesic in $\Gamma(G,X\sqcup\mathcal H)$ connecting $x$ to $y$, and the choice of a particular geodesic does not matter.  For any path $q$ in $\Gamma(G,X\sqcup\mathcal H)$, let $\ell(q)$ denote its length.

We begin by bounding the length of $H_i$--components of $[1,a]$ and $[g,ag]$ in the $\d_{Y_i}$--metric for each $i$.

\begin{lem} \label{lem:boundb}
If $b$ is an $H_i$--component of $[1,a]$ or $[g,ag]$ for some $i=1,\dots, n$, then $\d_{Y_i}(b_-,b_+)\leq 4 D\e$.
\end{lem}

\begin{proof}
After possibly replacing $g$ by $g^{-1}$, we may assume that $b$ is an $H_i$--component of $[1,a]$.   Let $p$ be a geodesic in $\Gamma(G,X\sqcup\mathcal Y)$ connecting $1$ to $a$.  By (\ref{thinquad}),  $\ell(p)\leq \e$.

The three segments, $[a,b_+]$,  $b$, and  $[b_-,1]$, along with the at most $\varepsilon$ edges of $p$ form an $n$--gon with $n\leq \varepsilon+3$.   If $b$ is isolated in this $n$--gon, then by Corollary \ref{Yibound}, $$\d_{Y_i}(b_-,b_+)\leq(\varepsilon+3)D\leq 4 D\e.$$  		
		
If $b$ is not isolated, then it is connected to at least one and at most $\varepsilon$ $H_i$--components of $p$ (see Figure \ref{figforlemma}). Note that $b$ cannot be connected to an $H_i$--component of $[1,b_-]$ or $[b_+,a]$ because $[1,a]$ is a geodesic.

\begin{figure}
\centering
\def\svgscale{0.75}
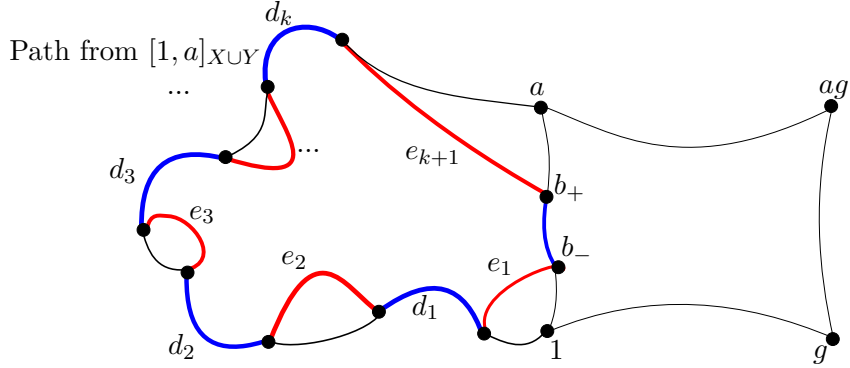
\caption{The blue segments denote $H_i$--components. The red segments denote single edges labeled by elements of $H_i$.}
\label{figforlemma}
\end{figure}

Let $d_1,\dots,d_k$ be the $H_i$-components of $p$ to which $b$ connects, labeled in the order they appear in $p$.  Obviously $k\leq \e$.  Let $e_1, e_{k+1}$ be edges labeled by elements of $H_i$ connecting $b_-$ to $(d_1)_-$ and $b_+$ to $(d_k)_+$, respectively, and let $e_j$ be edges labeled by elements of $H_i$ connecting $(d_{j-1})_+$ to $(d_{j})_-$ for $2\leq j\leq k$.  Then each $e_j$ is isolated in an $m_j$--gon whose sides are $e_j$ and edges of $p$ or $[1,a]$ such that $m_1+\dots +m_{k+1}\leq \ell(p)+\ell([1,a])\leq 2\e$.  Corollary \ref{Yibound} gives that $$\sum_{j=1}^{k+1}\d_{Y_i}((e_j)_-,(e_j)_+)\leq D\cdot\sum_{j=1}^{k+1} m_j \leq 2\e D.$$  Since $\sum_{j=1}^k\d_{Y_i}((d_j)_-,(d_j)_+)\leq \e$, the triangle inequality then implies that
$$
\d_{Y_i}(b_-,b_+)\leq \e(2D+1)< 4 \e D,
$$
completing the proof of the lemma.
\end{proof}

There are now two cases to consider.

 {\bf Case 1.}  Suppose for all $i=1,\dots,n$ and all $H_i$--components $c$ of $[1,g]$ we have $\d_{Y_i}(c_-,c_+)\leq M$.  Using (\ref{d1g}) we obtain $$\d_{X\cup \mathcal H}(1,g)\geq  \frac{\d_{X\cup\mathcal Y}(1,g)}{M}\ge \frac{R(\e)M}{M}\geq R(\e).$$  Since $\mathcal Y\subseteq \mathcal H$, we can use the acylindricity of the action of $G$ on $\Gamma(G,X\sqcup \mathcal H)$ to conclude that there are at most $N(\e)$ elements $a\in G$ satisfying (\ref{thinquad}).

 {\bf Case 2.}  Suppose for some $i$ there exists an $H_i$--component $c$ of $[1,g]$ such that \begin{equation}\label{clong} \d_{Y_i}(c_-,c_+)>M.\end{equation}  Since $M>4D$, $c$ cannot be isolated in $[1,a]\cup[a,ag]\cup[ag,g]\cup[g,1]$.  Then there are four possibilities.

	a)  If $c$ connects only to an $H_i$--component $b$ of $[1,a]$ (see Figure \ref{casea}), let $e_1$ and $e_2$ be the edges labeled by elements of $H_i$ connecting $c_-$ to $b_-$ and $c_+$ to $b_+$, respectively.  Then $e_1$ and $e_2$ are isolated in $e_1\cup [b_-,1]\cup[1,c_-]$ and $e_2\cup [b_+,a]\cup[a,ag]\cup[ag,g]\cup[g,c_+]$, respectively, and so by Corollary \ref{Yibound}, $\d_{Y_i}((e_j)_-,(e_j)_+)\leq 5D$ for $j=1,2$.  By Lemma \ref{lem:boundb}, $\d_{Y_i}(b_-,b_+)\leq 4 \e D$.    Applying the triangle inequality and (\ref{M}) yields $$\d_{Y_i}(c_-,c_+)\leq (10+4\e)D\leq M,$$ which contradicts our assumption on $c$.
	
If $c$ connects only to an $H_i$--component of $[g,ag]$, we reach the same contradiction by a symmetric argument.
\begin{figure}
\centering
\begin{minipage}{.49\textwidth}
  \centering
  \def\svgscale{0.43}
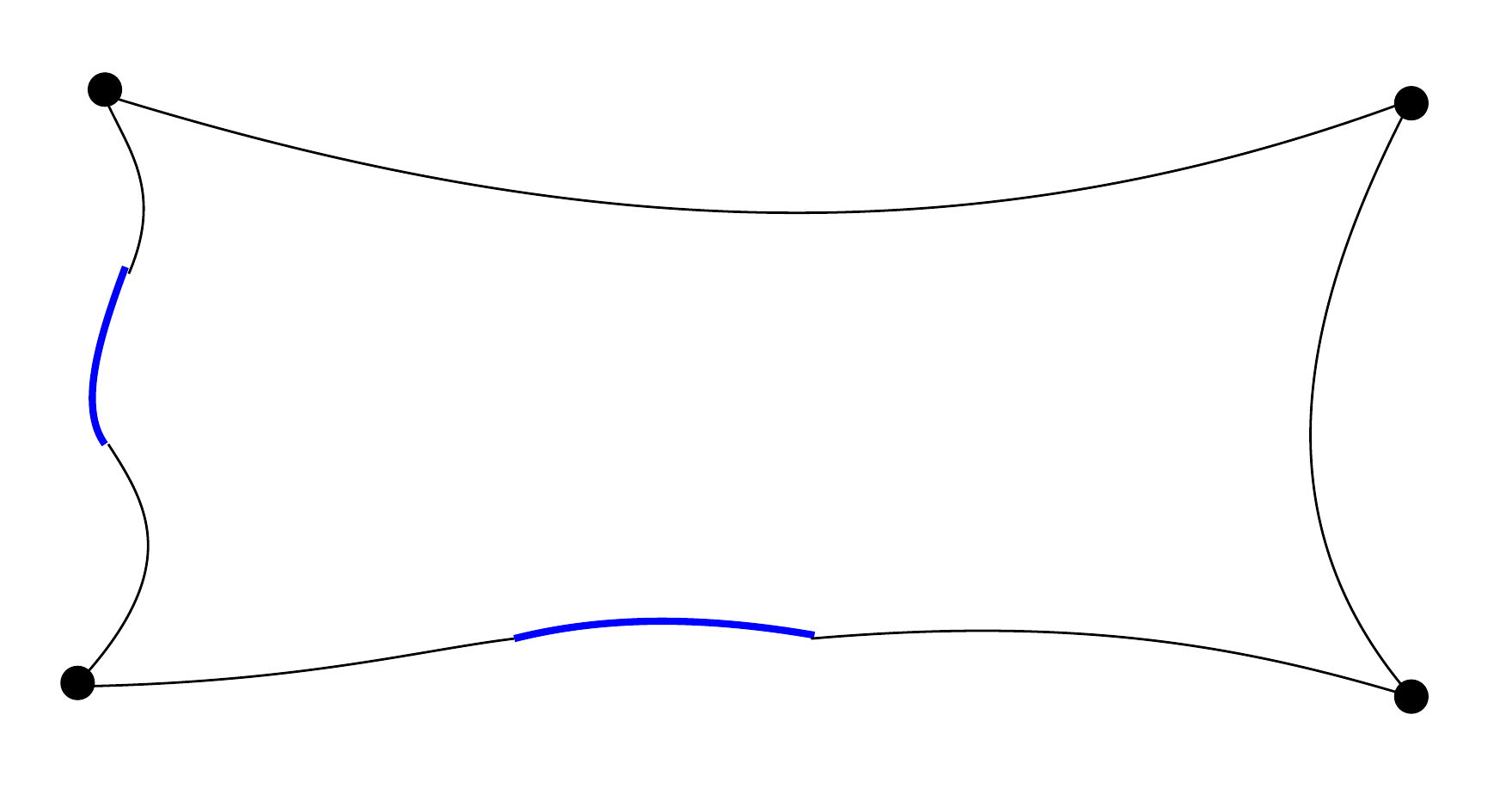
\caption{Case 2a}
\label{casea}
\end{minipage}%
\begin{minipage}{.49\textwidth}
  \centering
  \def\svgscale{0.43}
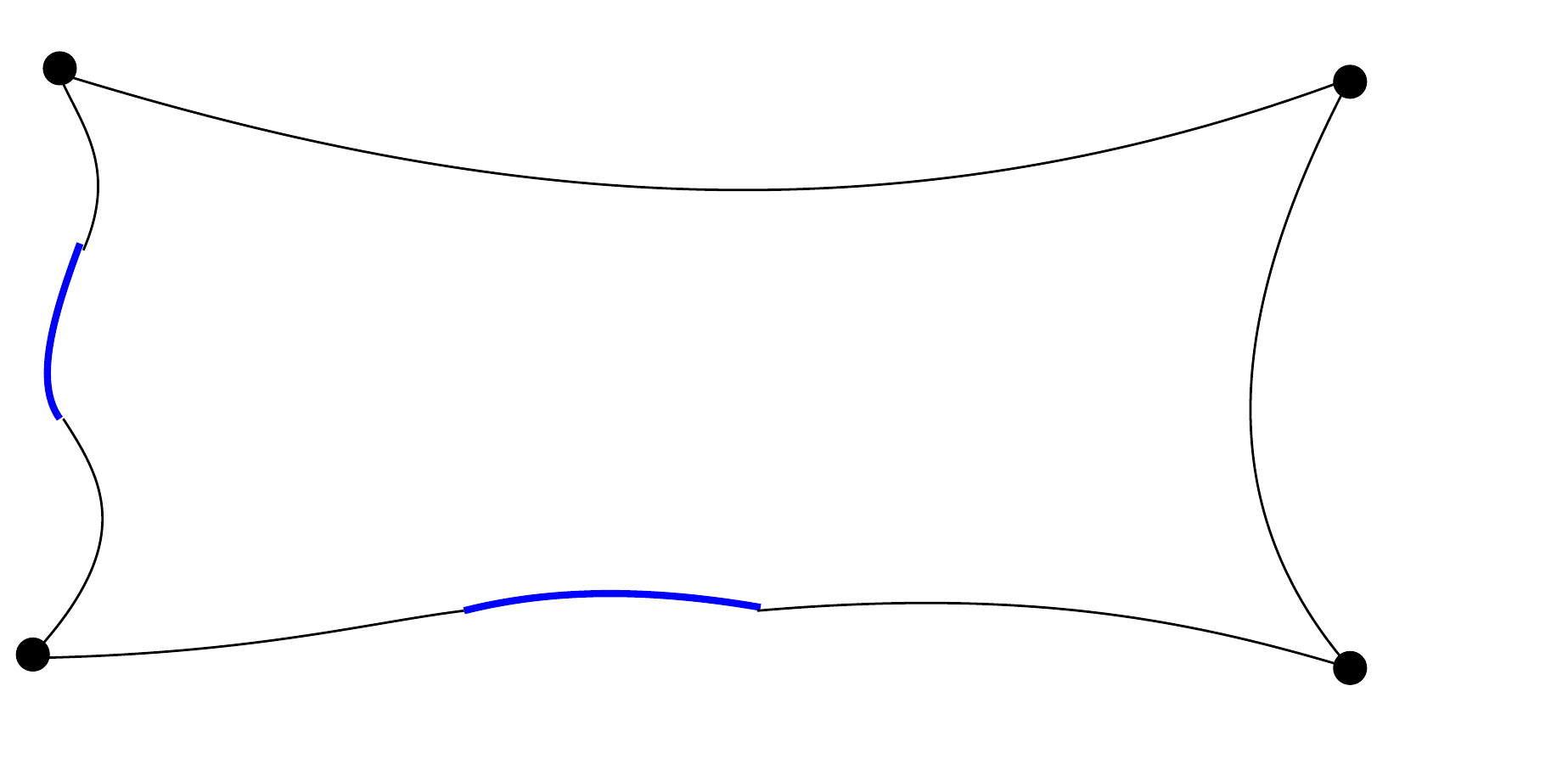
\caption{Case 2b}
\label{caseb}
\end{minipage}
\end{figure}

	b)  Suppose $c$ connects to an $H_i$--component $p$ of $[1,a]$ and an $H_i$--component $q$ of $[g,ag]$ but does not connect to any $H_i$--component of $[a,ag]$ (see Figure \ref{caseb}).  Then there is an edge $e$ labeled by an element of $H_i$ connecting $p_+$ to $q_+$, and $e$ is isolated in $e\cup [p_+,a]\cup[a,ag]\cup[ag,q_+]$.    Let $e_1,e_2$ be edges labeled by elements of $H_i$ connecting $c_-$ to $p_-$ and $c_+$ to $q_-$.  The edges $e_1,e_2$ are isolated in $e_1\cup [p_-,1]\cup[1,c_-]$ and $e_2\cup[q_-,g]\cup[g,c_+]$, respectively.  Thus by Corollary \ref{Yibound}, $\d_{Y_i}(p_+,q_+)\leq 4D$ and $\d_{Y_i}((e_j)_-,(e_j)_+)\leq 3D$ for $j=1,2$. By Lemma \ref{lem:boundb}, $\d_{Y_i}(p_-,p_+)\leq 4 \e D$ and $\d_{Y_i}(q_-,q_+)\leq 4 \e D$. It follows from the triangle inequality and (\ref{M}) that $$\d_{Y_i}(c_-,c_+)\leq (10+8\e)D\leq M,$$ contradicting (\ref{clong}).

	c) Suppose $c$ connects to an $H_i$--component $d$ of $[a,ag]$ and an $H_i$--component of at most one of $[1,a]$ and $[g,ag]$.
	
	c1) If $c$ does not connect to any $H_i$--component of $[ag,g]$, then let $e$ be an edge labeled by an element of $H_i$ connecting $c_+$ to $\d_+$ (see Figure \ref{casec1}).  The edge $e$ is isolated in $e\cup [\d_+,ag]\cup[ag,g]\cup [g,c_+]$, so by Lemma \ref{C}, there is a constant $C$ independent of $a$ and $g$ such that $\dl_{i}(e_-,e_+)\leq 4C.$ Let $u=[1,c_+]$ and $v=[a,\d_+]$.  Then the label of $uev^{-1}$ and $a$ represent the same element in $G$.  There are at most $R(\e) M$ choices for each of $u$ and $v$, as $u$ and $v$ are subpaths of $[1,g]$ and $[a,ag]$, respectively, both of which have length at most $R(\e) M$ by (\ref{d1g}).  Let $B$ be the number of elements in a ball of radius $4C$ in the metric space $(H_i,\widehat \d_i)$.  The number of choices for $e$ is bounded by $B$, and by Definition \ref{defhe}(b), $B<\infty$.  Therefore there are at most $(R(\e)M)^2B$ options for $a$.
\begin{figure}
\centering
\begin{minipage}{.49\textwidth}
  \centering
  \def\svgscale{0.43}
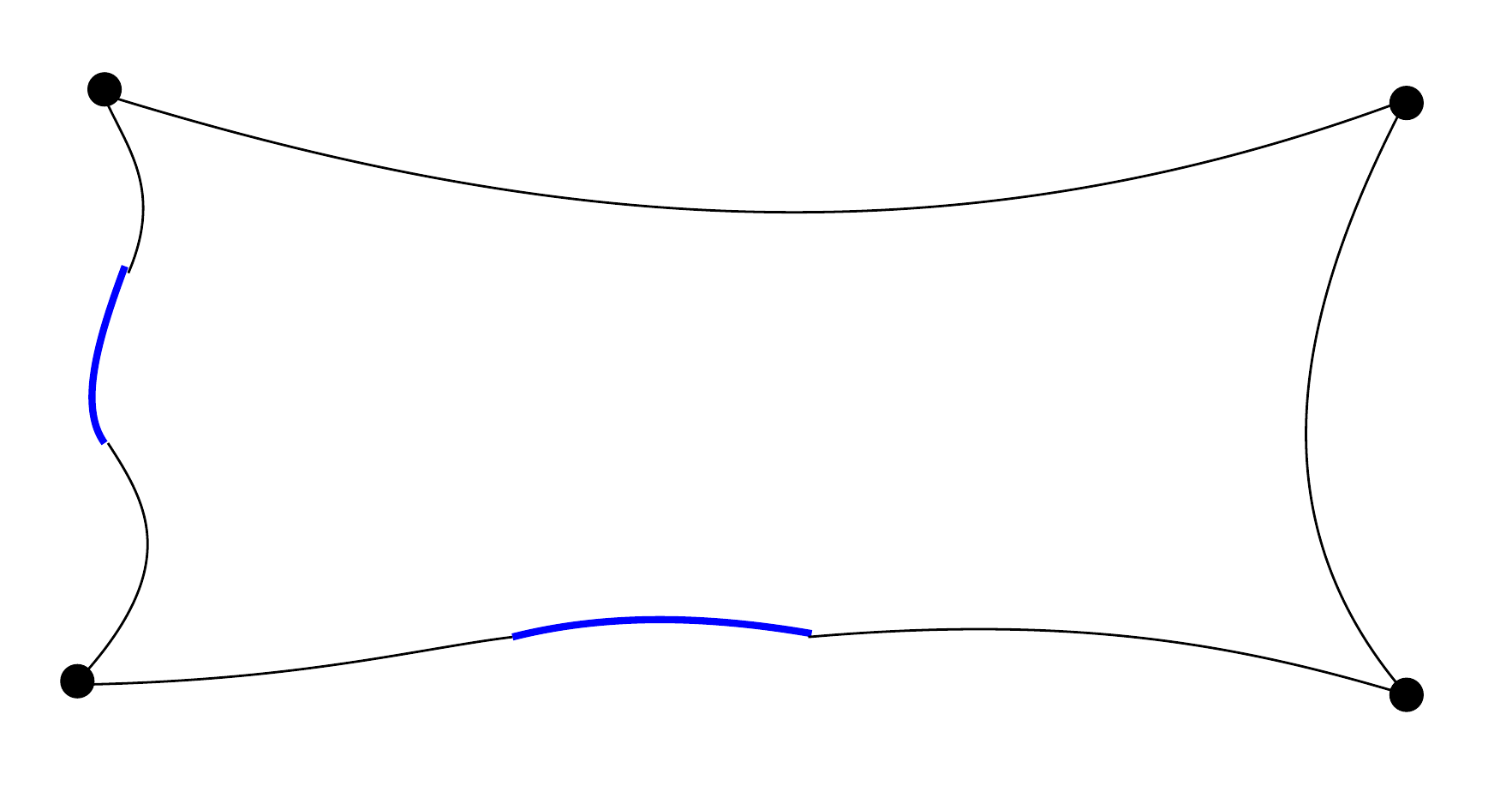
\caption{Case 2c1}
\label{casec1}
\end{minipage}%
\begin{minipage}{.49\textwidth}
  \centering
  \def\svgscale{0.43}
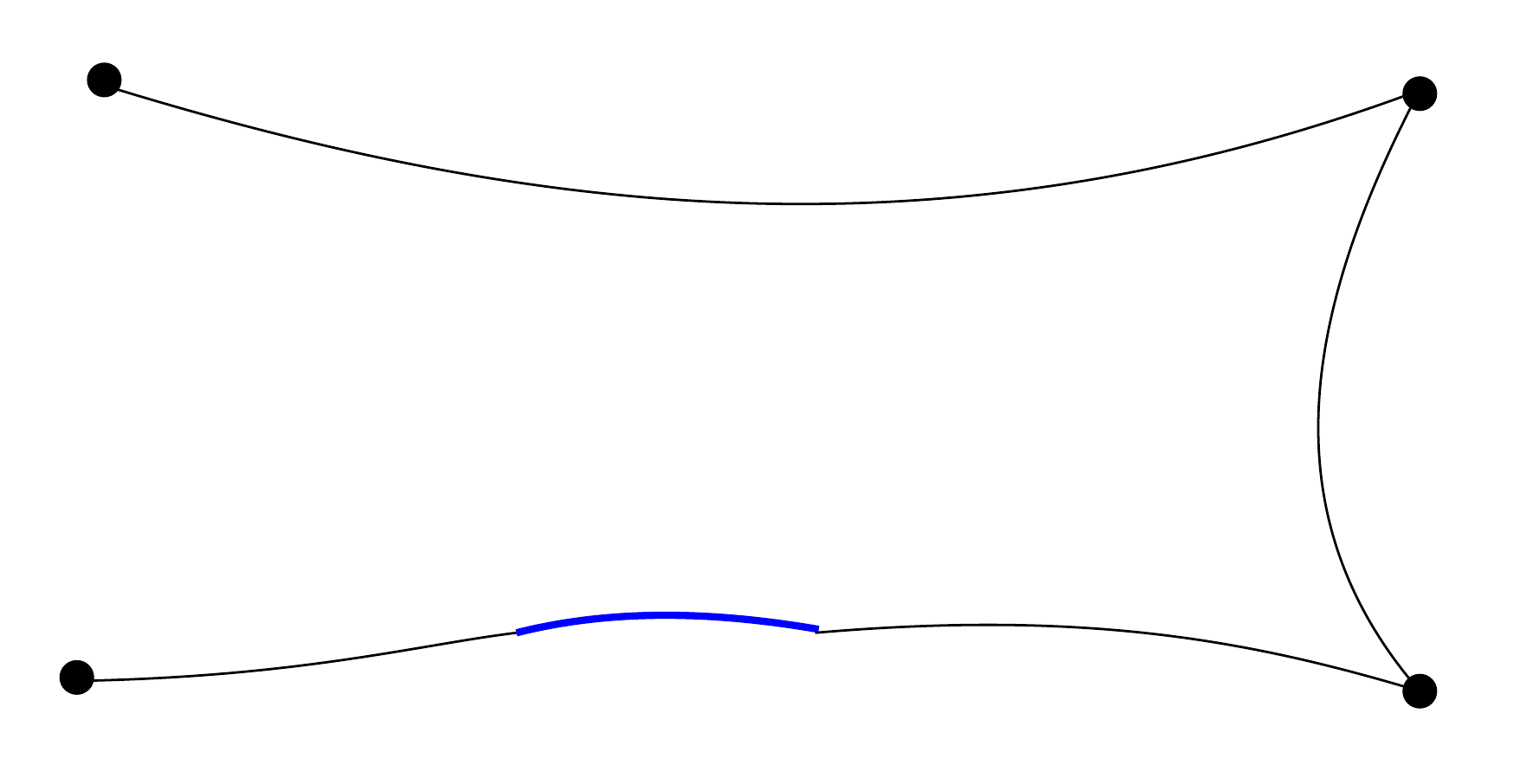
\caption{Case 2c2}
\label{casec2}
\end{minipage}
\end{figure}

	c2) If $c$ does not connect to any $H_i$--component of $[1,a]$, then by an argument symmetric to case c1), we obtain that there are at most $(R(\e)M)^2B$ options for $a$.

	d) Suppose $c$ connects to $H_i$--components $p$ of $[1,a]$, $d$ of $[a,ag]$, and $q$ of $[g,ag]$ (see Figure \ref{cased}).  Without loss of generality, we may assume that $a[1,g]=[a,ag]$.  Let $ac$ denote the image of the edge $c$ under the action of $a$.  Since $c$ is an $H_i$--component of $[1,g]$, it follows that $ac$ is an $H_i$--component of $[a,ag]$.
\begin{figure}
\centering
\def\svgscale{0.43}
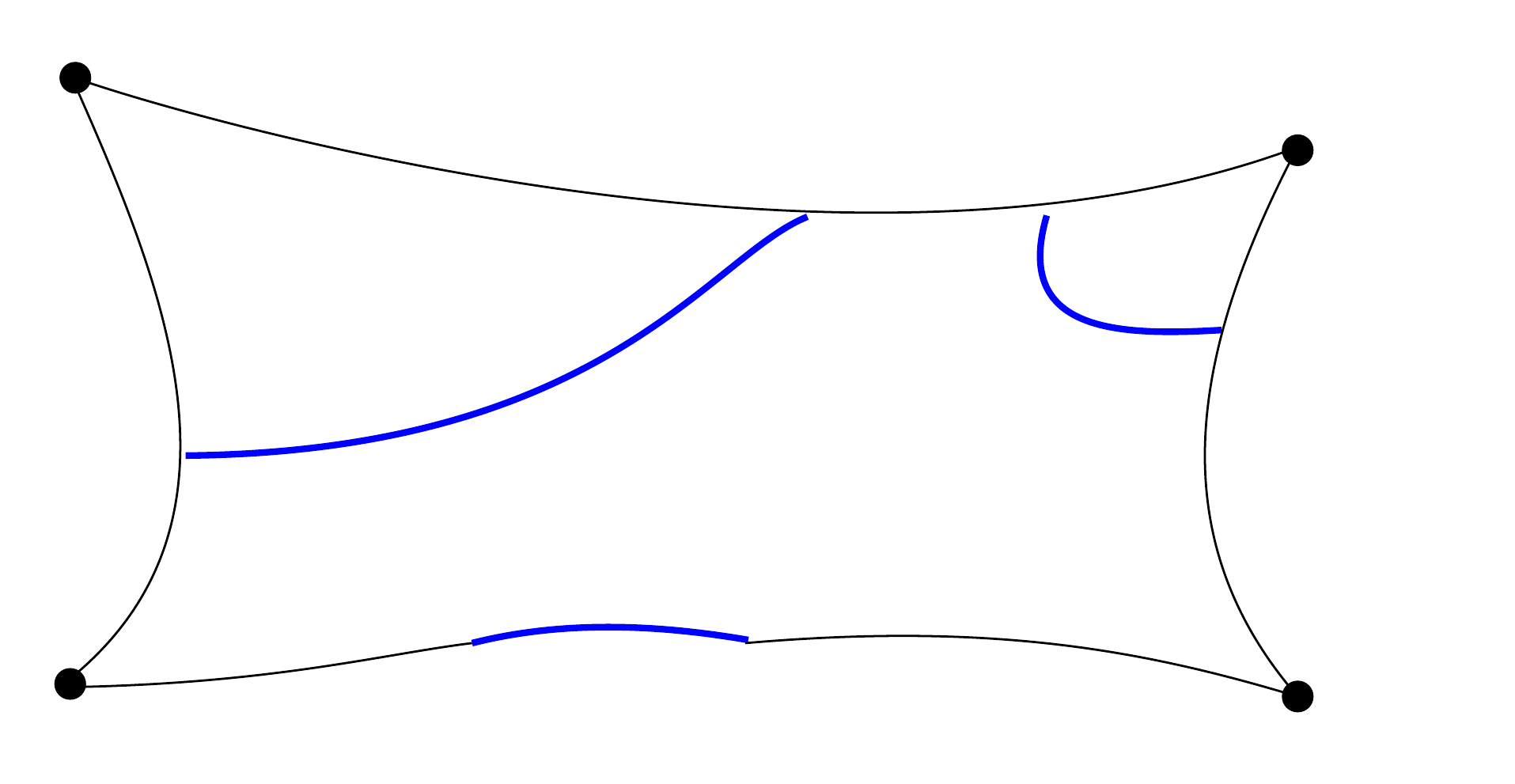
\caption{Case 2d}
\label{cased}
\end{figure}
Without loss of generality, we may assume that $ac$ belongs to $[a,d_-]$; the other case is symmetric.  Let $e_2$ be the edge labeled by an element of $H_i$ connecting $p_+$ to $d_-$.  By (\ref{clong}) and (\ref{M}), $$\d_{Y_i}(ac_-,ac_+)= \d_{Y_i}(c_-,c_+)>M>3D,$$ and thus Corollary \ref{Yibound} implies that $ac$ cannot be isolated in the quadrilateral $[1,a]\cup [a,ag]\cup [ag,g]\cup [g,1]$.

If $ac$ does not connect to $c$, then $ac$ must connect to an $H_i$--component  $f$ of $[1,a]$.
Since $\d_{Y_i}(ac_-,ac_+)>M,$ we can use the reasoning of Case 2a to reach a contradiction.  Therefore $c$ is connected to $ac$.  Since no two $H_i$--components of the geodesic $[a,ag]$ can be connected, we must have $ac=d$.
	
	Let $e_1,e_3,$ and $e_4$ be edges labeled by elements of $H_i$ connecting $c_-$ to $p_-$, $ac_+$ to $q_+$, and $q_-$ to $c_+$, respectively.  Applying Corollary \ref{Yibound} as in the previous cases, we obtain $\d_{Y_i}((e_j)_-,(e_j)_+)\leq 3D$ for $j=1,2,3,4$. 	
 By Lemma \ref{lem:boundb}, $\d_{Y_i}(p_-,p_+)\leq 4\e D$ and $\d_{Y_i}(q_-,q_+)\leq 4 \e D$. Thus, by the triangle inequality, $$\d_{Y_i}(c_-,ac_-)\leq 10\e D \quad \textrm{ and } \quad \d_{Y_i}(c_+,ac_+)\leq 10\e D.$$  It follows from (\ref{clong}) and (\ref{M}) that $$\d_{Y_i}(c_-,c_+)>M>R(10\e D).$$
 Therefore we can use the acylindricity of the action of $H_i$ on $\Gamma(H_i,Y_i)$ to conclude that there are at most $N(10\e D)$ choices for $(c_-)^{-1}ac_-$. Since there are at most $R$ choices for $c_-$, there are at most $RN(10\e D)$ choices for a.	

 In all cases, we have shown that there are at most $\max\{N(\e),RN(10\e D),(R(\e)M)^2B\}$ elements $a\in G$ satisfying (\ref{thinquad}), completing the proof.
 \end{proof}

We are now ready to prove Theorem \ref{main2}.

\begin{proof}[Proof of Theorem \ref{main2}]
That there is an injective, order preserving map $\iota_X\colon\mathcal H(H) \to\mathcal H(G)$ follows from Theorem \ref{AHO}.  It remains to show that elements of $\AHH$ are mapped to $\AHG$.
Since $H$ is hyperbolically embedded in $G$, there exists a subset $X$ of $G$ such that $H\h(G,X)$ and the action of $H$ on $\Gamma(G,X\sqcup H)$ is acylindrical (\cite[Theorem 5.4]{Osi16}).  The result then follows from Theorem \ref{genmain2}.
\end{proof}

\subsection{Applications}

We begin with the proof of Corollary \ref{SQ}. We will need the following two lemmas.

\begin{lem}\label{quotient}
Suppose that $Q$ is a quotient group of a group $P$. Then $\H(P)$ contains an isomorphic copy of $\H(Q)$.
\end{lem}

\begin{proof}
This is obvious if we use Proposition \ref{prop-SM} and think of hyperbolic structures in terms of cobounded actions: every cobounded action of $Q$ on a hyperbolic space gives rise to an action of $P$ and passing from actions of $Q$ to actions of $P$ preserves the relation $\preceq$ on group actions.
\end{proof}

\begin{lem}\label{SQ-lem}
For every finitely generated group $H$, there exists a quotient group $Q$ of $F_2$ such that $H\h Q$.
\end{lem}

This lemma can be seen as a particular case of the main result of \cite{AMO}, where it is proved for every non-elementary relatively hyperbolic group in place of $F_2$. For hyperbolic groups (in particular, for $F_2$) one could also use the embedding constructed in \cite{Ols95}, which, in fact, leads to hyperbolically embedded subgroups.

\begin{proof}[Proof of Corollary \ref{SQ}]
Let $G$ be an acylindrically hyperbolic group. By \cite[Theorem 2.24]{DGO}, $G$ contains a hyperbolically embedded subgroup isomorphic to $F_2\times K$ for some (finite) group $K$. By Lemma \ref{quotient}, Theorem \ref{genmain2}, and Lemma \ref{SQ-lem}, 
we have $$\H(H)\hookrightarrow \H(F_2)\hookrightarrow \H(F_2\times K)\hookrightarrow \H(G),$$ where $\hookrightarrow$ denotes the embedding of posets.
\end{proof}

Our next goal is to discuss the action of $Out(G)$ on $\H (G)$ and $\AHG$.

For every group $G$, one can define a natural action of $Aut(G)$ on $\GG$ by precomposing $G$-actions with elements of $Aut(G)$. More precisely, for every automorphism $\alpha$ of a group $G$ and any $[X]\in \GG$, we define
\begin{equation}\label{aX}
\alpha([X])=[\alpha(X)].
\end{equation}

\begin{lem} \label{Out} Let $G$ be a group.
\begin{enumerate}
\item[(a)] Equation (\ref{aX}) defines an order preserving action of $Aut(G)$ on $\GG$, which leaves $\HG$ and $\AHG$ setwise invariant.
\item[(b)] The induced action of $Inn(G)$ is trivial.
\end{enumerate}
\end{lem}

\begin{proof}
We first show that $\alpha ([X])$ is well-defined. If $X$ is a generating set of $G$, then clearly so is $\alpha (X)$. Furthermore, if $Y$ is another generating set of $G$ such that $X\preceq Y$, then we obviously have $\alpha (X)\preceq \alpha (Y)$. In particular, if $X\sim Y$ then $\alpha (X)\sim \alpha (Y)$. This implies that $\alpha ([X])$ is independent of the choice of a particular representative in the equivalence class $[X]$ and the map $\alpha\colon \GG \to \GG $ is order preserving. Note also that the map $g\mapsto \alpha(g)$ induces an isometry between vertex sets of the Cayley graphs $\Gamma (G,X)$ and $\Gamma (G,\alpha(X))$ and therefore $\Gamma (G,X)$ is hyperbolic if and only if $\Gamma (G, \alpha (X))$ is as well.

Assume that the action of $G$ on $\Gamma (G,X)$ is acylindrical. Fix any $\e>0$ and let $R=R(\e)$, $N=N(\e)$ be the corresponding acylindricity constants. Let $x,y\in G$ be any two elements such that $\d_{\alpha(X)} (x,y)\ge R$ and let $$A=\{ g\in G\mid \ \d_{\alpha(X)} (x,gx)\le \e, \; \d_{\alpha(X)}(y, gy) \le \e\}.$$ Then every $f\in \alpha^{-1}(A)$ satisfies $\d_X(\alpha^{-1}(x), f\alpha^{-1}(x))\le \e$ and $\d_X(\alpha^{-1}(y), f\alpha^{-1}(y))\le \e$. Since $\d_X (\alpha^{-1}(x), \alpha^{-1}(y))=\d_{\alpha(X)} (x,y)\ge R$, it follows from acylindricity of the action of $G$ on $\Gamma (G,X)$ that $|A|=|\alpha^{-1} (A)|\le N$. Thus the action of $G$ on $\Gamma (G,\alpha(X))$ is acylindrical.

Finally let $[X]\in \GG$ and let $g\in G$. Let $L=|g|_X$. Then for every $y\in g^{-1}Xg$ we have $|y|_X\le 2L+1$. In particular, $X\preceq g^{-1}Xg$ and, symmetrically, $g^{-1}Xg\preceq X$. Thus $\alpha ([X])=[X]$ for all $\alpha \in Inn(G)$.
\end{proof}

\begin{cor}
Let $G$ be a group and let $\hat\alpha \in Out(G)$. Let $\alpha $ be a preimage of $\hat \alpha $ in $Aut (G)$. Then the formula $\hat\alpha ([X])=[\alpha (X)]$ defines an order preserving action of $Out(G)$ on $\GG$ which leaves $\HG$ and $\AHG$ setwise invariant.
\end{cor}

We are now ready to prove the main result of this section, Theorem \ref{main6}. There are two main ingredients in the proof: the result about induced actions proved in the previous section (namely, Theorem \ref{genmain2}) and the classification of commensurating automorphisms of acylindrically hyperbolic groups obtained in \cite{AMS} (see also \cite{MO10} for the particular case of relatively hyperbolic groups).

\begin{proof}[Proof of Theorem \ref{main6}]
We denote by $K$ the kernel of the action of $Aut(G)$ on $\AHG$, i.e., the set of all automorphisms of $G$ that fix every element of $\AHG$. We want to show that $|K:Inn(G)|<\infty $ under the assumptions of part (a) and $K=Inn(G)$ under the assumptions of part (b).

Let $A=[X]\in \AHG$ be a non-elementary acylindrically hyperbolic structure. Since every $\alpha \in K$ fixes $A$, for every loxodromic element $g\in \mathcal L(A)$, we have $\alpha (g)\in \L(\alpha(A))=\mathcal L(A)$. Recall that two elements $a$, $b$ of a group $G$ are called \emph{commensurable} if some powers of $a$ and $b$ with non-zero exponents are conjugate. Our first goal is to prove the following claim for every $\alpha \in K$:

($\ast$) for every $g\in \mathcal L(A)$, the elements $g$ and $h=\alpha (g)$ are  commensurable.

Arguing by contradiction, assume that the elements $g$ and $h=\alpha (g)$ are not commensurable for some $g\in \mathcal L(A)$. Since $g$ and $h$ are loxodromic with respect to $A=[X]$, there exist virtually cyclic subgroups $E(g), E(h)\le G$ containing $g$ and $h$, respectively, such that $\{ E(g), E(h)\}\h (G, X)$ by \cite[Corollary 3.17]{Hull}. Let $A_1=[E(g)]$ and let $A_2$ be the equivalence class of a finite generating set $Y$ of $E(h)$. By Theorem \ref{genmain2}, $B=\iota_X(A_1,A_2)=[X\cup Y \cup E(g)]\in \AHG $. Note that $h\in \mathcal L(B)$ by part (b) of Theorem \ref{AHO}.  On the other hand, we obviously have $g\notin \mathcal L(B)$. Thus $\alpha $ does not fix $B$, which contradicts the assumption that $\alpha \in K$. The claim is proved.

Assume now that $G$ is finitely generated and let $K(G)$ denote its finite radical. By \cite[Corollary 7.4]{AMS}, for every automorphism $\alpha$ of $G$ satisfying ($\ast$), there exists a map (not necessarily a homomorphism) $\e \colon G\to K(G)$ and an element $w\in G$ such that $\alpha (g)= wg\e(g)w^{-1}$. In particular, the map $\alpha ^\prime \colon g\mapsto g\e(g)$ is an automorphism of $G$ and $\alpha ^\prime Inn(G)=\alpha Inn(G)$. Since $G$ is finitely generated, the automorphism $\alpha^\prime$ is completely defined by finitely many values; since $K(G)$ is finite, we conclude that there are finitely many cosets $\alpha Inn(G)$ of automorphisms $\alpha $ satisfying ($\ast$). Thus $|K:Inn(G)|<\infty $.

In case $K(G)=\{ 1\}$, every automorphism $\alpha$ of $G$ satisfying ($\ast$) is inner by \cite[Corollary 7.4]{AMS}. Therefore $K= Inn(G)$. Note that we do not use finite generation of $G$ in this case.
\end{proof}

The example below shows that the assumption that $G$ is finitely generated cannot be dropped from part (a) of Theorem \ref{main6}.

\begin{ex}\label{infker}
Let $G=F_\infty\times \mathbb Z_2$, where $F_\infty $ is the free group of countably infinite rank with a basis $x_1, x_2, \ldots $. Let $a$ be the non-trivial element of $\mathbb Z_2$. For any subset $A\subseteq \mathbb N$, the maps
$$
x_i\mapsto \left\{ \begin{array}{cc}
                       x_ia & {\rm if}\; i\in A \\
                       x_i & {\rm if}\; i\notin A
                     \end{array}\right.
$$
and $a\mapsto a$ extend to an automorphism $\alpha _A\in Aut(G)$. It is easy to see that $\alpha _A (g)\in \{ g, ga\}$ for all $g\in G$. It follows that $\alpha _A$ acts trivially on $\GG$ for every $A\subseteq \mathbb N$. On the other hand, we obviously have $\alpha _A Inn(G)\ne \alpha_B Inn (G)$ whenever $A$ and $B$ are distinct subsets of $\mathbb N$. Thus the kernel of the action of $Out(G)$ on $\GG$ (and hence on $\H(G)$ and $\AHG$) is infinite.
\end{ex}


\section{Loxodromic equivalence and rigidity}


\subsection{Loxodromic equivalence classes}

The goal of this section is to prove Theorem \ref{main3}.  We will first prove it for the particular case of purely loxodromic actions of $G=F(a,b)$, the free group on two generators, $a$ and $b$. Recall that, given some $[X]\in \AHG$, we denote by $[X]^{\AH}_{\mathcal L}$ the set of all acylindrically hyperbolic structures on $G$ with the same set of loxodromic elements (see Definition \ref{def-loxeq}).

\begin{prop} \label{prop:f2} Let $X=[\{a,b\}]\in\mathcal{AH}(F(a,b))$.  Then $\PN$ embeds into $[X]^{\AH}_{\mathcal L}$.
\end{prop}

The proof of the proposition relies on the existence of a collection of words in the alphabet $\{a,b\}$ satisfying certain properties.  A word $w\in F(X)$ is $l$--\emph{aperiodic} if it has no non-empty subwords of the form $v^l$ for any $v\in F(X)$.

Given an infinite collection $\mathcal Q$ of distinct $7$--aperiodic words in $F(a,b)$ whose lengths approach infinity, I. Kapovich in \cite{Kap} shows how to construct a generating set $Z$ of $F(a,b,c)$ such that $\Gamma(F(a,b,c),Z)$ is hyperbolic and the natural action of $F(a,b)$  on $\Gamma(F(a,b,c),Z)$ is acylindrical and purely loxodromic.  Moreover, for every $p\in\Gamma(F(a,b,c),Z)$, the orbit $F(a,b)p$ is quasi-convex.  As this construction will be important in the proof of Proposition \ref{prop:f2}, we review it here.

Given $v_n\in\mathcal Q$, let $w_n=v_n c\in F(a,b,c)$.   A nontrivial freely reduced word $z\in F(a,b,c)$ is a \emph{$\mathcal W$--word} if for some $n\geq 1$ and some integer $m\neq 0$, the word $z$ is a subword of $w_n^m$, and $z$ is \emph{positive} if it does not contain $a^{-1},$ $b^{-1}$, or $c^{-1}$.  The set $Z$ is defined to be the set of all positive $\mathcal W$--words in $F(a,b,c)$.  Since $\{a,b,c\}\subseteq Z$, the set $Z$ generates $F(a,b,c)$.

Kapovich also proved the following distance formula \cite[Lemma 3.5]{Kap}, which will be useful.

\begin{lem} \label{distformula}
Given a nontrivial freely reduced word $w\in F(a,b,c)$, $|w|_{Z}$ is equal to the smallest $k\geq 1$ such that there exists a $\mathcal W$--decomposition of length $k$, that is, a decomposition
\begin{equation}  \label{wdecomp}
w\equiv z_1\dots z_k,
\end{equation} so that $z_i$ is a $\mathcal W$--word for $1\leq i\leq k$.
\end{lem}

For the proof of Proposition \ref{prop:f2}, we will require that the collection $\mathcal Q$ satisfies a small cancellation condition which is slightly stronger than $C'(1/6)$ and which we now introduce.

\begin{defn}\label{C*}
A set of words $\mathcal Q$ in an alphabet $X$ satisfies the \emph{$C^*(\lambda)$ condition} if the following two conditions hold.
\begin{enumerate}
\item[(a)] Let $v,w\in \mathcal Q$ be two distinct words and let $u$ be a subword of two cyclic shifts of $v$ and $w$. Then $|u| <\lambda \min \{ | v|, |w|\}$.
\item[(b)] Let $v\in \mathcal Q$. Then any word of length $| u| \ge \lambda |v|$ occurs in every cyclic shift of $v$ at most once.
\end{enumerate}
\end{defn}

The following lemma shows that the collection of words required for the proof of Proposition \ref{prop:f2} exists.

\begin{lem}\label{lem:setQ}
There exists a collection  $\mathcal Q=\{v_1,v_2,\dots\}$ of words in the alphabet $\{a,b\}$ satisfying the following properties:
\begin{enumerate}
\item[(a)] each $v_i$ is $7$--aperiodic;
\item[(b)] $|v_i|\to\infty$ as $i\to\infty$;
\item[(c)] each subset $\mathcal Q_n=\{v_n,v_{n+1},\dots\}$ satisfies $C^*(3/n)$.
\end{enumerate}
\end{lem}

\begin{proof}
We note that the existence of a collection of words satisfying (b) and (c) was proven in \cite[Proposition 3.3]{EO}, following an idea of Olshanskii \cite{Ols99}. We will show that a slight modification of this construction also satisfies (a).

Let $f(n)$ be the number of positive $6$--aperiodic words in the alphabet $\{a,b\}$ of length $n$ that start and end with $b$.  It is shown in \cite[Lemma 3]{Ols99} that $f$ is an exponential function.  Fix $k_0> 6$ such that for every $k\geq k_0$, $f(k-6)\geq k$.

Let $\mathcal X(k)=\{X_{k,1},\dots, X_{k,f(k)}\}$ be the set of all distinct $6$--aperiodic words of length $k$ over the alphabet $\{a,b\}$ that start and end with $b$ .  For every $k\geq k_0$,  consider the word
 \begin{align*}
 v_k=(a^6X_{k-6,1})(a^6X_{k-6,2})\cdots(a^6X_{k-6,k}).
 \end{align*}
  Let $\mathcal Q=\{ v_k\mid k\geq k_0\}$ and for each $n\geq k_0$, let $\mathcal Q_n=\{v_j\mid j\geq n\}$.   It is straightforward to show that $v_k$ is $7$--aperiodic for all $k\geq k_0$ (see \cite[Lemma 1.2]{Ols16}), so (a) is satisfied.  It is also clear that $|v_k|=k^2$, where $|\cdot |=|\cdot |_{\{a,b\}}$, and so (b) is satisfied as well.

The proof of (c) is essentially that of \cite[Lemma 3.6]{EO}, although with a different collection of words.  Since $\cal Q$ consists of positive words, we do not need to consider inverses of words from $\cal Q$.  Let $n\geq k_0$ and consider $v_j\in\mathcal Q_n$ and a subword $w$ of a cyclic shift $v_j'$ of $v_j$ such that $|w|\geq (3/n)|v_j'|\geq (3/j)|v_j'|=3j>2j+6.$  Then $v_j$ must contain a subword of the form \begin{equation}\label{subword} a^6X_{j-6,i}a^6,\end{equation} where $X_{j-6,i}\in\mathcal X(j-6)$.  Such a subword can only occur once in $v_j$, since elements of $\mathcal X(j-6)$ are distinct.

 Next suppose $v_i,v_j\in \mathcal Q_n$ are two distinct words and $w$ is a common initial subword of cyclic shifts of $v_i$ and $v_j$, respectively, such that $|w|>(3/n)\min\{|v_i|,|v_j|\}.$  By a similar argument, $w$ must contain a subword of the form (\ref{subword}) for $j=\min\{|v_i'|,|v_j'|\}$.  However, such a subword occurs in a unique word in $\mathcal Q_n$, which completes the proof of (c) and the lemma.
\end{proof}

We are now ready to prove Proposition \ref{prop:f2}.

\begin{proof}[Proof of Proposition \ref{prop:f2}]  Let $\mathcal Q$ be the infinite collection of words in the alphabet $\{a,b\}$ provided by Lemma \ref{lem:setQ}, and let $S$ be an infinite subset of $\mathcal Q$.  We will construct a generating set of $F(a,b,c)$ in an analogous way to \cite{Kap}.  Given $v_n\in S$, we let $w_n=v_nc\in F(a,b,c)$.  Let us call a nontrivial freely reduced word $z\in F(a,b,c)$ a \emph{$\mathcal W_S$--word} if for some $n\geq 1$ and some $m\neq 0$, the word $z$ is a subword of $w_n^m=(v_nc)^m$ for some $v_n\in S$.  Let $Z_S$ be set of all positive $\mathcal W_S$--words in $F(a,b,c)$.  Then $a,b,c\in Z_S$, and so $Z_S$ generates $F(a,b,c)$.  Note that since $S$ consists only of positive words, every word in $F(a,b)$ representing a word in $S$ is an element of $Z_S$.  Since $S$ is an infinite collection of distinct $7$--aperiodic words in $\{a,b\}$ whose lengths approach infinity, by \cite[Theorem A]{Kap}, the graph $\Gamma(F(a,b,c),Z_S)$ is hyperbolic and the action of $F(a,b)$ on $\Gamma(F(a,b,c),Z_S)$ is acylindrical and purely loxodromic. Moreover, for any $p\in \Gamma(F(a,b,c),Z_S)$, the orbit $F(a,b)p$ is quasi-convex.

If $w\in F(a,b)$ is a nontrivial freely reduced word and $w\equiv z_1\cdots z_k$ is a $\mathcal W_S$--decomposition of $w$ as in (\ref{wdecomp}), then each $z_i$ is a subword of $v_n$ or $v_n^{-1}$ for some $v_n\in S$ (\cite[Lemma 3.3]{Kap}).  Moreover, it is straightforward to check that if $w\in F(a,b)$ is a positive word, then each $z_i$ is a subword of $v_n$ for some $v_n\in S$.

The action $F(a,b)\curvearrowright \Gamma(F(a,b,c),Z_S)$ need not be cobounded.  However, by Proposition \ref{cobdd} there exist $[X_S]\in\mathcal{AH}(F(a,b))$ such that $$F(a,b)\curvearrowright \Gamma(F(a,b,c),Z_S)\sim_w F(a,b)\curvearrowright \Gamma(F(a,b),X_S).$$  Moreover, $[X_S]$ is a purely loxodromic structure.

For a technical reason, it is convenient to work with the poset $\mathcal P(\omega)/Fin$ instead of $\PN$. Recall that $\mathcal P(\omega)/Fin$ is the poset of equivalence classes of subsets of $\mathbb N$; two subsets $A,B\subseteq N$ are \emph{equivalent} if $|A\bigtriangleup B|<\infty$, and $[A]\le [B]$ if $|A\setminus B|<\infty $. Note that $\PN$ embeds in $\mathcal P(\omega)/Fin$. Indeed, let $M=\bigsqcup\limits_{i=1}^\infty N_i$, where $N_i=\mathbb N$, and let $j\colon M\to \mathbb N$ be a bijection. Given $S\subseteq \mathbb N$, let $S_i$ denote the copy of $S$ in $N_i$.  Then the map $S\mapsto j(\bigsqcup\limits_{i=1}^\infty S_i)$ is a required embedding. Thus it suffices to show that $\mathcal P(\omega)/Fin$ embeds in $[X]^\AH_{\L}$.

Identifying $\mathcal Q$ with $\mathbb N$, let $$f\colon\mathcal P(\omega)/Fin \to \mathcal{AH}(F(a,b))$$ be the map defined by $$[S]\mapsto [X_S].$$

We first show that $f$ is well-defined.  Recall that two subsets $S,T\subset \mathbb N$ are equivalent if $|S\triangle T|<\infty$, and $S\leq T$ if $|S\setminus T|<\infty$.  Suppose $S,T\subset \mathbb N$ are equivalent.  It suffices to consider the case $S=T\cup\{s\}$.  It is clear that  $F(a,b)\curvearrowright \Gamma(F(a,b,c),Z_S)\preceq F(a,b)\curvearrowright \Gamma(F(a,b,c),Z_T)$.  Let $g\in F(a,b)$ and $K=|g|_{Z_S}$.  By Lemma \ref{distformula}, $g$ has a $\mathcal W_S$--decomposition of length $K$, that is, $g\equiv p_1\dots p_K$.  Since $g\in F(a,b)$, each $p_j$ is a subword of $v_{n_j}$ or $v_{n_j}^{-1}$ for some $v_{n_j}\in S$.  If $v_{n_j}\in T$, then $p_j$ is a $\mathcal W_T$--word and so $|p_j|_{Z_T}=1$.  Otherwise, $p_j=s$ or $s^{-1}$, and so 
$|g|_{Z_T}\leq |s|_{Z_T}|g|_{Z_S}$.  Since $|s|_{Z_T}$ is a constant, it follows from Definition \ref{def-poset} that $F(a,b)\curvearrowright \Gamma(F(a,b,c),Z_T)\preceq F(a,b)\curvearrowright \Gamma(F(a,b,c),Z_S)$.  Thus the actions $F(a,b)\curvearrowright \Gamma(F(a,b,c),Z_T)$ and $F(a,b)\curvearrowright \Gamma(F(a,b,c),Z_S)$ are equivalent, and hence $[X_S] \sim [X_T]$ by Lemma \ref{e-we}.

We next show that $f$ is injective.  Suppose $S,T\subset\mathcal Q$ satisfy $|S\triangle T|=\infty$, so that $S$ and $T$ are not equivalent in $\mathcal P(\omega)/Fin$.  Without loss of generality, assume that $S\setminus T$ is infinite and consider distinct words $u_i\in S\setminus T$ for $i=1,2,\dots$.  Suppose $F(a,b)\curvearrowright \Gamma(F(a,b,c),Z_T)\preceq F(a,b)\curvearrowright \Gamma(F(a,b,c),Z_S)$.   Then by Lemma \ref{EA}, there exists a constant $C$ such that $|g|_{Z_T}\leq C|g|_{Z_S} +C$ for all $g\in F(a,b)$.  In particular, $|u_i|_{Z_T}\leq 2C$ for all $i$.  By Lemma \ref{distformula}, for each $i$ there is some $k\leq 2C$ such that $u_i$ has a $\mathcal W_T$--decomposition of length $k$, that is, $u_i\equiv p_1\dots p_k$.  Each $p_j$ is a subword of $u_i$, and since $u_i\in F(a,b)$ is a positive word, each $p_j$ is also a subword of some $v_{n_j}\in T$.  Moreover, at least one $p_j$ must have length at least $\frac{1}{k}|u_i|_{Z_T}$.    
The words $u_i$ and $v_{n_j}$  belong to $\mathcal Q_{\max\{n_j,i\}}$, and so for large enough $i$ these words satisfy $C^*(\frac{1}{2C})$.  Therefore $v_{n_j}=u_i$ for large enough $i$.  However, this contradicts the fact that $u_i\not\in T$.  Thus $X_T\not\sim X_S$.

Finally we show that $f$ is order-reversing.  Suppose $[T]\leq[S]$ in $\mathcal P(\omega)/Fin$. Changing the representatives of the equivalence classes if necessary, we can assume that $T\subseteq S$ and the inequality follows.

To see that $\mathcal P(\omega)/Fin$ embeds in $\mathcal{AH}(F(a,b))$, we precompose $f$ with the order-reversing automorphism of $\mathcal P(\omega)/Fin$ defined by $A\mapsto A^c$. Note that for every $[S]\in \mathcal P(\omega)/Fin$, $[X_S]$ is a purely loxodromic structure by \cite[Theorem A]{Kap}, and so $X_S\in [X]^\AH_{\L}$, as desired.
\end{proof}

To prove Theorem \ref{main3} we will need the following proposition, which describes how the set of loxodromic elements of an acylindrically hyperbolic structure changes after adding a hyperbolically embedded subgroup to the generating set.

\begin{prop} \label{lem:lox}
Let $G$ be an acylindrically hyperbolic group and let $[X]\in\AHG$.  Suppose a subgroup $H\h (G,X)$.  If $g\in\L([X])$, then either $g$ is conjugate to an element in $H$ or $g$ is a loxodromic element with respect to $[X\cup H]$.\end{prop}

We break the proof into several lemmas. The next three results are stated under the hypotheses of Proposition \ref{lem:lox}.

\begin{lem}\label{lem:ellconjlength}
There exists a constant $K$ such that every element which is elliptic with respect to $[X\cup H]$ is conjugate to an element of $(X\cup H)$--length at most $K$.
\end{lem}

\begin{proof}
Let $\delta$ be the hyperbolicity constant of $\Gamma(G,X\sqcup H)$.  The element $g$ is elliptic in the action $G\curvearrowright \Gamma(G,X\sqcup H)$, and so $g$ has an orbit with diameter $4\delta+1$ by \cite[Corollary 6.7]{Osi15}.  Let $v$ be a representative of this orbit.  After increasing the constant to $2(4\delta+1)$, we may assume that $v$ is a vertex, i.e., $v\in G$.  Then the orbit $\langle v^{-1}gv\rangle \cdot 1$ is contained in a $2(4\delta+1)$--ball around $1$.  Setting $K=2(4\delta+1)$ yields the result.
\end{proof}

Let us call an element $g\in G$ \emph{nice} if it can be represented by a word $w\in\langle X\cup H\rangle$ such that every cyclic permutation of $w$ is geodesic.  Let $K$ be the constant from Lemma \ref{lem:ellconjlength}.  We say an element $f\in G$ is \emph{super nice} if it is nice and satisfies $|f^n|_{X\cup H}\leq K+2K^2$ for all $n$.

\begin{lem}\label{lem:Xlength}
There exists a constant $C$ such that every element $g\in G$ that is super nice and is not conjugate to an element of $H$ satisfies $|g|_X\leq C$.
\end{lem}

\begin{proof}
Since $g$ is super nice, $|g^n|_{X\cup H}\leq K+2K^2$.  Fix an integer $N$ such that $N>K+2K^2$.  Let $w$ be a representative of $g$ in $\langle X\cup H\rangle$ such that every cyclic permutation of $w$ is geodesic in $\Gamma(G,X\sqcup H)$.  Such a $w$ exists since $g$ is nice.  Then there is an $(N+1)$--gon $P$ in $\Gamma(G,X\sqcup H)$ consisting of $N$ sides that are each labeled by $w$, which we call  $w$--\emph{sides}, and a single side which is a geodesic of length at most $K+2K^2$ connecting $1$ to $g^N$.

We will use $P$ to bound the $X$--length of all $H$--components of $w$.  There are three cases to consider, depending on how the $H$--components of the $w$--sides connect.

{\bf Case 1.}  Suppose all $H$--components of all $w$--sides are isolated in $P$.  Then by Corollary \ref{Yibound}, there is a constant $D$ such that each $H$--component of $w$ has $X$--length at most $D(N+1)$.

{\bf Case 2.} Suppose no $H$--component of any $w$--side is connected to an $H$--component of another $w$--side.  Then any $H$--component of a $w$--side which is not isolated must connect to the single side of $P$ which is not labeled by $w$, and there are at most $K+2K^2$ such connections.  Since there are at least $K+2K^2+1$ $w$--sides, there is at least one copy of each $H$--component of $w$ that is isolated in $P$.   Therefore, as in Case 1, each $H$--component of $w$ has $X$--length at most $D(N+1)$.

\begin{figure}
\centering
\def\svgscale{0.7}
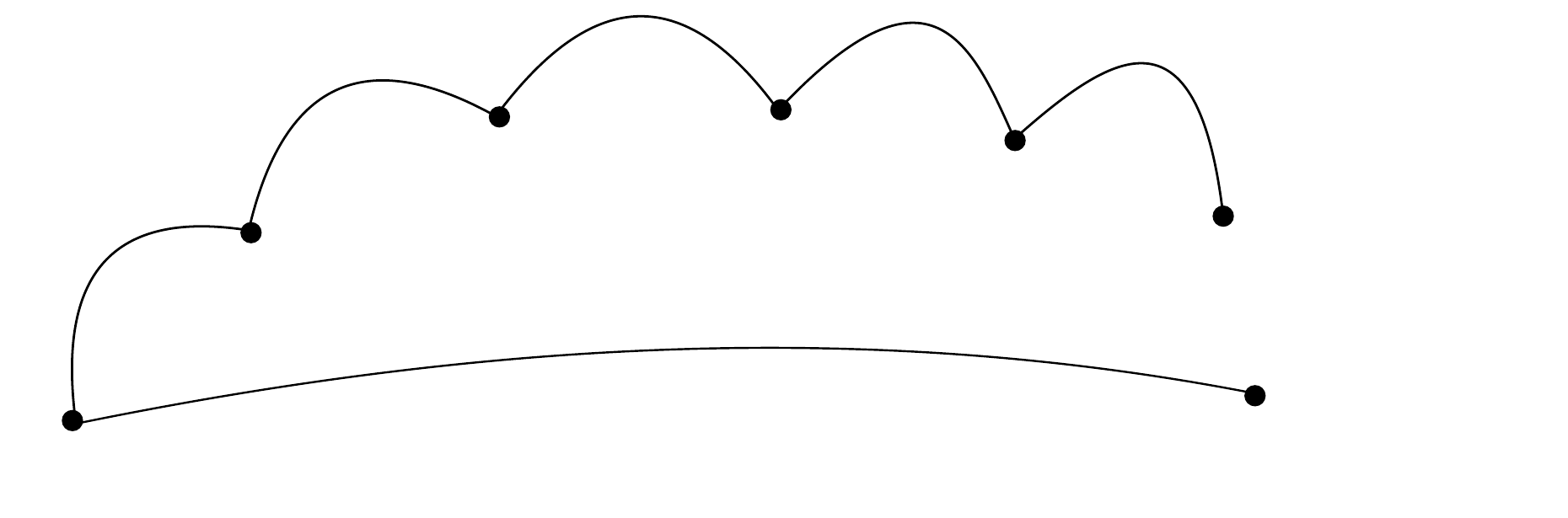
\caption{Case 3a}
\label{case673a}
\end{figure}

{\bf Case 3.}  Suppose an $H$--component $a$ of one $w$--side connects to an $H$--component $b$ of another $w$--side.

{\bf Case 3a.}  Suppose $a$ and $b$ are $H$--components of consecutive $w$--sides (See Fig. \ref{case673a}). If $a=w$, then $\d_{X\cup H}(w_-,w_+)=1$ and so $g$ is conjugate to an element $H$, contradicting our assumption on $g$.  Thus $a$ is a proper subsegment of $w$.  It follows that there is an edge $e$ labeled by an element $h\in H$ connecting $a^+$ to $b_-$.  Then there are two paths from $a_-$ to $e_+$, one a subpath of $P$ and one labeled by $ah$.  The length of the latter path is strictly less than the length of the former, so the former path is not geodesic. However, the label of the subpath of $P$ is a cyclic permutation of $w$, so this contradicts our choice of representative $w$.

{\bf Case 3b.}  Suppose $a$ and $b$ are $H$--components of two $w$--sides which are not consecutive (See Fig. \ref{case673b}).  Let $A$ be the number of $w$--sides between these two sides, so that $1\leq A\leq N-2$.  Without loss of generality, we may assume that every $H$--component of every $w$--side between the two connected $w$--sides is isolated.  To see this, notice that if not, we can choose the ``innermost" connected $H$--components instead; since no two consecutive $w$--sides are connected, this ensures that $A$ is still at least $1$.
Therefore each $H$--component of one of these sides is isolated in an $(A+3)$--gon formed by $A$ $w$--sides, the edge labeled by an element of $H$ connecting $a^+$ to $b_-$, and a subgeodesics of each of the $w$--sides containing $a$ and $b$.  Since $A\leq N-2$, every $H$--component of $w$ has $X$--length at most $D(N+1)$.

\begin{figure}
\centering
\def\svgscale{0.7}
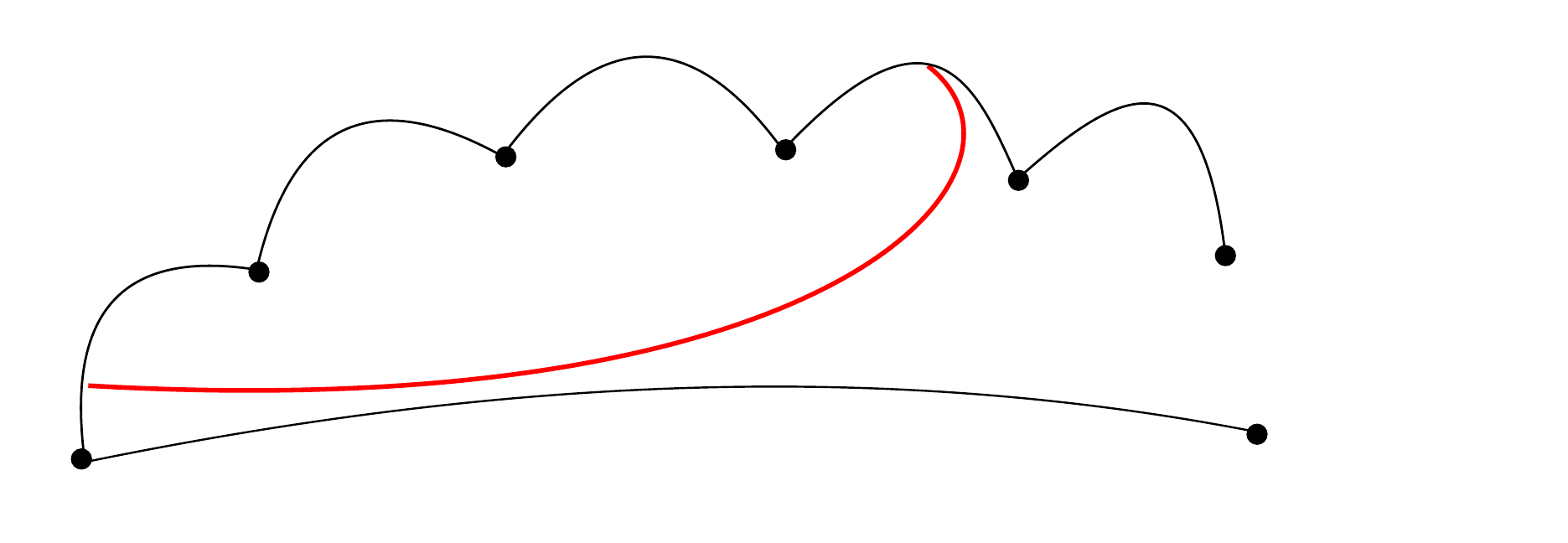
\caption{Case 3b}
\label{case673b}
\end{figure}

In all cases, we have bounded the $X$--length of every $H$--component of $w$ by $D(N+1)$.  Moreover, since $g$ is super nice, $|g|_{X\cup H}\leq K+2K^2$, and so
\[|g|_{X}\leq (K+2K^2)D(N+1).\]
Letting $C=(K+2K^2)D(N+1)$ completes the proof.
\end{proof}

\begin{cor}\label{cor:conjXlength}
There exists a constant $C$ such that every element $g\in G$ which is elliptic with respect to $[X\cup H]$ is conjugate to an element of $X$--length at most $C$.
\end{cor}

\begin{proof}
First observe that every element $g\in G$ is conjugate to a nice element $g'\in G$ by an element of $(X\cup H)$--length at most $|g|_{X\cup H}^2$.
To see this, note that to any given representative $w$ of $g$, to produce a nice element $g'$ from $g$, one cyclically permutes $w$ and then reduces, if possible.  Each cyclic permutation corresponds to conjugation of $g$ by a word of $(X\cup H)$--length at most $|g|_{X\cup H}$, and one needs to perform at most $|g|_{X\cup H}$ cyclic permutations.

Combining Lemma \ref{lem:ellconjlength} with the above observation shows that every element $g\in G$ which is elliptic with respect to $[X\cup H]$ is conjugate to a super nice element of $G$.  The result then follows by Lemma \ref{lem:Xlength}.
\end{proof}

\begin{proof}[Proof of Proposition \ref{lem:lox}]
Suppose $g$ is is not conjugate to an element of $H$ and is elliptic with respect to $[X\cup H]$.  We will show that $g$ must be elliptic with respect to $[X]$, which will contradict our assumption on $g$.

First, for any $n$, applying Corollary \ref{cor:conjXlength} to the element $g^n$ yields $\tau_X(g^n)\leq C$, where $C$ is independent of $n$.  Indeed, $g^n$ is elliptic with respect to $[X\cup H]$, since $g$ is, and thus it is conjugate to an element $g'$ such that $|g'|_X\leq C$.  The translation length of an element cannot exceed its length, and so $\tau_X(g')\leq C$.  Therefore $\tau_X(g^n)\leq C$, as translation length is constant on conjugacy classes.

On the other hand,  $$\tau_X(g)=\frac1n\tau_X(g^n)\leq \frac{C}{n},$$ for all $n$.  Letting $n\to \infty$ we have $$\tau_X(g)=0,$$  and so $g$ is not loxodromic with respect to $[X]$.  Since $[X]\in\AHG$, it follows that $g$ must be elliptic with respect to $[X]$.
\end{proof}

We are now ready to prove the general case, Theorem \ref{main3}.

\begin{proof}[Proof of Theorem \ref{main3}]
Let $G$ be an acylindrically hyperbolic group, and let $[Y]\in\AHG$ be non-elementary. By Theorem \ref{strongheF2}, there is a subgroup $H\simeq F_2\times K(G)$ which is strongly hyperbolically embedded in $G$ with respect to $Y$.  We naturally identify each $f\in F_2$ and $k\in K(G)$ with $(f,1)\in H$ and $(1,k)\in H$, respectively.

We first describe how to produce an element of $\mathcal{AH}(H)$ from an element of $\mathcal{AH}(F_2)$.
Let $\phi$
be the map defined by $$[X]\mapsto [X\cup K(G)].$$  It is clear that $\phi$ is an order-preserving injective map from $\mathcal{AH}(F_2)$ to $\AHH$.  It also clear that $(f,k)$ is a loxodromic element of $H$ with respect to $[X\cup K(G)]\in \AHH$ if and only if $f$ is a loxodromic element of $F_2$ with respect to $[X]\in\mathcal{AH}(F_2)$.

Let $\iota_Y$ be the map defined by (\ref{iotadef}). By Theorems \ref{AHO} and \ref{main2}, $\iota_Y$ is an order-preserving injective map $\mathcal{AH}(H)\hookrightarrow\mathcal{AH}(G)$.  Recall that given $[A]\in\mathcal{AH}(H)$, $\iota_Y([A])=[A\cup Y]$.  By Theorem \ref{AHO} (b), $H\curvearrowright \Gamma(H,A)\sim_w H\curvearrowright \Gamma(G,A\cup Y)$.  Therefore an element $h\in H$ is a loxodromic element of $H$ with respect to $[A]\in\AHH$ if and only if $h$ is a loxodromic element of $G$ with respect to $[A\cup Y]\in\AHG$.

Recall that $H=\langle a,b,K(G)\rangle$, where $\langle a,b\rangle \cong F_2$, and let $X=\{a,b\}$.   By Proposition \ref{prop:f2}, there is an embedding $\psi$ of $\mathcal P(\omega)$ into $[X]^\AH_{\L}$.  Thus there is an embedding of posets $$\iota_Y\circ\phi\circ \psi\colon P(\omega)\hookrightarrow \AHG.$$ We want to show that the image of $\iota_Y\circ \phi\circ \psi$ is contained in $[Y]^\AH_{\L}$.

Note that since $X\cup K(G)$ is finite, we have $(\iota_Y\circ\phi)([X])=[Y]$.  Thus it remains only to show that if $[X_1]\sim_{\L}[X]$ in $\mathcal{AH}(F_2)$, then $[X_1\cup K(G)\cup Y]\sim_{\L}[Y]$ in $\AHG$.

Since $ [X_1\cup K(G)\cup Y]\preccurlyeq [Y]$, it follows that every element of $G$ which is loxodromic with respect to $ [X_1\cup K(G)\cup Y]$ is loxodromic with respect to  $[Y]$. To show the opposite inclusion, suppose $g\in G$ is loxodromic with respect to $[Y]$ but is not conjugate to an element of $H$.  Then $g$ is loxodromic with respect to $[Y\cup H]$ by Lemma \ref{lem:lox}.  Since $X_1\cup K(G)\cup Y\subseteq Y\cup H$,
it follows that $g$ is loxodromic with respect to $[X_1\cup K(G)\cup Y]$.

An element $(f,k)\in H$ is a loxodromic element of $G$ with respect to $[X_1\cup K(G)\cup Y]$ exactly when $f$ is a loxodromic element of $F_2$ with respect to $[X_1]$.  Since $[X]$ is purely loxodromic and $[X_1]\sim_\L [X]$, every $f\in F_2\setminus\{1\}$ is a loxodromic element of $F_2$ with respect to $[X_1]$.  By Proposition \ref{strongheF2}, these are precisely the elements of $H$ which are loxodromic elements of the group $G$ with respect to $[Y]$, completing the proof.
\end{proof}

\subsection{Coarsely isospectral actions}

Recall that two actions $A=G\curvearrowright S$ and $B=G\curvearrowright T$ of the same group $G$ are said to be \emph{coarsely isospectral} if for every sequence of elements $(g_i)_{i\in \mathbb N} \subseteq G$, we have $$\lim\limits_{i\to\infty}\tau_{A}(g_i)=\infty \;\; \Longleftrightarrow \;\; \lim\limits_{i\to\infty}\tau_{B}(g_i)=\infty, $$ where $\tau _A(g_i)$, $\tau_B(g_i)$ are the translation numbers of $g_i$ defined as in (\ref{def-trn}).

We first verify that coarse isospectrality is invariant under weak equivalence (see Definition \ref{def-we}).

\begin{lem}\label{ciinv}
Any two weakly equivalent actions of the same group are coarsely isospectral.
\end{lem}

\begin{proof}
Let $A=G\curvearrowright S$ and $B=G\curvearrowright T$ be weakly equivalent actions. 
By Definition \ref{def-poset}, the inequality $A\preceq B$ means that there exist $C>0$ and $s\in S$, $t\in T$ such that $$\d_S (s,fs) \le C\d_T(t, ft)+C$$ for all $f\in G$. Let $g\in G$. Applying the previous inequality to powers of $g$, we obtain
$$
\tau_A(g)=\lim\limits_{n\to \infty} \frac{\d_S(s, g^ns)}{n} \le \lim\limits_{n\to \infty} \frac{C\d_T(t, g^nt)+C}{n}= C\lim\limits_{n\to \infty} \frac{\d_T(t, g^nt)}{n}=C\tau_B(g).
$$
Thus, for every sequence of elements $(g_i)$ of the group $G$, $\tau_A(g_i)\to \infty $ as $i\to \infty$ implies $\tau_B(g_i)\to \infty $ as $i\to \infty$. Similarly, the inequality $B\preceq A$ yields the converse implication.
\end{proof}

Lemma \ref{ciinv} allows us to define the notion of coarsely isospectral hyperbolic structures on a group $G$ as follows.

\begin{defn}\label{def-cis}
Two hyperbolic structures $[X], [Y]\in \H(G)$ are called \emph{coarsely isospectral} if the actions $G\curvearrowright \Gamma(G,X)$ and $G\curvearrowright \Gamma (G,Y)$ are coarsely isospectral.
\end{defn}

Our next goal is to prove Theorem \ref{main4} and Corollary \ref{cormain4}. To this end we will need several technical lemmas. The first one is well-known (see, for example, the proof of Proposition 21 and Chapter 2 of \cite{GH}).
\begin{lem}\label{GHs23}
Let $S$ be a $\delta$--hyperbolic space. Then for any $x,y,z,t\in S$, we have
\begin{equation}\label{Gprod}
(x|z)_t\ge \min \{ (x|y)_t, (y|z)_t\} -8\delta.
\end{equation}
\end{lem}

As before, given two points $x$, $y$ in a geodesic metric space $S$ we denote by $[x,y]$ a geodesic segment connecting $x$ to $y$.

The next result is similar to the result stated in \cite[7.2.C]{Gro}; proofs of analogous results can be found in \cite{GH} (see Theorem 16 in Chapter 5) as well as in \cite[Lemma 21]{Ols91}.

\begin{lem}\label{lemx0xn}
Let $x_0, x_1, \ldots, x_n$ be a sequence of points in a $\delta$--hyperbolic space $S$ such that
\begin{equation}\label{Gp}
(x_{i-1}| x_{i+1})_{x_i} \le C, \;\;\; 1\le i \le n-1,
\end{equation}
for some constant $C$ and
\begin{equation}\label{length}
\d_S (x_{i-1}, x_{i})> 2C+16\delta, \;\;\; 1\le i\le n.
\end{equation}
Then we have
\begin{equation}\label{dx0xn}
\d_S (x_0, x_n) \ge \sum\limits_{i=1}^n \d_S (x_{i-1}, x_{i}) - 2(n-1)(C+8\delta).
\end{equation}
\end{lem}

\begin{proof}
We will prove the lemma by induction on $n$. In addition to (\ref{dx0xn}), at each step we will also prove the inequality
\begin{equation}\label{x0xn-1}
(x_0| x_{n-1})_{x_n}> C + 8\delta ,
\end{equation}
which will be necessary to make the inductive step.

For $n=1$ both (\ref{dx0xn}) and (\ref{x0xn-1}) are obvious. Assume now that (\ref{dx0xn}) and (\ref{x0xn-1}) hold for some $n\ge 1$; we want to prove the lemma for $n+1$ points.

Note first that
\begin{equation}\label{gpxkx0}
(x_0|x_{n+1})_{x_{n}} \le C+8\delta.
\end{equation}
Indeed, otherwise using (\ref{x0xn-1}) and applying Lemma \ref{GHs23} we obtain $(x_{n-1}|x_{n+1})_{x_{n}} > C$, which contradicts our assumption. Using the definition of the Gromov product, (\ref{gpxkx0}), and the inductive assumption, we obtain
$$
\d_S(x_0, x_{n+1}) = \d_S(x_0, x_n) +\d_S(x_n, x_{n+1}) - 2(x_0|x_{n+1})_{x_n} \ge \sum\limits_{i=1}^{n+1} \d_S (x_{i-1}, x_{i}) - 2n(C+8\delta).
$$
We obviously have
\begin{equation}\label{dxy}
(x|y)_z + (x|z)_y = \d_S (y,z)\;\;\; \forall \, x,y,z \in S.
\end{equation}
Using (\ref{dxy}), (\ref{gpxkx0}), and (\ref{length}) we obtain
$$
(x_0| x_{n})_{x_{n+1}} = \d_S (x_{n+1}, x_n)-(x_0, x_{n+1})_{x_n} > 2C+16\delta  - (C+8\delta) = C+8\delta.
$$
This completes the inductive step.
\end{proof}

\begin{lem}\label{z}
Let $G$ be a group and let $[X]\in \H_{gt} (G)$. Then for every infinite sequence of elements $a_1, a_2, \ldots$ of $G$, there exists $z\in \mathcal L ([X])$ and an infinite subsequence $b_1, b_2, \ldots $ of $a_1, a_2, \ldots$ such that
\begin{equation}\label{supz}
\sup_{n,i\in \mathbb N} (z^n| b_i)_1<\infty,\;\;\; {\rm and}\;\;\; \sup_{n,i\in \mathbb N} (z^{-n}| b_i^{-1})_1<\infty ,
\end{equation}
 where $(x |y)_1$ denotes the Gromov product of $x,y\in G$ with respect to $1$ computed in $\Gamma (G,X)$.
\end{lem}

\begin{proof}
Since the action of $G$ on $\Gamma (G,X)$ is of general type, there exist $3$ independent loxodromic elements $f,g,h\in G$. Note that if $\sup_{n,i\in \mathbb N} (z^n| a_i)_1=\infty $ (respectively, $\sup_{n,i\in \mathbb N} (z^{-n}| a_i^{-1})_1=\infty $) for some loxodromic element $z\in G$ then, up to passing to an infinite subsequence of $a_1, a_2, \ldots $, we have $\lim_{i\to \infty}a_i= z^+$ (respectively, $\lim_{i\to \infty}a_i^{-1} = z^-$). Using this observation and the fact that the limit points $f^+, f^-, g^+, g^-, h^+, h^-\in \partial \Gamma (G,X)$ are pairwise distinct, it is easy to see that the claim of the lemma holds for at least one $z\in \{ f,g,h\}$.
\end{proof}

\begin{proof}[Proof of Theorem \ref{main4}]
Verifying that equivalent actions are coarsely isospectral is straightforward.

To prove the other implication, assume $A=G\curvearrowright S$ and $B=G\curvearrowright T$ are two coarsely isospectral cobounded non-quasi-parabolic actions on hyperbolic spaces. By Proposition \ref{prop-SM}, we can also assume that $S=\Gamma(G,X)$, $T=\Gamma(G,Y)$ for some generating sets $X$ and $Y$ of $G$. Obviously if one of the actions $A$, $B$ is elliptic, then so is the other. Thus it suffices to consider the case $[X],[Y]\in \H_\ell(G)\cup \H_{gt}(G)$.

Arguing by contradiction, assume that $X\not\sim Y$. Without loss of generality, there exists a sequence $a_1, a_2, \ldots \in X$ such that
\begin{equation}\label{limai}
\lim\limits_{i\to \infty} |a_i|_Y=\infty.
\end{equation}
We consider two cases.

{\bf Case 1.} Assume first that $[Y]$ is lineal, i.e., $T$ is quasi-isometric to a line. Then (\ref{limai}) implies that $(a_i)\to \xi \in \partial T$ (possibly after passing to a subsequence). Let ${\bf x}= (a_i)$, $s=1$, and let $q_{\bf x}$ and $p$ be the corresponding Busemann quasi-character and pseudocharacter, see Definition \ref{Bpc}. By Remark \ref{rempq}, there is a constant $K$ such that for every $g\in G$ we have $$|q_{\bf x} (g)| \le |\d_Y(g, a_i)-\d_Y(1,a_i)| +K\le |g|_Y+K.$$ Applying this inequality to $g=a_i^n$, we obtain
$$
|p(a_i)|=\lim_{n\to \infty}\frac{|q_{\bf x} (a_i^n)|}{n}\le \lim_{n\to \infty}\frac{|a_i^n|_Y+K}{n}=\tau_B(a_i).
$$
It follows immediately from the definition of $p$ that $|p(a_i)|\to \infty$ as $i\to \infty$ and, therefore, we have $\tau_B(a_i)\to \infty$ as $i\to \infty$. On the other hand, we obviously have $\tau_A(a_i)\le |a_i|_X=1$. This contradicts our assumption that the actions $G\curvearrowright S$ and $G\curvearrowright T$ are coarsely isospectral.

{\bf Case 2.} Now we suppose that $[Y]\in \H_{gt}(G)$. Since $A$ and $B$ are coarsely isospectral, they have the same set of loxodromic elements $\mathcal L([X])=\mathcal L([Y])$. In what follows, by a \emph{loxodromic element} of $G$ we always mean loxodromic with respect to the actions on $S$ and $T$.

By Lemma \ref{z}, there exists a loxodromic element $z\in G$ and an infinite subsequence $b_1, b_2, \ldots $ of $a_1, a_2, \ldots$ satisfying (\ref{supz}), where the Gromov products are computed in $\Gamma (G,Y)$. Let $C$ be a constant such that
\begin{equation}\label{supzC}
 (z^n| b_i)_1\le C,\;\;\; {\rm and}\;\;\; (z^{-n}| b_i^{-1})_1\le C
\end{equation}
for all $n,i\in \mathbb N$. Again passing to a subsequence of $b_1, b_2, \ldots $ if necessary, we can assume that
\begin{equation}\label{bi}
|b_i|_Y \ge 2C+16\delta,\;\;\; i\in \mathbb N,
\end{equation}
where $\delta $ is the hyperbolicity constant of $\Gamma (G, Y)$. Further since $z$ is loxodromic, there exists $k\in \mathbb N$ such that
\begin{equation}\label{zn}
|z^k|_Y \ge 2C+16\delta.
\end{equation}
We now fix $i\in \mathbb N$ and let $h_i=b_i^{-1} z^k$. Inequalities (\ref{supzC}), (\ref{bi}), (\ref{zn}), and the obvious equality $(gx\mid gy)_{g}=(x\mid y)_1$ for every $g\in G$ and $x,y\in \Gamma (G,Y)$ allow us to apply Lemma \ref{lemx0xn} to $\Gamma (G,Y)$ and the sequence of points
$$
x_0=1, \;\; x_1=b_i^{-1},\;\; x_2=h_i=b_i^{-1} z^k,\;\; x_3= b_i^{-1}z^kb_i^{-1}, \;\; x_4 = h_i^2=b_i^{-1}z^k b_i^{-1} z^k, \;\; \ldots
$$
We conclude that for every $m\in \mathbb N$,
$$
\begin{array}{rcl}
|h^m_i|_Y & = & \d_Y (x_0, x_{2m})\ge  \sum\limits_{j=1}^{2m}  \d_Y (x_{i-1}, x_{i}) - 2(2m-1)(C+8\delta) \\&&\\
&\ge & m(|z^k|_Y +|b_i|_Y) - 2(2m-1)(C+8\delta).
\end{array}
$$
Consequently,
$$
\tau_B(h_i)=\lim_{m\to \infty} \frac{|h^m_i|_Y}{m} \ge |z^k|_Y +|b_i|_Y - 4(C+8\delta) \to \infty
$$
as $i\to \infty$ by (\ref{limai}). On the other hand, we have
$$
\tau _A(h_i) \le |h_i|_X \le |z^k|_X +|b_i|_X\le |z^k|_X+1.
$$
Since $z$ and $k$ are fixed, the rightmost side of the latter inequality is independent of $i$ and thus the translation lengths $\tau _A(h_i)$ are uniformly bounded. This contradicts our assumption that the actions $A$ and $B$ are coarsely isospectral.
\end{proof}

\begin{proof}[Proof of Corollary \ref{cormain4}]
Let $G$ be a group and let $[X],[Y]\in \AHG$ be coarsely isospectral structures. Since acylindrically hyperbolic structures cannot be quasi-parabolic (see Theorem \ref{tricho}), Theorem \ref{main4} applies to the actions $G\curvearrowright \Gamma(G,X)$ and $G\curvearrowright \Gamma (G,Y)$.
\end{proof}

We conclude this section with two examples. The first one shows that coboundedness of the action cannot be dropped from Theorem \ref{main4}.

\begin{ex}\label{ex-non-cb}
We say that an action $A=G\curvearrowright S$ is {\it translationally proper} if for every $c>0$ there are only finitely many conjugacy classes of elements $g\in G$ with $\tau_{A}(g)<c$. For a translationally proper action $A=G\curvearrowright S$ and a sequence $(g_i)_{i\in \mathbb N} \subseteq \mathcal L_A(G)$, we have $\lim_{i\to\infty}\tau_{A}(g_i)=\infty $ if and only if the sequence $(g_i)_{i\in \mathbb N}$ contains infinitely many pairwise non-conjugate elements. It follows that every two translationally proper actions of $G$ on metric spaces are coarsely isospectral.

Let $G=F_n\rtimes _\phi \mathbb Z$, where $F_n$ is a free group of rank $n$ and the automorphism $\phi $ is atoroidal, i.e., no power of $\phi $ fixes a non-trivial conjugacy class of $F_n$. Then $G$ is a hyperbolic group (see \cite{BeFe} and \cite{Bri}). We denote by $X$ and $Y$ some finite generating sets of $G$ and $F_n\le G$, respectively, and consider the standard actions of $F_n$ on the Cayley graphs $\Gamma (G,X)$ and $\Gamma (F_n, Y)$. Then both actions are translationally proper, hence they are coarsely isospectral. However these actions are not equivalent. It follows, for example, from the well-known fact that every infinite normal subgroup of a hyperbolic group is at least exponentially distorted, see \cite{BH}.
\end{ex}

The next example shows that the theorem may fail for quasi-parabolic actions.

\begin{ex}\label{pqcis}
Let $[X]$, $[Y]$ be the quasi-parabolic structures on the Baumslag-Solitar group $G$ constructed in Example \ref{BS}. Let $A=G\curvearrowright \Gamma(G,X)$, $B=G\curvearrowright\Gamma(G,Y)$. Let $\e \colon  G\to \mathbb Z $ be the homomorphism sending $a$ to $0$ and $b$ to $1$. It is easy to see that $\lim_{i\to\infty} \tau _A(g_i)\to \infty $ for a sequence $(g_i)\subseteq G$ if and only if $\lim_{i\to \infty} |\e (g_i)|=\infty$. The same holds true for translation numbers with respect to $B$. Hence $A$ and $B$ are coarsely isospectral. However, we proved in Example \ref{BS} that they are not equivalent.
\end{ex}

In fact, it is not difficult to show that \emph{all} non-trivial hyperbolic structures (not necessarily quasi-parabolic) on the Baumslag-Solitar group are coarsely isospectral. The same is true for the wreath product $\mathbb Z\, {\rm wr}\,\mathbb Z$. We leave the proofs of these fact as an exercise to the reader.


\section{Hyperbolic accessibility}


\subsection{Examples of inaccessible groups}

Recall that a group $G$ is said to be \emph{$\H$-accessible} (respectively \emph{$\AH$-accessible}) if the poset $\H (G)$ (respectively $\AHG$) contains the largest element.

It is easy to find examples of groups which are not $\H $-accessible. For instance, a rich source of examples is provided by direct products.

\begin{ex}\label{H-inacc}
Suppose that $G=A\times B$, where both $A$ and $B$ have hyperbolic structures of general type (e.g., we can take $A=B=F_2$). Then $G$ is not
$\H$-accessible. Indeed, let $[X]$ be the largest element in $\H (G)$. If one of the subgroups $A$, $B$ acts on $\Gamma (G,X)$ with bounded orbits, $[X]$ cannot be largest since we assume that $\H_{gt}(A)$ and $\H_{gt}(B)$ are non-empty. Hence by Lemma \ref{product} $[X]\in \H_\ell(G)$. Since every lineal structure is minimal, see Corollary \ref{linmin}, $[X]$ must be unique hyperbolic structure on $G$. However this again contradicts our assumption as every general type action of $A$ or $B$ on a hyperbolic space gives rise to a general type action of $G$ on the same space.
\end{ex}

Note that groups described in Example \ref{H-inacc} above are $\AH$-accessible. Indeed by \cite[Corollary 7.3]{Osi16} an acylindrically hyperbolic group cannot split as a direct product of two infinite groups. Since these groups are also not virtually cyclic, they only have the trivial acylindrically hyperbolic structure by Theorem \ref{tricho}.

It is much harder to find examples of groups, especially finitely generated or finitely presented ones, which are not $\mathcal{AH}$--accessible. To this end, we first need to recall several definitions, restated in the terminology of the present paper.  An element $g$ of an acylindrically hyperbolic group $G$ is called \emph{generalized loxodromic} if there exists an acylindrically hyperbolic structure $A\in \AHG$ such that $g\in \L(A)$. It is an open question whether for any two generalized loxodromic elements $f,g\in G$, there exists an acylindrically hyperbolic structure with respect to which \emph{both} $f$ and $g$ are loxodromic. Moreover, one could ask whether there exists a \emph{single} $A\in \AHG$ so that $\L(A)$ contains \emph{all} generalized loxodromic elements of $G$. We call such a structure \emph{loxodromically universal}.  It is easy to see that if $\AHG$ contains the largest element, then it is necessarily loxodromically universal since $A\preccurlyeq B$ for some $A,B\in \AHG$ obviously implies $\L(A)\subseteq \L (B)$.

The first (rather obvious) example of a group that does not have a loxodromically universal acylindrically hyperbolic structure was provided in \cite{Osi16}; it was the free product of groups $\mathbb Z\times \mathbb Z_n$ for $n\in \mathbb N$. This group is infinitely generated. The first finitely generated example was given in \cite{A}, where the first author proved that Dunwoody's group has no loxodromically universal acylindrically hyperbolic structure. We briefly describe Dunwoody's group $J$ and refer the reader to \cite{Dun93} for further details.  Let $H$ be the subgroup of the group of permutations of $\mathbb Z$ generated by the transposition $(0\, 1)$ and the shift map $i\mapsto i+1$, which we denote by $s$.  Let $H_\omega$ be the group of finitely supported permutations of $\mathbb Z$.  For any $n$, we have the following decomposition of Dunwoody's group:
$$J=G_1*_{K_1}\cdots G_{n-1}*_{K_{n-1}} J_n,$$
where each $G_i$ is a finite group and, letting $Q_n=G_n*_{K_n} G_{n+1}*_{K_{n+1}}\cdots$, $$J_n=Q_n*_{H_\omega}H.$$
Thus we obtain the following.

\begin{ex}\label{Dun}
The group $\ast _{n\in \mathbb N} \mathbb Z\times \mathbb Z_n$ and Dunwoody's group are not $\AH $-accessible.
\end{ex}

Finally, we give a finitely presented example.

\begin{thm}\label{AH-inacc}
There exists a finitely presented group that is neither $\mathcal{AH}$-accessible nor $\H$-accessible.
\end{thm}

\begin{proof}
Let $N$ be a finitely presented non-hyperbolic normal subgroup of a hyperbolic group $G$ such that $G/N\cong \mathbb Z$. Such subgroups do exist, see \cite{Bra}. Let $t$ be an element of $G$ such that $tN$ generates $G/N$. Further let $a\in N$ be an element of infinite order such that $a\notin E(t)$. Then there exists $n\in \mathbb N$ such that $G$ is hyperbolic relative to any of the following $3$ collections of subgroups: $\{ E(t)\}$, $\{E(s)\}$, $\{E(s), E(t)\}$, where $s=ta^n$ (this follows for example from \cite[Corollary 1.7 and Lemma 7.4]{Osi06b}).

Let $X$ be a finite generating set of $G$ and $S=\Gamma (G, X\cup E(s))$, $T=\Gamma (G, X\cup E(t))$. Then the graphs $S$ and $T$ are hyperbolic and the actions of $N$ on both of them is cobounded since $G=N E(t)=N E(s)$; this action is also acylindrical and so is the action of $G$ \cite[Proposition 5.2]{Osi16}. Let $A=\sigma ([N\curvearrowright S])$ and  $B=\sigma ([N\curvearrowright T])$ be the corresponding acylindrically hyperbolic structures on $N$. It suffices to prove the following.

{\bf Claim.} If some $[Z]\in \mathcal G (N)$ satisfies $A\preccurlyeq [Z]$ and $B\preccurlyeq [Z]$, then  $Z$ is finite.

Indeed if the claim is proved, then the largest element in either $\H (N)$ or $\AH (N)$ must correspond to a finite generating set, i.e., the group $N$ must be hyperbolic, which contradicts our assumption.

Let us prove the claim now. Since $A\preccurlyeq [Z]$ (respectively $B\preccurlyeq [Z]$), $Z$ must be a bounded set considered as a set of vertices in $S$ (respectively $T$). Let $K$ be a constant such that $Z$ belongs to the balls of radius $K$ centered at $1$ in both $S$ and $T$. Given $z\in Z$,  let $p_z$  (respectively $q_z$) be a geodesic going from  $1$ to $z$ in $S$ (respectively $T$). We can naturally think of $S$ and $T$ as subgraphs of $R=\Gamma (G, X\cup E(s)\cup E(t))$. Then $c_z=p_zq_z^{-1}$ is a cycle of length at most $2K$ in $R$. Note that all components of $c_z$  are isolated. Indeed two components of $p_z$ (respectively $q_z$) cannot be connected since this path is a geodesic in $S$ (respectively $T$), and a component of $p_z$ cannot be connected to a component of $q_z$ since they correspond to different hyperbolically embedded subgroups.  By Lemmas \ref{C} and \ref{lem:Z}, there is a finite subset $\mathcal F\subseteq G$ such that every component of $c_z$ (and hence every component of $p_z$) is labeled by an element of $\mathcal F$. Since $X$ is finite, it follows that there are only finitely many possible labels of paths $p_z$, $z\in Z$, i.e., $|Z|<\infty$.
\end{proof}

\begin{rem}
The group $N$ described in the proof of Theorem \ref{AH-inacc} does have a loxodromically universal acylindrical action, for example the one on a locally finite Cayley graph of $G$.
\end{rem}

\subsection{$\AH$-accessible groups}

The goal of this section is to prove Theorem \ref{main5}. In fact, we will prove something stronger for each of the classes of groups mentioned in Theorem \ref{main5}. We would like to note that parts (b) and (c) of Theorem \ref{main5} were independently and subsequently proven in \cite{ABD}. Part (d) of Theorem \ref{main5} is also proven in \cite{ABD} in the special case when the 3-manifold has no Nil or Sol in its prime decomposition. \cite{ABD} additionally proves the $\AH$-accessibility of certain other groups using different methods.

\begin{defn} A group $G$ is said to be \textit{strongly $\AH$-accessible} if there exists an element $[X] \in \AHG$ such that for every acylindrical action (not necessarily cobounded) of $G$ on any hyperbolic space $S$, we have $G \acts S \preceq G \acts \Ga(G,X)$, with respect to the preorder on group actions from Definition \ref{def-poset}.  We say that the structure $[X]$ \textit{realizes} the strong $\AH$-accessibility of $G$. Such a structure, if it exists, is obviously unique. In particular, a strongly $\AH$-accessible group is $\AH$-accessible.
\end{defn}

\begin{rem} Note that in the above definition, we do not restrict the cardinality of $S$. We consider actions of $G$ on \textit{all} hyperbolic metric spaces, not just those of bounded cardinality. \end{rem}

\begin{ex}\label{hypacc} Every hyperbolic group is strongly $\AH$-accessible. Indeed, if $G$ is a hyperbolic group, then there exists a finite generating set $X$ such that $[X] \in \AHG$. The result then follows from Lemma \ref{lem-Lip}.
\end{ex}

\begin{ex}\label{ztwoacc} Every group $G$ which is not acylindrically hyperbolic is strongly $\AH$-accessible. Indeed, this is an immediate consequence of Theorem \ref{tricho}. If $G$ is virtually cyclic, the result follows from Example \ref{hypacc}. If every acylindrical action of $G$ on a hyperbolic space is elliptic, then the trivial structure realizes the strong $\AH$-accessibility of $G$.

In particular, this example applies to groups with infinite amenable radicals (e.g. infinite center) (see \cite[Corollary 7.2]{Osi16}) and to direct products of two infinite groups (see \cite[Corollary 7.2]{Osi16}). In fact, if $G$ is the direct product of two groups with infinite order elements, then the trivial structure realizes the strong $\AH$-accessibility of $G$ (since $G$ contains $\ztwo$ as a subgroup, it is not virtually cyclic).

\end{ex}

We first turn our attention to relatively hyperbolic groups. Recall that a group $G$ is \emph{hyperbolic relative to subgroups} $H_1, \ldots, H_n$  if $\Hi \hookrightarrow_h (G,X)$ for some finite $X \subset G$ (see Theorem \ref{rhhe}). The subgroups $H_1, \ldots, H_n$ are called \textit{peripheral subgroups} of $G$.

\begin{thm} \label{relhypmax} Let $G$ be a relatively hyperbolic group with peripheral subgroups $H_1, \dots, H_n$. If each $H_i$ is strongly $\AH$-accessible, then $G$ is strongly $\AH$-accessible. \end{thm}

\begin{proof} Let $X$ be a finite subset of $G$ such that $\Hi \hookrightarrow_h (G,X)$. Let $\mathcal{H} = \sqcup_{i =1}^{n} H_i$. By \cite[Proposition 5.2]{Osi16} (see also Example \ref{rhshe}), the action of $G$ on $\Ga(G, X \sqcup \mathcal{H})$ is acylindrical, i.e., $\Hi$ is strongly hyperbolically embedded in $G$ with respect to $X$.

Let $[Y_i]$ be the element in $\mathcal{AH}(H_i)$ that realizes the strong $\AH$-accessibility of $H_i$, for each $1 \leq i \leq n$. Then by Theorem \ref{genmain2}, $[X \cup Y_1 \cup... \cup Y_n] \in \AHG$. We will show that this element realizes the strong $\AH$-accessibility of $G$.

Indeed, suppose that $G \acts Z$ is an acylindrical action of $G$ on a hyperbolic space and let $d_Z$ denote the metric on $Z$.  Fix a base point $z \in Z$. Restricting the action of $G$ on $Z$ to each $H_i$, we obtain an acylindrical action of each $H_i$ on a hyperbolic space.  But then $$H_i \acts (Z, d_Z)  \preceq H_i \acts \Ga(H_i, Y_i)$$ for every $1 \leq i \leq n$. We can assume that there exists a constant C such that for every $i=1,...,n$, $d_Z(z, hz) \leq Cd_{Y_i} (1,h) + C$ for all $h \in H_i$.
In particular, if $y \in Y_i$, then $d_Z(z, yz) \leq 2C.$

Since $X$ is finite, Lemma \ref{lem-Lip} applies and we conclude that $$G \acts Z \preceq G \acts \Ga(G, X \cup Y_1 \cup... \cup Y_n).$$ \end{proof}

\begin{proof}[Proof of Theorem \ref{main5}(a)]
If G is a finitely generated relatively hyperbolic group, then it follows from \cite[ Theorem 1.1]{Osi06}  that the collection of peripheral subgroups is finite. Let $H$ be a peripheral subgroup of $G$, which by assumption, is not acylindrically hyperbolic. Example \ref{ztwoacc}  applies, and we conclude that each peripheral subgroup is strongly $\AH$-accessible. The result  follows from Theorem \ref{relhypmax}. \end{proof}

We next deal with the case of mapping class groups of closed punctured surfaces, for which we will need several facts and definitions taken from \cite{farb}, stated below. We refer the reader to \cite{farb} for proofs and details.

\begin{defn}[Complex of curves] A closed curve on $S$ is called \textit{essential} if it is not homotopic to a point or a puncture.  The complex of curves associated to $S$ is a graph defined as follows$\colon$ vertices of the complex of curves are isotopy classes of essential, simple closed curves, and two vertices are joined by an edge if the curves have disjoint representatives on $S$. The complex of curves is denoted by $C(S)$.
\end{defn}

We let $g$ denote the genus of the surface $S$ and $p$ denote the number of punctures. We adopt the convention that if $a$ is a vertex of $C(S)$, then by slight abuse of notation, we let $a$ denote the associated curve, and let $T_a$ be the Dehn twist about $a$. We now list some facts about Dehn twists and $C(S)$ which we will require for the proof.

\begin{itemize}
\item[(a)]\cite[Propositions 3.1 and 3.2]{farb} $T_a$ is a non-trivial infinite order element of $G$.
\item[(b)]\cite[Fact 3.6]{farb} $T_a = T_b$ if and only if $a = b$.
\item[(c)]\cite[Fact 3.8]{farb} For any $f \in \operatorname{Mod}(S)$ and any isotopy class $a$ of simple closed
curves in $S$, we have that $f$ commutes with $T_a$ if and only if $f(a) = a$.
\item[(d)]\cite[Fact 3.9]{farb} For any two isotopy classes $a$ and $b$ of simple closed curves in a
surface $S$, we have that $a$ and $b$ are connected by an edge in $C(S)$ if and only if $T_aT_b = T_bT_a$.
\item[(e)]\cite[Powers of Dehn twists, pg 75]{farb} For non-trivial Dehn twists $T_a$ and $T_b$, and non-zero integers $j,k$, we have $T^j_a = T^k _b$ if and only if $a= b ; j = k$.
\item[(f)]\cite[Sec 1.3]{farb}  $G$ admits a cocompact, isometric action on $C(S)$. (This follows from the change of co-ordinates principle).
\item[(g)]\cite[Theorem 1.1]{MaMin} $C(S)$ is a hyperbolic space. Except when $S$ is a sphere with 3 or
fewer punctures, C(S) has infinite diameter.
\item[(h)]\cite[Theorem 1.3]{bow} If $S$ is a surface satisfying $3g + p \geq 5$, then the action of $\operatorname{Mod}(S)$ on $C(S)$ is acylindrical.
\end{itemize}

\begin{proof}[Proof of Theorem \ref{main5}(b)]  Let $S$ be a compact, punctured surface without boundary, of genus $g$ and with $p$ punctures. We consider the following two cases.

{\bf Case 1.} First assume that $3g + p < 5$. The mapping class groups for the cases $g=0$ and $p=0,1,2,3$ are finite and hence $\AH$-accessible. In the cases of $g =0, p=4$ (the four-punctured sphere) and $g=1, p=0,1$ (the torus and the once punctured torus), the mapping class groups are hyperbolic groups. Example \ref{hypacc} applies and we conclude that these groups are also $\AH$-accessible.

{\bf Case 2.} We now assume that $3g + p \geq 5$. In this case, we will prove the result by using Corollary \ref{maxact-cor} applied to $\AHG \subset \GG$ for $G= \operatorname{Mod}(S)$.

By fact (f) above, $G$ admits a cocompact (hence cobounded), isometric action on $C(S)$. By facts (g) and (h) above, $C(S)$ is an infinite diameter hyperbolic space and the action of $\operatorname{Mod}(S)$ on $C(S)$ is acylindrical.

To apply Corollary \ref{maxact-cor}, we must consider stabilizers of vertices of $C(S)$. Let $H = Stab_G(a)$, where $a$ is a vertex of $C(S)$. By fact (a) above, $T_a$ is a non-trivial infinite order element of $G$. Further $T_a(a) =a$, so $T_a \in H$ by fact (c) above. For every element $f \in H$, using fact (c) again, we must have $fT_a = T_af$  since $f(a) =a$. This implies that $H$ has an infinite center and is thus not acylindrically hyperbolic (see Example \ref{ztwoacc}).

Since $C(S)$ is connected and unbounded, there exists a vertex $b \ne a$ of $C(S)$ connected by an edge to $a$. Then $T_b$ is a non-trivial infinite order element by fact (a), and by applying facts (d)and (c) above, $T_b \in H$. By using fact (e) above, one can easily show that $\langle T_a, T_b \rangle \cong \mathbb{Z}^2 \leq H$, so $H$ is not virtually cyclic. By Theorem \ref{tricho}, every acylindrical action of $H$ on a hyperbolic space is elliptic. In particular, for every acylindrical action of $G$ on a hyperbolic space, the induced action of $H$ is also acylindrical and thus elliptic. Applying Corollary \ref{maxact-cor}, it follows that $G$ is $\AH$-accessible with largest element $\sigma ([G \acts C(S)])$. \end{proof}

\begin{rem} The above proof also applies to the set $\mathcal{A}$ of acylindrical actions of $\op{Mod}(S)$ on hyperbolic spaces. In this case, Proposition \ref{maxact} allows us to conclude that $\op{Mod}(S)$ is strongly $\AH$-accessible. The same holds true for the cases of RAAGs and 3-manifolds discussed below.

\end{rem}

We now proceed to the case of right-angled Artin groups (RAAGs). We begin by defining a RAAG and its extension graph.

\begin{defn} Given a finite graph $\Ga$, the \textit{right-angled Artin group} on $\Ga$ is the group defined by the presentation $$A(\Ga) = \langle V(\Ga) \mid [a, b] = 1 \hspace{10pt} \forall \{a, b\} \in E(\Ga) \rangle.$$  For example, the RAAG corresponding to a complete graph on $n$ vertices is $\mathbb{Z}^n$. \end{defn}

\begin{defn} The \textit{extension graph} $\Ga ^e$ corresponding to $A(\Ga)$, introduced in \cite{KK}, is a graph with vertex set $$\{ v^g \mid v \in V(\Ga), g \in A(\Ga) \}$$ and edges defined by the following rule: two distinct vertices $u^g$ and $v^h$ are joined by an edge if and only if they commute in $A(\Ga)$.

There is a natural right-conjugation action of $A(\Ga)$ on $\Ga^e$ given by $gv^h = v^{hg}$ for $v \in V(\Ga)$ and $g,h \in A(\Ga)$. Further, we may write $$\Ga^e  = \bigcup_{g \in A(\Ga)} g \Ga ,$$ where $g \Ga$ denotes the graph $\Ga$ with its vertices replaced by the corresponding conjugates by $g$.
\end{defn}

We will require the following theorems for the proof. For the proofs of these theorems and further details concerning RAAGs, we refer the reader to \cite{KK}.

\begin{thm}\label{raagshyp} \cite[Lemma 26]{KK2} Let $\Ga$ be a finite connected graph. Then $\Ga^e$ is a quasi-tree.
\end{thm}

\begin{thm}\label{raagsacyl}\cite[Theorem 30]{KK} The action of $A(\Ga)$ on $\Ga^e$ is acylindrical.
\end{thm}

We first prove the strong $\AH$-accessibility of RAAGs arising from finite connected graphs. We will then use this result to prove the $\AH$-accessibility of RAAGs arising from any finite graph.

\begin{lem}\label{conngrs} Let $\Gamma$ be a connected finite graph and $G = A(\Gamma)$. Then $G$ is strongly $\AH$-accessible.
\end{lem}
\begin{proof}

Let $V(\Ga)$ denote the set of vertices of the graph $\Ga$. If $|V(\Ga)| =1$, then  $G \cong \mathbb{Z}$. Example \ref{hypacc} applies in this case and we conclude that $G$ is strongly $\AH$-accessible.

Thus we may assume that $|V(\Ga)| \geq 2$. In this case, we will prove the result by using Proposition \ref{maxact} applied to the set $\mathcal{A}$ of acylindrical actions of $G$ on hyperbolic spaces. Observe that $G \acts \Ga^e$ is cocompact and isometric. By Theorem \ref{raagshyp} and \ref{raagsacyl}, $G \act \Gamma^e$ is acylindrical and $\Gamma^e$ is a quasi-tree and hence a hyperbolic space. Thus $\sigma([G \acts \Ga^e]) \in \AHG$.

Since the action of $G$ on $\Ga^e$ is by conjugation, stabilizers of vertices of the extension graph correspond to centralizers of conjugates of standard generators of $G$. So we must consider $H = C_G(a^g)$, where $a$ represents a vertex of $\Gamma$, and $g$ is any element of $G$.

Since $\Gamma$ is connected and $|V(\Gamma)| \geq 2$, there exists a vertex $b \neq a$ such that $b$ is connected to $a$ in $\Ga^e$, i.e. $[a, b] =1$ in $G$. But then $[a^g, b^g] =1$, so $b^g \in H$. It can be easiliy shown that $\langle a^g, b^g \rangle \cong \langle a,b\rangle \cong \mathbb{Z}^2 \leq H$, since the RAAG corresponding to a graph with 2 vertices and an edge connecting them is $\mathbb{Z}^2$. Thus $H$ cannot be virtually cyclic.

Since the center of $H$ contains the infinite cyclic group $\langle a^g \rangle$, $H$ cannot be acylindrically hyperbolic by Example \ref{ztwoacc}. Thus $H$ cannot act non-elementarily and acylindrically on a hyperbolic space. By Theorem \ref{tricho}, for any acylindrical action of $G$ on a hyperbolic space, the induced action of $H$ is elliptic. Applying Proposition \ref{maxact}, we conclude that $G$ is strongly $\AH$-accessible.
\end{proof}

\begin{proof}[Proof of Theorem \ref{main5}(c)] If $\Ga$ is connected, the result follows from Lemma \ref{conngrs}. If $\Ga$ is a disconnected finite graph, then $\Ga$ has two or more connected components, say $\Ga_1, \Ga_2,..., \Ga_n$. Let $A(\Ga_i)$ denote the RAAG associated to the connected subgraph $\Ga_i$ of $\Ga$. It is easy to see that $G = A(\Ga_1) * A(\Ga_2) *...*A(\Ga_n)$, and so $G$ is hyperbolic relative to the collection $\{A(\Ga_i) \mid 1 \leq i \leq n\}$. By Lemma \ref{conngrs}, each RAAG $A(\Ga_i)$ is strongly $\AH$-accessible. Using Theorem \ref{relhypmax}, we conclude that $G$ is $\AH$-accessible.
\end{proof}

Lastly, we consider the case of fundamental groups of  compact, orientable 3-manifolds with empty or toroidal boundary. In order to prove the theorem, we will need the following results and definitions.

\begin{defn} A 3-manifold $N$ is said to be \textit{irreducible} if every embedded $S^2$ bounds a 3-ball. \end{defn}

\begin{defn} A 3-manifold $N$ is said to be \textit{atoroidal} if any map $T \rightarrow N$ which induces a monomorphism of fundamental groups can be homotoped into the boundary of $N$, i.e., $N$ contains no essential tori.
\end{defn}

\begin{defn} A \textit{Seifert fibered} manifold is a 3-manifold $N$ together with a decomposition into disjoint simple closed curves (called \textit{Seifert fibers}) such that each fiber has a tubular neighborhood that forms a standard fibered torus.

The \textit{standard fibered torus} corresponding to a pair of coprime integers $(a, b)$ with $a > 0$ is the surface bundle of the automorphism of a disk given by rotation by an angle of $\frac{2 \pi b}{a}$, equipped with natural fibering by circles.
 \end{defn}

\begin{lem}\label{normalsubgrp} \cite[Lemma 1.5.1]{AFW} Let $N$ be a Seifert fibered manifold. If $\pi_1(N)$ is infinite, then it contains a normal, infinite cyclic subgroup.

\end{lem}

\begin{defn} A compact 3-manifold is said to be \textit{hyperbolic} if its interior admits a complete metric of constant negative curvature $-1$.  \end{defn}

The following theorem was first announced by Waldhausen (\cite{wan}), and was proved independently by Jaco-Shalen (\cite{JS}) and Johannson (\cite{jon}).

\begin{thm}\label{jsj} \cite[Theorem 1.6.1]{AFW}[JSJ decomposition Theorem] Let $N$ be a compact, orientable, irreducible 3-manifold with empty or toroidal boundary. Then there exists a collection of disjointly embedded incompressible tori $T_1, T_2,..., T_k$ such that each component of $N$ cut along $T_1 \cup T_2 \cup...\cup T_k$ is atoroidal or Seifert fibered. Furthermore, any such collection of tori with a minimal number of components is unique up to isotopy.   \end{thm}

The tori in the above theorem are referred to as \textit{JSJ-tori}. If $T = \cup_{i =1}^{k} T_k$, the connected components of $N \backslash T$ are called \textit{JSJ-components}. For details of atoroidal and Seifert fibered manifolds, we refer the reader to \cite[Sections 1.5 and 1.6]{AFW}.

The following was proven by Perelman in his seminal papers (See \cite{per02, per03a, per03b}).

\begin{thm}\label{hypthm} \cite[Theorem 1.7.5]{AFW}[Hyperbolization Theorem] Let $N$ be a compact, orientable, irreducible 3-manifold with empty or toroidal boundary. Suppose that $N$ is not homeomorphic to $S^1 \times D^2$ (solid torus), $T^2 \times I$ (torus bundle), $K^2 \stackrel{\sim}{\times} I$ (twisted klein bottle bundle). If $N$ is atoroidal and $\pi_1(N)$ is infinite, then $N$ is a hyperbolic manifold.
\end{thm}

\begin{rem}\label{atorexp} Note that the manifolds $S^1 \times D^2$, $T^2 \times I$, or $K^2 \tilde{\times} I$, although atoroidal, are also Seifert fibered manifolds, and are hence considered to be Seifert fibered JSJ components. Under this convention, the Hyperbolization theorem implies that JSJ components of $N$ are either hyperbolic or Seifert fibered manifolds.
\end{rem}

\begin{defn} Let $N$ be a compact, orientable, irreducible 3-manifold with empty or toroidal boundary. We say $N$ is a \textit{graph manifold} if all its JSJ components are Seifert fibered manifolds.
\end{defn}

The next result can be found in \cite[Theorem 7.2.2]{AFW}. This result follows easily from a combination theorem proved by Dahmani (see \cite[Theorem 0.1]{Dah}) or a more general combination theorem, later proved by the third author (see \cite[Corollary 1.5]{Osidehn}). The result has also been proved by Bigdely and Wise (see \cite[Corollary E]{BigWise}).

\begin{thm}\label{manifoldshyprel} Let $N$ be a compact, orientable, irreducible 3-manifold with empty or toroidal boundary. Let $M_1,\ldots, M_k$ be the maximal graph manifold pieces of the JSJ-decomposition of $N$. Let $S_1,\ldots, S_l$ be the tori in the boundary of $N$ that adjoin a hyperbolic piece and let $T_1,\ldots, T_m$ be the tori in the JSJ-decomposition of $N$ that
separate two (not necessarily distinct) hyperbolic components of the JSJ-decomposition. The fundamental group of $N$ is hyperbolic relative to the set of peripheral subgroups $$\{H_i\} = \{\pi_1(M_p)\} \cup \{\pi_1(S_q)\} \cup \{\pi_1(T_r)\}.$$
\end{thm}

The last theorem we mention here is a combination of \cite[Lemma 2.4]{WZ} and \cite[Lemma 5.2]{MO}. Although \cite[Lemma 2.4]{WZ} was originally stated and proved for closed manifolds, the same proof also holds for manifolds with toroidal boundary. (See proofs of \cite[Lemma 2.3 and Lemma 2.4]{WZ}).

\begin{thm}\label{manifolddicho} Let $N$ be an orientable, irreducible 3-manifold with empty or toroidal boundary. Then either $N$ has a finite-sheeted covering space that is a torus bundle over a circle or the action of $\pi_1(N)$ on the Bass-Serre tree associated to the JSJ decomposition of $\pi_1(N)$ is acylindrical.
\end{thm}

\begin{proof}[Proof of Theorem \ref{main5}(d)]
We first observe that it suffices to prove the theorem for a compact, orientable, irreducible 3-manifold $N$ with empty or toroidal boundary. Indeed, if $N$ is not irreducible, we let $\widehat{N}$ denote the 3-manifold obtained from $N$ by gluing 3-balls to all spherical components of $\partial N$. Then $\widehat{N}$ is irreducible, and $\pi_1(\widehat{N}) = \pi_1(N)$. Also observe that if $\pi_1(N)$ is finite, then it is $\AH$-accessible by Example \ref{hypacc}, so we may assume that $\pi_1(N)$ is infinite in what follows. We consider the following two cases.

{\bf Case 1.} If there are no JSJ-tori, then it follows from Theorem \ref{jsj} that $N$ is either an atoroidal manifold or is Seifert fibered. If $N$ is atoroidal and not homeomorphic to $S^1 \times D^2$, $T^2 \times I$, or $K^2 \tilde{\times} I$, then it follows from Theorem \ref{hypthm} that $N$ is hyperbolic. Consequently, if $N$ is closed, $\pi_1(N)$ is a hyperbolic group and hence $\AH$-accessible by Example \ref{hypacc}. If $N$ has toroidal boundary, then $\pi_1(N)$ is hyperbolic relative to its peripheral subgroups, which are isomorphic to $\ztwo$ (see \cite{farbrel}). Applying Example \ref{ztwoacc}, we get that $\ztwo$ is stongly $\AH$-accessible. By Theorem \ref{relhypmax}, we conclude that $\pi_1(N)$ is $\AH$-accessible.

If $N$ is Seifert fibered (recall that $S^1 \times D^2$, $T^2 \times I$, or $K^2 \tilde{\times} I$ are considered Seifert fibered JSJ components, as explained in Remark \ref{atorexp}), then by Lemma \ref{normalsubgrp}, $\pi_1(N)$ has an infinite cyclic, normal subgroup. Since $\mathbb{Z}$ is not acylindrically hyperbolic, we can use \cite[Corollary 1.5]{Osi16} to conclude that $\pi_1(N)$ is not acylindrically hyperbolic. Applying Theorem \ref{tricho}, we conclude that either $\pi_1(N)$ is virtually cyclic, or every acylindrical action of $\pi_1(N)$ on a hyperbolic space is elliptic. In the former situation, $\pi_1(N)$ is $\AH$-accessible by Example \ref{hypacc}. In the latter case, $\pi_1(N)$ is obviously $\AH$-accessible.

{\bf Case 2.} We now assume that $N$ admits at least one JSJ torus, i.e., the JSJ decomposition of $N$ is non-trivial. By the Seifert-van Kampen theorem, the JSJ decomposition of $N$ induces a graph of groups decomposition of $\pi_1(N)$ whose vertex groups are the fundamental groups of the JSJ components, and the edge groups are the fundamental groups of the JSJ tori.

By Theorem \ref{manifoldshyprel}, it suffices to prove the strong $\AH$-accessibility of each peripheral subgroup $H_i$ provided by the theorem. Following the notation of Theorem \ref{manifoldshyprel}, if $H_i = \pi_1(S_q)$ or $H_i = \pi_1(T_r)$, then $H_i \simeq \ztwo$. By Example \ref{ztwoacc}, such $H_i$ are strongly $\AH$-accessible.

It thus remains to consider the graph manifolds $M_p$, which have at least one JSJ component. Using Theorem \ref{manifolddicho}, either $M_p$ has a finite-sheeted covering space that is a torus bundle over a circle or the action of $\pi_1(M_p)$ on the Bass-Serre tree associated to the JSJ decomposition of $\pi_1(M_p)$ is acylindrical. We denote this Bass-Serre tree by $\mathcal{T}_p$.

If $M_p$ has a finite-sheeted covering space that is a torus bundle over a circle, then $\pi_1(M_p)$ is virtually polycyclic and is hence not acylindrically hyperbolic. Further, since we have at least one JSJ torus, $\ztwo \leq \pi_1(M_p)$, which means that $\pi_1(M_p)$ is not virtually cyclic.  By Theorem \ref{tricho}, every acylindrical action of $\pi_1(M_p)$ on a hyperbolic space is elliptic, allowing us to conclude that the trivial structure realizes the strong $\AH$-accessibility of $\pi_1(M_p)$.

If the action of $\pi_1(M_p)$ on the Bass-Serre tree $\mathcal{T}_p$ is acylindrical, then we will use Proposition \ref{maxact} applied to the set $\mathcal{A}$ of acylindrical actions of $\pi_1(M_p)$ on hyperbolic spaces in order to prove that $\pi_1(M_p)$ is strongly $\AH$-accessible. Note that the action $\pi_1(M_p)\acts \mathcal{T}_p$ is cocompact and so $\sigma([\pi_1(M_p) \acts \mathcal{T}_p]) \in \mathcal{AH}(\pi_1(M_p))$. Stabilizers of vertices for this action are isomorphic to the vertex groups, which are the fundamental groups of Seifert fibered components. Let $M$ be a Seifert fibered component of $M_p$. Since we have at least one JSJ torus, $\ztwo \leq \pi_1(M)$ and $\pi_1(M)$ is infinite. Arguing as in Case 1 by using Lemma \ref{normalsubgrp}, we can conclude that $\pi_1(M)$ is not acylindrically hyperbolic.  Applying Theorem \ref{tricho} allows us to conclude that the induced action of $\pi_1(M)$ in any acylindrical action of $\pi_1(M_p)$ on hyperbolic spaces is  elliptic. Applying Proposition \ref{maxact}, we get that $\pi_1(M_p)$ is strongly $\AH$-accessible.
\end{proof}


\section{Open problems}\label{sec:OP}


The goal of this section is to discuss several open problems motivated by our work.

We begin with two problems about the preorder on group actions, see Definition \ref{def-poset}. The preorder $\preceq$ induces an order on the set of weak equivalence classes of $G$-actions on metric spaces (of cardinality at most $c$). The resulting poset is a lattice, denoted $\mathcal L(G)$, whose least element is the weak equivalence class consisting of $G$-actions with bounded orbits (see Example \ref{bo-ex}) and a maximal element exists if and only if $G$ is finitely generated, in which case the weak equivalence class consisting of geometric actions is the largest element. The proofs of these results are not difficult and we leave them as exercises for the reader.

\begin{prob}
Does there exist any interesting connection between algebraic or geometric properties of $G$ and algebraic properties of the lattice $\LG$?
\end{prob}

The cardinality of $\LG$ is either finite or at least $2^{\aleph_0}$.  For countable groups, finiteness of $\LG$ is equivalent to $G$ being finite, while for uncountable groups this is not true: the lattice corresponding to $G=Sym(\mathbb N)$ consists of a single element since all actions of this group on metric spaces have bounded orbits, see \cite{Cor}. In particular, the cardinality of $\LG$ is a trivial example of a quasi-isometry invariant of finitely generated groups. We can ask the following.

\begin{prob}
Which algebraic properties of $\LG$ are quasi-isometry invariants for finitely generated groups?
\end{prob}

In contrast, $\GG$ may not be a lattice, but is always a meet-semilattice. Indeed it is easy to see that for every $[A],[B]\in \GG$,  the meet operation is defined by
\begin{equation}\label{meet}
[A]\wedge [B] =[A\cup B].
\end{equation}
On the other hand, let
$$
G=\bigoplus_{n=2}^\infty \mathbb Z_{n^2-1}.
$$
Let $A=\{ f_i\mid i=2,3, \ldots\}$ and $B=\{ g_i\mid i=2,3, \ldots\}$,
where $f_i$ and $g_i$ are elements of $G$ defined by
$$
f_i(n)=\left\{ \begin{array}{ll}
1, & {\rm if}\; n=i,
\\ &\\
0, & {\rm otherwise}
\end{array}\right.
$$
and
$$
g_i(n)=\left\{ \begin{array}{ll}
i, & {\rm if}\; n=i,
\\ &\\
0, & {\rm otherwise}.
\end{array}\right.
$$
It is straightforward to verify that both $A$ and $B$ generate $G$ and there is no upper bound  for the subset $\{[A], [B]\}$ in $\GG$. It is also not difficult to construct a finitely generated group $G$ for which $\GG$ is not a lattice. We leave details to the reader.

\begin{prob}
When is $\GG$ a lattice? Or, more generally, is there any interesting connection between algebraic and geometric properties of $G$ and properties of the poset $\GG$?
\end{prob}

We now turn to problems about hyperbolic structures of groups. Our understanding of posets of quasi-parabolic and general type structures on groups is far from being complete. The ultimate goal would be to obtain a classification of possible isomorphism types of $\H_{qp}(G)$ and $\H_{gt}(G)$ similar to Theorem \ref{lineal} for lineal structures. Achieving this goal does not seem realistic, but there are many particular open questions about $\H_{qp}(G)$ and $\H_{gt}(G)$ which seem more approachable. In particular, we can ask the following.

\begin{prob}
 Does there exist a finitely generated group $G$ such that $\H_{qp} (G)$ contains an uncountable chain?
\end{prob}

Recall that there are uncountable antichains in $\H _{qp}(\mathbb Z\, {\rm wr}\, \mathbb Z)$, see Example \ref{ZwrZ}. Since $\mathcal D$ contains countable chains, so does $\H _{qp}(\mathbb Z\, {\rm wr}\, \mathbb Z)$. However we do not know if $\H _{qp}(\mathbb Z\, {\rm wr}\, \mathbb Z)$ contains uncountable chains.

The next question is motivated by our results showing that the cardinality of $\H^+_\ell (G)$ is always $0$, $1$, or at least continuum, while $\H_\ell (G)$ and $\H_{gt}(G)$ can have arbitrary cardinality.

\begin{prob}\footnote{Partially solved in \cite{Bal18}.}
What values can the cardinality of $\H_{qp} (G)$ take? In particular, does there exist a group $G$ such that $\H_{qp}(G)$ is non-empty and finite?
\end{prob}

It would be also nice to characterize general type actions for which the strategy of constructing smaller general type actions described after Theorem \ref{n-gt-intr} actually works. For example, it should work for actions considered in \cite{BFu}.

\begin{conj}
Suppose that a group $G$ acts on a hyperbolic space and there exist two loxodromic elements of $G$ which are non-equivalent in the sense of \cite{BFu}. Prove that $\H (F_2)$ embeds in $\H (G)$ in this case.
\end{conj}

It follows from Theorem \ref{main2} that $\H (F_m)$ and $\H (F_n)$ embed in each other for all $m,n \ge 2$. The same result is true for the posets of acylindrically hyperbolic structures. However the following basic questions remain open.

\begin{prob}
\begin{enumerate}
\item[(a)] Is $\H (F_2)\cong \H (F_3)$?
\item[(b)] Is $\AH (F_2)\cong \AH (F_3)$?
\end{enumerate}
\end{prob}

Theorems \ref{main1} and \ref{inacc} imply that there are at least $4$ isomorphism types of posets $\AH (G)$: the posets of cardinality $1$ and $2$ and two posets of infinite cardinality, one with the largest element and one without it. However we do not know if the number of possible isomorphism types of $\AH (G)$ is infinite for countable groups. Note that the number of possible isomorphism types of $\H(G)$ is infinite by Theorem \ref{n-gt-intr}.

\begin{prob}
How many isomorphism types of posets $\AH(G)$ are there for countable groups $G$? What about finitely generated groups?
\end{prob}

\begin{prob}
Does there exist a group $G$ such that $\AH(G)$ contains a maximal element, but no largest element? How many maximal elements can $\AH(G)$ have?
\end{prob}

In particular, it would be interesting to understand if the acylindrically hyperbolic structures $A$ and $B$ on the group $N$ considered in the proof of Theorem \ref{AH-inacc} are maximal.

\begin{prob}
Are $\H$-accessibility and $\AH$-accessibility invariant under passing to subgroups of finite index and finite extensions?
\end{prob}

It is worth noticing that the structure of $\H (G)$ can be very different for a group $G$ and its subgroup of finite index; even the cardinality on $\H (G)$ can be different. Indeed the group $G=D_\infty ^2$ has exactly $3$ hyperbolic structures while the group $\mathbb Z^2$, which is a subgroup of index $4$ in $G$, has uncountably many, see the proof of Theorem \ref{lineal} and Example \ref{ex-Zn}. In contrast, Theorem \ref{main1} implies that the cardinality of $\AH (G)$ is invariant under passing to finite index subgroups and finite extensions.

\end{document}